\documentclass[12pt,reqno,twoside]{amsart}%

\usepackage{amsmath,amsthm,amscd,amsthm,upref,indentfirst}
\usepackage{amsfonts,mathrsfs}

\usepackage{amssymb,amsbsy,bm}
\usepackage{graphicx}
\usepackage{times}
\usepackage{amsaddr}
\usepackage{soul}
\usepackage{slashbox}

\usepackage[all]{xy}
\xyoption{v2}
\xyoption{knot} 
\xyoption{arc}
\xyoption{2cell}
\xyoption{rotate}
\xyoption{arc}
\xyoption{rotate}
\UseAllTwocells 
\xyoption{tips}
\SelectTips {xy} {12}

\usepackage{mathabx}

\usepackage[usenames]{color}

\theoremstyle{plain}
\newtheorem{theorem}{Theorem}
\newtheorem{assertion}[theorem]{Assertion}

\newtheorem{proposition}[theorem]{Proposition}

\theoremstyle{definition}
\newtheorem{definition}[theorem]{Definition}

\newtheorem{conjecture}[theorem]{Conjecture}
\newtheorem{corollary}[theorem]{Corollary}

\theoremstyle{remark}
\newtheorem{remark}[theorem]{Remark}

\newtheorem{example}[theorem]{Example}
\numberwithin{equation}{section}
\numberwithin{theorem}{section}

\usepackage[scaled=0.86]{helvet}
\renewcommand{\mathfrak}[1]{{\textbf{\upshape #1}}}
\renewcommand{\mathbf}{\bm}
\renewcommand{\mathrm}[1]{\scalebox{1.15}{\textsf{\upshape #1}}}
\renewcommand{\emph}[1]{\textrm{{\upshape #1}}}
\usepackage[scr=euler,scrscaled=1.15,bb=txof,bbscaled=1.16,cal=boondox,calscaled=1.15]{mathalfa}
\renewcommand{\mathit}[1]{\mathscr #1}
\renewcommand{\mathtt}[1]{\scalebox{1}{\bfseries \texttt{\upshape #1}}}
\usepackage{exscale}

\numberwithin{equation}{section}
\numberwithin{theorem}{section}

\usepackage[normalem]{ulem}
\usepackage[square,comma]{natbib}

\usepackage{mparhack}

\usepackage{keyval}

\usepackage[verbose]{backref}
\backrefsetup{verbose=false}
\renewcommand*{\backref}[1]{}
\renewcommand*{\backrefalt}[4]{[{\tiny%
    \ifcase #1 Not cited%
          \or Cited on page~\textcolor{BrickRed}{#2}%
          \else Cited on pages \textcolor{BrickRed}{#2}%
    \fi%
    }]}

\usepackage{fancyhdr}
\author{\small\scshape S\lowercase{teven} D\lowercase{uplij}}

\address{%
Center for Information Technology (WWU IT),
Universit\"at M\"unster,
R\"ontgenstrasse 7-13\\
D-48149 M\"unster,
Deutschland}
\email{\small \sf douplii@uni-muenster.de; http://ivv5hpp.uni-muenster.de/u/douplii}
\title[\scshape G\lowercase{raded medial ${n}$-ary algebras and polyadic tensor categories}
]{\large\bfseries\scshape G\lowercase{{raded medial} $\lowercase{n}$-ary algebras}
\lowercase{and polyadic tensor categories}}

\date{\textit{of start} July 17, 2019. \textit{Date}: \textit{of completion} {January 11, 2020}.
\newline
\mbox{}\hskip 1.16em
\textit{Total}: 107 references, 20 diagrams.
}

\renewcommand{\refname}{\textsc{References}}

\let\origsection\section
\renewcommand{\section}[1]{\sectionmark{#1}\origsection{#1}}
\let\origsubsection\subsection
\renewcommand{\subsection}[1]{\subsectionmark{#1}\origsubsection{#1}}

\makeatletter
\renewenvironment{thebibliography}[1]{%
  \@xp\origsection\@xp*\@xp{\refname}%
  \normalfont\footnotesize\labelsep .9em\relax
  \renewcommand\theenumiv{\arabic{enumiv}}\let\p@enumiv\@empty
  \vspace*{-5pt}
  \list{\@biblabel{\theenumiv}}{\settowidth\labelwidth{\@biblabel{#1}}%
    \leftmargin\labelwidth \advance\leftmargin\labelsep
    \usecounter{enumiv}}%
  \sloppy \clubpenalty\@M \widowpenalty\clubpenalty
  \sfcode`\.=\@m
}{%
  \def\@noitemerr{\@latex@warning{Empty `thebibliography' environment}}%
  \endlist
}
\makeatother

\begin{document}

\mbox{}
\vspace{1cm}
\maketitle
\mbox{}
\vspace{1cm}

\begin{abstract}

\noindent Algebraic structures in which the property of commutativity is
substituted by the mediality property are introduced. We consider
(associative) graded algebras and instead of almost commutativity (generalized
commutativity or $\varepsilon$-commutativity) we introduce almost mediality
("commutativity-to-mediality" ansatz). Higher graded twisted products and
\textquotedblleft deforming\textquotedblright\ brackets (being the medial
analog of Lie brackets) are defined. Toyoda's theorem which connects
(universal) medial algebras with abelian algebras is proven for the almost
medial graded algebras introduced here. In a similar way we generalize tensor
categories and braided tensor categories. A polyadic (non-strict) tensor
category has an $n$-ary tensor product as an additional multiplication with
$n-1$ associators of the arity $2n-1$ satisfying a $\left(  n^{2}+1\right)
$-gon relation, which is a polyadic analog of the pentagon axiom. Polyadic
monoidal categories may contain several unit objects, and it is also possible
that all objects are units. A new kind of polyadic categories (called groupal)
is defined: they are close to monoidal categories, but may not contain units:
instead the querfunctor and (natural) functorial isomorphisms, the quertors,
are considered (by analogy with the querelements in $n$-ary groups). The
arity-nonreducible $n$-ary braiding is introduced and the equation for it is
derived, which for $n=2$ coincides with the Yang-Baxter equation. Then,
analogously to the first part of the paper, we introduce \textquotedblleft
medialing\textquotedblright\ instead of braiding and construct
\textquotedblleft medialed\textquotedblright\ polyadic tensor categories.

\end{abstract}

\thispagestyle{empty}

\mbox{}

\newpage
\mbox{}
\vspace{1cm}

\tableofcontents

\pagestyle{fancy}
\addtolength{\footskip}{15pt}

\renewcommand{\sectionmark}[1]{%
\markboth{\textmd{
\  \thesection.} { \scshape #1}}{}}
\renewcommand{\subsectionmark}[1]{%
\markright{
\mbox{\;}\\[5pt]
\textmd{#1}}{}}

\fancyhead{}
\fancyhead[EL,OR]{\leftmark}
\fancyhead[ER,OL]{\rightmark}
\fancyfoot[C]{\scshape - \textcolor{BrickRed}{\thepage} \ -}

\renewcommand\headrulewidth{0.5pt}
\fancypagestyle {plain1}{ %
\fancyhf{}
\renewcommand {\headrulewidth }{0pt}
\renewcommand {\footrulewidth }{0pt}
}

\fancypagestyle{plain}{ %
\fancyhf{}
\fancyfoot[C]{\scshape - \thepage \ -}
\renewcommand {\headrulewidth }{0pt}
\renewcommand {\footrulewidth }{0pt}
}

\fancypagestyle{fancyref}{ %
\fancyhf{} 
\fancyhead[C]{\scshape R\lowercase{eferences} }
\fancyfoot[C]{\scshape - \thepage \ -}
\renewcommand {\headrulewidth }{0.5pt}
\renewcommand {\footrulewidth }{0pt}
}

\fancypagestyle{emptyf}{
\fancyhead{}
\fancyfoot[C]{\thepage}
\renewcommand{\headrulewidth}{0pt}
}

\section{\textsc{Introduction}}

The commutativity property and its \textquotedblleft
breaking\textquotedblright\ are quite obvious and unique for binary algebraic
structures, because the permutation group $S_{2}$ has only one non-identity
element. If the operation is $n$-ary however, then one has $n!-1$ non-identity
permutations from $S_{n}$, and the uniqueness is lost. The standard way to
bring uniqueness to an $n$-ary structure is by restricting to a particular
$n$-ary commutation by fixing one chosen permutation using external (sometimes
artificial) criteria. We introduce a different, canonical approach: to use
another property which would be unique by definition, but which can give
commutativity in special cases. Mediality \cite{mur39} (acting on $n^{2}$
elements) is such a property which can be substituted for commutativity
(acting on $n$ elements) in the generators/relations description of $n$-ary
structures. For $n=2$, any medial magma is a commutative monoid, and moreover
for binary groups commutativity immediately follows from mediality.

In the first part of our paper we consider $n$-ary graded algebras and propose
the following idea: instead of considering the non-unique commutativity
property and its \textquotedblleft breaking\textquotedblright, to investigate
the unique property of mediality and its \textquotedblleft
breaking\textquotedblright. We exploit this \textquotedblleft
commutativity-to-mediality\textquotedblright\ ansatz to introduce and study
almost medial $n$-ary graded algebras by analogy with almost commutative
algebras (generalized or $\varepsilon$-commutative graded algebras)
\cite{rit/wyl,sch79}, and $\beta$-commutative algebras \cite{bah/mon/zai}
(see, also, \cite{bon/pij,cov/mic,mor/ovs10}). We prove an analogue of
Toyoda's theorem, which originally connected medial algebras with abelian
algebras \cite{toy41}, for almost medial $n$-ary graded algebras, which we
introduce. Note that almost co-mediality for polyadic bialgebras was
introduced earlier in \cite{dup2018d}. For other (binary) generalizations of
grading, see, e.g. \cite{dal/sba,eld05,nys05}.

The second part of the paper is devoted to a similar consideration of tensor
categories \cite{maclane1,eti/gel/nik/ost}. We define polyadic tensor
categories by considering an $n$-ary tensor product (which may not be iterated
from binary tensor products) and $n$-ary coherence conditions for the
corresponding associators. The peculiarities of polyadic semigroupal and
monoidal categories are studied and the differences from the corresponding
binary tensor categories are outlined. We introduce a new kind of tensor
categories, polyadic nonunital \textquotedblleft groupal\textquotedblright%
\ categories, which contain a \textquotedblleft querfunctor\textquotedblright%
\ and \textquotedblleft quertors\textquotedblright\ (similar to querelements
in $n$-ary groups \cite{dor3,pos}). We introduce arity-nonreducible $n$-ary
braidings and find the equation for them that in the binary case turns into
the Yang-Baxter equation in the tensor product form. Finally, we apply the
\textquotedblleft commutativity-to-mediality\textquotedblright\ ansatz to
braided tensor categories \cite{joy/str1} and introduce \textquotedblleft
medialing\textquotedblright\ and corresponding \textquotedblleft
medialed\textquotedblright\ tensor categories.

The proposed \textquotedblleft commutativity-to-mediality\textquotedblright%
\ ansatz can lead to medial $n$-ary superalgebras and Lie superalgebras, as
well as to a medial analog of noncommutative geometry.

\section{\textsc{Preliminaries}}

The standard way to generalize the commutativity is using graded vector spaces
and corresponding algebras together with the commutation factor defined on
some abelian grading group (see, e.g. \cite{rit/wyl,sch79} and
\cite{bourbaki98,nas/oys04}). First, recall this concept from a slightly
different viewpoint.

\subsection{Binary gradation}

Let $\mathcal{A}\equiv\mathcal{A}^{\left(  2\right)  }=\left\langle A\mid
\mu_{2},\nu_{2};\lambda_{1}\right\rangle $ be an associative (binary) algebra
over a field $\Bbbk$ (having unit $\mathtt{1}\in\Bbbk$ and zero $\mathtt{0}%
\in\Bbbk$) with unit $e$ (i.e. it is a unital $\Bbbk$-algebra) and zero $z\in
A$. Here $A$ is the underling set and $\mu_{2}:A\otimes A\rightarrow A$ is the
(bilinear) binary multiplication (which we write as $\mu_{2}\left[
a,b\right]  $, $a,b\in A$), usually in the binary case denoted by dot $\mu
_{2}\equiv\left(  \cdot\right)  $, and $\nu_{2}:A\otimes A\rightarrow A$ is
the (binary) addition denoted by $\left(  +\right)  $, and a third (linear)
operation $\lambda_{1}$ is the action $\lambda_{1}:K\otimes A\rightarrow A$
(widely called a \textquotedblleft scalar multiplication\textquotedblright,
but this is not always true, as can be seen from the polyadic case
\cite{dup2019}).

Informally, if $\mathcal{A}$ as a vector space can be decomposed into a direct
sum, then one can introduce the \textit{binary gradation concept}: each
element $a\in A$ is endowed by an additional characteristic, its
\textit{gradation} denoted by a prime $a^{\prime}$ showing to which subspace
it belongs, such that $a^{\prime}$ belongs to a discrete abelian group
(initially $\mathbb{N}$ simply to \textquotedblleft
enumerate\textquotedblright\ the subspaces, and this can be further
generalized to a commutative semigroup). This group is called the
\textit{binary grading group} $\mathcal{G}=\left\langle G,\nu_{2}^{\prime
}\right\rangle $, and usually its operation is written as plus $\nu
_{2}^{\prime}\equiv\left(  +^{\prime}\right)  $, and the neutral element by
$0^{\prime}$. Denote the subset of \textit{homogeneous elements} of degree
$a^{\prime}\in G$ by $A_{a^{\prime}}$ \cite{sch79,dade80}.

\begin{definition}
An associative algebra $\mathcal{A}$ is called a \textit{binary graded
algebra} over $\Bbbk$ (or $G$-algebra $\mathcal{A}_{G}$), if the algebra
multiplication $\mu_{2}$ respects the gradation i.e.%
\begin{equation}
\mu_{2}\left[  A_{a^{\prime}},A_{b^{\prime}}\right]  \equiv A_{a^{\prime}%
}\cdot A_{b^{\prime}}\subseteq A_{a^{\prime}+^{\prime}b^{\prime}%
},\ \ \ \ \ \ \forall a^{\prime},b^{\prime}\in G, \label{maa}%
\end{equation}
where equality corresponds to \textit{ strong gradation}.
\end{definition}

If there exist invertible elements of each degree $a^{\prime}\in G$, then
$\mathcal{A}$ is called a \textit{cross product}, and if all non-zero
homogeneous elements are invertible, $\mathcal{A}$ is a \textit{graded
division algebra} \cite{dade80}. Homogeneous (binary) morphisms $\varphi
:\mathcal{A}_{G}\rightarrow\mathcal{B}_{G}$ preserve the grading
$\varphi\left(  A_{a^{\prime}}\right)  \subset B_{a^{\prime}},\forall
a^{\prime}\in G$, and the kernel of $\varphi$ is an homogeneous ideal. The
corresponding class of $G$-algebras and the homogeneous morphisms form a
category of $G$-algebras $G$-$\mathtt{Alg}$ (for details, see, e.g.
\cite{bourbaki98,dade80}).

\subsection{Almost commutativity}

The graded algebras have a rich multiplicative structure, because of the
possibility to deform (or twist) the algebra product $\mu_{2}$ by a function
depending on the gradation. Let us consider the \textit{twisting function
}(\textit{twist factor}) $\tau:G\times G\rightarrow\Bbbk$.

\begin{definition}
A \textit{twisted graded product} $\mu_{2}^{\left(  \tau\right)  }$ is defined
for homogeneous elements by%
\begin{equation}
\mu_{2}^{\left(  \tau\right)  }\left[  a,b\right]  =\tau\left(  a^{\prime
},b^{\prime}\right)  \mu_{2}\left[  a,b\right]  ,\ \ \ a,b\in A;\ \ a^{\prime
},b^{\prime}\in G. \label{mt}%
\end{equation}

\end{definition}

\begin{proposition}
If the twisted algebra $\left\langle A\mid\mu_{2}^{\left(  \tau\right)
}\right\rangle $ is associative, then the twisting function becomes a
2-cocycle $\tau\mapsto\sigma:G\times G\rightarrow\Bbbk^{\times}$ on the
abelian group $G$ satisfying%
\begin{equation}
\sigma\left(  a^{\prime},b^{\prime}\right)  \sigma\left(  a^{\prime}%
+b^{\prime},c^{\prime}\right)  =\sigma\left(  a^{\prime},b^{\prime}+c^{\prime
}\right)  \sigma\left(  b^{\prime},c^{\prime}\right)  ,\ \ \ a^{\prime
},b^{\prime},c^{\prime}\in G. \label{s}%
\end{equation}

\end{proposition}

\begin{proof}
The result follows from the binary associativity condition for $\mu
_{2}^{\left(  \sigma\right)  }$.
\end{proof}

\begin{example}
An example of a solution to the functional equation (\ref{s}) is
$\sigma\left(  a^{\prime},b^{\prime}\right)  =\left(  \exp\left(  a^{\prime
}\right)  \right)  ^{b^{\prime}}$.
\end{example}

The classes of $\sigma$ form the (Schur) multiplier group \cite{sch79}, and
for further properties of $\sigma$ and a connection with the cohomology
classes $H^{2}\left(  G,\Bbbk\right)  $, see, e.g., \cite{cov/mic}.

In general, the twisted product (\ref{mt}) can be any polynomial in algebra
elements. Nevertheless, the special cases where $\mu_{2}^{\left(
\varepsilon_{0}\right)  }\left[  a,b\right]  $ becomes a fixed expression for
elements $a,b\in A$ are important.

\begin{definition}
If the twisted product coincides with the opposite product for all $a,b\in A$,
we call the twisting function a \textit{ 0-level} \textit{commutation factor
}$\tau\mapsto\varepsilon_{0}:G\times G\rightarrow\Bbbk^{\times}$, such that%
\begin{equation}
\mu_{2}^{\left(  \varepsilon_{0}\right)  }\left[  a,b\right]  =\mu_{2}\left[
b,a\right]  ,\ \ \text{or\ \ }\varepsilon_{0}\left(  a^{\prime},b^{\prime
}\right)  a\cdot b=b\cdot a,\ \ \ \forall a,b\in A,\ \ \ a^{\prime},b^{\prime
}\in G. \label{e0}%
\end{equation}

\end{definition}

\begin{definition}
A binary algebra $\mathcal{A}_{2}^{\left(  \varepsilon_{0}\right)  }$ for
which the twisted product coincides with the opposite product (\ref{e0}), is
called $0$\textit{-level almost commutative} ($\varepsilon_{0}$-commutative).
\end{definition}

\begin{assertion}
If the algebra for which (\ref{e0}) takes place is associative, the 0-level
commutation factor $\varepsilon_{0}$ satisfies the relations%
\begin{align}
\varepsilon_{0}\left(  a^{\prime},b^{\prime}\right)  \varepsilon_{0}\left(
b^{\prime},a^{\prime}\right)   &  =\mathtt{1},\label{e01}\\
\varepsilon_{0}\left(  a^{\prime}+b^{\prime},c^{\prime}\right)   &
=\varepsilon_{0}\left(  a^{\prime},c^{\prime}\right)  \varepsilon_{0}\left(
b^{\prime},c^{\prime}\right)  ,\label{e02}\\
\varepsilon_{0}\left(  a^{\prime},b^{\prime}+c^{\prime}\right)   &
=\varepsilon_{0}\left(  a^{\prime},b^{\prime}\right)  \varepsilon_{0}\left(
a^{\prime},c^{\prime}\right)  ,\ \ \ a^{\prime},b^{\prime},c^{\prime
},d^{\prime}\in G. \label{e03}%
\end{align}

\end{assertion}

\begin{proof}
The first relation (\ref{e01}) follows from permutation in (\ref{e0}) twice.
The next ones follow from permutation in two ways: for (\ref{e02}) $a\cdot
b\cdot c\mapsto a\cdot c\cdot b\mapsto c\cdot a\cdot b$ and $\left(  a\cdot
b\right)  \cdot c\mapsto c\cdot\left(  a\cdot b\right)  $, and for (\ref{e03})
$a\cdot b\cdot c\mapsto b\cdot a\cdot c\mapsto b\cdot c\cdot a$ and
$a\cdot\left(  b\cdot c\right)  \mapsto\left(  b\cdot c\right)  \cdot a$,
using (\ref{e0}).
\end{proof}

In a more symmetric form this is%
\begin{equation}
\varepsilon_{0}\left(  a^{\prime}+b^{\prime},c^{\prime}+d^{\prime}\right)
=\varepsilon_{0}\left(  a^{\prime},c^{\prime}\right)  \varepsilon_{0}\left(
b^{\prime},c^{\prime}\right)  \varepsilon_{0}\left(  a^{\prime},d^{\prime
}\right)  \varepsilon_{0}\left(  b^{\prime},d^{\prime}\right)  . \label{e00}%
\end{equation}

The following general expression%
\begin{equation}
\varepsilon_{0}\left(  \sum\limits_{i_{a}=1}^{j_{a}}a_{i_{a}}^{\prime}%
,\sum\limits_{i_{b}=1}^{j_{b}}b_{i_{b}}^{\prime}\right)  =\prod\limits_{i_{a}%
=1}^{j_{a}}\prod\limits_{i_{b}=1}^{j_{b}}\varepsilon_{0}\left(  a_{i_{a}%
}^{\prime},b_{i_{b}}^{\prime}\right)  ,\ \ \ a_{i_{a}}^{\prime},b_{i_{b}%
}^{\prime}\in G,\ \ \ i_{a},i_{b},j_{a},j_{b}\in\mathbb{N},
\end{equation}
can be written. In the case of equal elements we have%
\begin{equation}
\varepsilon_{0}\left(  j_{a}a^{\prime},j_{b}b^{\prime}\right)  =\left(
\varepsilon_{0}\left(  a^{\prime},b^{\prime}\right)  \right)  ^{j_{a}j_{b}}.
\end{equation}

\begin{remark}
\label{rem-estand}Recall that the standard commutation factor $\varepsilon
:G\times G\rightarrow\Bbbk^{\times}$ of an almost commutative ($\varepsilon
$-commutative or $\varepsilon$-symmetric) associative algebra is defined in a
different way \cite{rit/wyl,sch79}%
\begin{equation}
\varepsilon\left(  a^{\prime},b^{\prime}\right)  b\cdot a=a\cdot b. \label{e}%
\end{equation}
Comparing with (\ref{e0}) we have%
\begin{equation}
\varepsilon_{0}\left(  a^{\prime},b^{\prime}\right)  =\varepsilon\left(
b^{\prime},a^{\prime}\right)  ,\ \ \ \forall a^{\prime},b^{\prime}\in G.
\label{ee0}%
\end{equation}

\end{remark}

\subsection{Tower of higher level commutation brackets}

Let us now construct the tower of higher level commutation factors and
brackets using the following informal reasoning. We \textquotedblleft
deform\textquotedblright\ the almost commutativity relation (\ref{e0}) by a
function $L_{0}:A\times A\rightarrow A$ as%
\begin{equation}
\varepsilon_{0}\left(  a^{\prime},b^{\prime}\right)  a\cdot b=b\cdot
a+L_{0}^{\left(  \varepsilon_{0}\right)  }\left(  a,b\right)  ,\ \ \ \forall
a,b\in A,\ \ \ a^{\prime},b^{\prime}\in G, \label{L0}%
\end{equation}
where $\varepsilon_{0}\left(  a^{\prime},b^{\prime}\right)  $ is the $0$-level
commuting factor satisfying (\ref{e01})--(\ref{e03}).

Consider the function (bracket) $L_{0}^{\left(  \varepsilon_{0}\right)
}\left(  a,b\right)  $ as a multiplication of a new algebra%
\begin{equation}
\mathcal{A}_{2}^{L_{0}}=\left\langle A\mid\mu_{2}^{\left(  \varepsilon
_{0},L_{0}\right)  }=L_{0}^{\left(  \varepsilon_{0}\right)  }\left(
a,b\right)  \right\rangle \label{al}%
\end{equation}
called a $0$\textit{-level bracket algebra}. Then (\ref{L0}) can be treated as
its \textquotedblleft representation\textquotedblright\ by the associative
algebra $\mathcal{A}$.

\begin{proposition}
The algebra $\mathcal{A}_{2}^{L_{0}}$ is almost commutative with the
commutation factor $\left(  -\varepsilon_{0}^{-1}\right)  $.
\end{proposition}

\begin{proof}
Using (\ref{L0}) and (\ref{e01})--(\ref{e03}) we get $\varepsilon_{0}\left(
b^{\prime},a^{\prime}\right)  L_{0}^{\left(  \varepsilon_{0}\right)  }\left(
a,b\right)  +L_{0}^{\left(  \varepsilon_{0}\right)  }\left(  b,a\right)  =0$,
which can be rewritten in the almost commutativity form (\ref{e0}) as $\left(
-\varepsilon_{0}\left(  b^{\prime},a^{\prime}\right)  \right)  L_{0}^{\left(
\varepsilon_{0}\right)  }\left(  a,b\right)  =L_{0}^{\left(  \varepsilon
_{0}\right)  }\left(  b,a\right)  $. It follows from (\ref{e01}) that%
\begin{equation}
\left(  -\varepsilon_{0}^{-1}\left(  a^{\prime},b^{\prime}\right)  \right)
L_{0}^{\left(  \varepsilon_{0}\right)  }\left(  a,b\right)  =L_{0}^{\left(
\varepsilon_{0}\right)  }\left(  b,a\right)  . \label{el}%
\end{equation}

\end{proof}

The triple identity for $L_{0}^{\left(  \varepsilon_{0}\right)  }\left(
a,b\right)  $ can be obtained using (\ref{e01})--(\ref{e03}), (\ref{L0})\ and
(\ref{el})%
\begin{align}
&  \varepsilon_{0}\left(  c^{\prime},a^{\prime}\right)  L_{0}^{\left(
\varepsilon_{0}\right)  }\left(  L_{0}^{\left(  \varepsilon_{0}\right)
}\left(  a,b\right)  ,c\right)  +\varepsilon_{0}\left(  a^{\prime},b^{\prime
}\right)  L_{0}^{\left(  \varepsilon_{0}\right)  }\left(  L_{0}^{\left(
\varepsilon_{0}\right)  }\left(  b,c\right)  ,a\right) \label{j0}\\
&  +\varepsilon_{0}\left(  b^{\prime},c^{\prime}\right)  L_{0}^{\left(
\varepsilon_{0}\right)  }\left(  L_{0}^{\left(  \varepsilon_{0}\right)
}\left(  c,a\right)  ,b\right)  =0,\ \ \ \ \forall a,b,c\in A,\ \ \ a^{\prime
},b^{\prime},c^{\prime}\in G.\nonumber
\end{align}

In the more symmetric form using (\ref{e00}) we have{\tiny
\begin{align}
&  \varepsilon_{0}\left(  c^{\prime},b^{\prime}\right)  \varepsilon_{0}\left(
d^{\prime},a^{\prime}\right)  L_{0}^{\left(  \varepsilon_{0}\right)  }\left(
L_{0}^{\left(  \varepsilon_{0}\right)  }\left(  a,b\right)  ,L_{0}^{\left(
\varepsilon_{0}\right)  }\left(  c,d\right)  \right)  +\varepsilon_{0}\left(
d^{\prime},c^{\prime}\right)  \varepsilon_{0}\left(  a^{\prime},b^{\prime
}\right)  L_{0}^{\left(  \varepsilon_{0}\right)  }\left(  L_{0}^{\left(
\varepsilon_{0}\right)  }\left(  b,c\right)  ,L_{0}^{\left(  \varepsilon
_{0}\right)  }\left(  d,a\right)  \right) \nonumber\\
&  \varepsilon_{0}\left(  a^{\prime},d^{\prime}\right)  \varepsilon_{0}\left(
b^{\prime},c^{\prime}\right)  L_{0}^{\left(  \varepsilon_{0}\right)  }\left(
L_{0}^{\left(  \varepsilon_{0}\right)  }\left(  c,d\right)  ,L_{0}^{\left(
\varepsilon_{0}\right)  }\left(  a,b\right)  \right)  +\varepsilon_{0}\left(
b^{\prime},a^{\prime}\right)  \varepsilon_{0}\left(  c^{\prime},d^{\prime
}\right)  L_{0}^{\left(  \varepsilon_{0}\right)  }\left(  L_{0}^{\left(
\varepsilon_{0}\right)  }\left(  d,a\right)  ,L_{0}^{\left(  \varepsilon
_{0}\right)  }\left(  b,c\right)  \right)  =0. \label{j1}%
\end{align}
}By analogy with (\ref{L0}) we successively further \textquotedblleft
deform\textquotedblright\ (\ref{el}) then introduce \textquotedblleft
deforming\textquotedblright\ functions and higher level commutation factors in
the following way.

\begin{definition}
The $k$-level commutation factor $\varepsilon_{k}\left(  a^{\prime},b^{\prime
}\right)  $ is defined by the following \textquotedblleft
difference-like\textquotedblright\ equations%
\begin{align}
\varepsilon_{1}\left(  a^{\prime},b^{\prime}\right)  L_{0}^{\left(
\varepsilon_{0}\right)  }\left(  a,b\right)   &  =L_{0}^{\left(
\varepsilon_{0}\right)  }\left(  b,a\right)  +L_{1}^{\left(  \varepsilon
_{0},\varepsilon_{1}\right)  }\left(  a,b\right)  ,\label{e1}\\
\varepsilon_{2}\left(  a^{\prime},b^{\prime}\right)  L_{1}^{\left(
\varepsilon_{0},\varepsilon_{1}\right)  }\left(  a,b\right)   &
=L_{1}^{\left(  \varepsilon_{0},\varepsilon_{1}\right)  }\left(  b,a\right)
+L_{2}^{\left(  \varepsilon_{0},\varepsilon_{1},\varepsilon_{2}\right)
}\left(  a,b\right)  ,\\
&  \vdots\nonumber\\
\varepsilon_{k}\left(  a^{\prime},b^{\prime}\right)  L_{k-1}^{\left(
\varepsilon_{0},\varepsilon_{1},\ldots,\varepsilon_{k-1}\right)  }\left(
a,b\right)   &  =L_{k-1}^{\left(  \varepsilon_{0},\varepsilon_{1}%
,\ldots,\varepsilon_{k-1}\right)  }\left(  b,a\right)  +L_{k}^{\left(
\varepsilon_{0},\varepsilon_{1},\ldots,\varepsilon_{k}\right)  }\left(
a,b\right)  . \label{ek}%
\end{align}

\end{definition}

\begin{definition}
$k$-level almost commutativity is defined by the vanishing of the last
\textquotedblleft deforming\textquotedblright\ function%
\begin{equation}
L_{k}^{\left(  \varepsilon_{0},\varepsilon_{1},\ldots,\varepsilon_{k}\right)
}\left(  a,b\right)  =0,\ \ \ \forall a,b\in A,
\end{equation}
and can be expressed in a form analogous to (\ref{e0})%
\begin{equation}
\varepsilon_{k}\left(  a^{\prime},b^{\prime}\right)  L_{k-1}^{\left(
\varepsilon_{0},\varepsilon_{1},\ldots,\varepsilon_{k-1}\right)  }\left(
a,b\right)  =L_{k-1}^{\left(  \varepsilon_{0},\varepsilon_{1},\ldots
,\varepsilon_{k-1}\right)  }\left(  b,a\right)  .
\end{equation}

\end{definition}

\begin{proposition}
All higher level \textquotedblleft deforming\textquotedblright\ functions
$L_{i}^{\left(  \varepsilon_{0},\varepsilon_{1},\ldots,\varepsilon_{i}\right)
}$, $i=1,\ldots,k$ can be expressed through $L_{0}^{\left(  \varepsilon
_{0}\right)  }\left(  a,b\right)  $ from (\ref{L0}) multiplied by a
combination of the lower level commutation factors $\varepsilon_{i}\left(
a^{\prime},b^{\prime}\right)  $, $i=1,\ldots,k$.
\end{proposition}

\begin{proof}
This follows from the equations (\ref{e1})--(\ref{ek}).
\end{proof}

The first such expressions are%
\begin{align}
L_{1}^{\left(  \varepsilon_{0},\varepsilon_{1}\right)  }\left(  a,b\right)
&  =\left[  \varepsilon_{1}\left(  a^{\prime},b^{\prime}\right)
+\varepsilon_{0}\left(  b^{\prime},a^{\prime}\right)  \right]  L_{0}^{\left(
\varepsilon_{0}\right)  }\left(  a,b\right)  ,\\
L_{2}^{\left(  \varepsilon_{0},\varepsilon_{1},\varepsilon_{2}\right)
}\left(  a,b\right)   &  =\left[  \varepsilon_{2}\left(  a^{\prime},b^{\prime
}\right)  \left(  \varepsilon_{1}\left(  a^{\prime},b^{\prime}\right)
+\varepsilon_{0}\left(  b^{\prime},a^{\prime}\right)  \right)  +\varepsilon
_{1}\left(  a^{\prime},b^{\prime}\right)  \varepsilon_{0}\left(  b^{\prime
},a^{\prime}\right)  +\mathtt{1}\right]  L_{0}^{\left(  \varepsilon
_{0}\right)  }\left(  a,b\right)  ,\\
&  \vdots\nonumber
\end{align}

Recall the definition of the $\varepsilon$-Lie bracket \cite{sch79}%
\begin{equation}
\left[  a,b\right]  _{\varepsilon}=a\cdot b-\varepsilon\left(  a^{\prime
},b^{\prime}\right)  b\cdot a,\ \ \ \forall a,b\in A,\ \ \ a^{\prime
},b^{\prime}\in G. \label{ab}%
\end{equation}

\begin{assertion}
The $0$-level \textquotedblleft deforming\textquotedblright\ function
$L_{0}^{\left(  \varepsilon_{0}\right)  }\left(  a,b\right)  $ is the
$\varepsilon_{0}$-twisted $\varepsilon$-Lie bracket%
\begin{equation}
L_{0}^{\left(  \varepsilon_{0}\right)  }\left(  a,b\right)  =\varepsilon
_{0}\left(  a^{\prime},b^{\prime}\right)  \left[  a,b\right]  _{\varepsilon
=\varepsilon_{0}}. \label{ll}%
\end{equation}

\end{assertion}

\begin{proof}
This follows from (\ref{L0}) and (\ref{ab}).
\end{proof}

\begin{remark}
The relations (\ref{j0}) and (\ref{j1}) are analogs of the $\varepsilon
$-Jacobi identity of the $\varepsilon$-Lie algebra \cite{sch79}.
\end{remark}

\begin{corollary}
All higher level \textquotedblleft deforming\textquotedblright\ functions
$L_{i}^{\left(  \varepsilon_{0},\varepsilon_{1},\ldots,\varepsilon_{i}\right)
}$, $i=1,\ldots,k$ can be expressed through the twisted $\varepsilon$-Lie
bracket (\ref{ab}) with twisting coefficients.
\end{corollary}

In search of a polyadic analog of almost commutativity, we will need some
additional concepts, beyond the permutation of two elements (in the binary
case), called commutativity, and various sums of permutations (of $n$
elements, in $n$-ary case, which are usually non-unique).

Instead we propose to consider a new concept, \textit{polyadic mediality}
(which gives a unique relation between $n^{2}$ elements in $n$-ary case), as a
polyadic inductive generalization of commutativity. We then twist the
multiplication by a gradation (as in the binary case above) to obtain the
polyadic version of almost commutativity as \textit{almost mediality}.
However, let us first recall the binary and polyadic versions of the mediality property.

\subsection{Medial binary magmas and quasigroups}

The mediality property was introduced as a generalization of the associative
law for quasigroups, which are a direct generalization of abelian groups
\cite{mur39}. Other names for mediality are entropicity, bisymmetry,
alternaton and abelianness (see, e.g., \cite{acz48,jez/kep,eva63}).

Let $\mathcal{M}=\left\langle M\mid\mu_{2}\right\rangle $ be a binary
\textit{magma} (a closed set $M$ with one binary operation $\mu_{2}$ without
any additional properties, also called a (Hausmann-Ore) \textit{groupoid}%
\footnote{This should not be confused with the Brandt groupoid or virtual
group.}).

\begin{definition}
A (binary) magma $\mathcal{M}$ is called \textit{medial}, if%
\begin{equation}
\mu_{2}\left[  \mu_{2}\left[  a,b\right]  ,\mu_{2}\left[  c,d\right]  \right]
=\mu_{2}\left[  \mu_{2}\left[  a,c\right]  ,\mu_{2}\left[  b,d\right]
\right]  ,\ \ \ \forall a,b,c,d\in M. \label{mm}%
\end{equation}

\end{definition}

\begin{definition}
We call the product of elements in the r.h.s. of (\ref{mm}) \textit{medially
symmetric} to the l.h.s. product.
\end{definition}

Obviously, if a magma $\mathcal{M}$ contains a neutral element (identity)
$e\in M$, such that $\mu_{2}\left[  a,e\right]  =\mu_{2}\left[  e,a\right]
=a$, $\forall a\in M$, then $\mathcal{M}$ is commutative $\mu_{2}\left[
a,b\right]  =\mu_{2}\left[  b,a\right]  $, $\forall a,b\in M$. Therefore, any
commutative monoid is an example of a medial magma. Numerous different kinds
of magma and their classification are given in \cite{jez/kep}. If a magma
$\mathcal{M}$ is cancellative ($\mu_{2}\left[  a,b\right]  =\mu_{2}\left[
a,c\right]  \Rightarrow b=c$, $\mu_{2}\left[  a,c\right]  =\mu_{2}\left[
b,c\right]  \Rightarrow a=b$, $\forall a,b,c\in M$), it is a binary quasigroup
$\mathcal{Q}=\left\langle Q\mid\mu_{2}\right\rangle $ for which the equations
$\mu_{2}\left[  a,x\right]  =b$, $\mu_{2}\left[  y,a\right]  =b$, $\forall
a,b\in Q$ , have a unique solution \cite{how73}. Moreover \cite{sho49}, every
medial cancellative magma can be embedded in a \textit{medial quasigroup}
(satisfying (\ref{mm})), and the reverse statement is also true
\cite{jez/kep93}. For a recent comprehensive review on quasigroups(including
medial and $n$-ary ones), see, e.g. \cite{shcherbacov}, and references therein.

The structure of medial quasigroups is determined by the
(Bruck-Murdoch-)Toyoda theorem \cite{bru44,mur41,toy41}.

\begin{theorem}
[Toyoda theorem]Any medial quasigroup $\mathcal{Q}_{medial}=\left\langle
Q\mid\mu_{2}\right\rangle $ can be presented in the linear (functional) form%
\begin{equation}
\mu_{2}\left[  a,b\right]  =\nu_{2}\left[  \nu_{2}\left[  \varphi\left(
a\right)  ,\psi\left(  b\right)  \right]  ,c\right]  =\varphi\left(  a\right)
+\psi\left(  b\right)  +c,\ \ \ \forall a,b,c\in Q, \label{ma}%
\end{equation}
where $\left\langle Q\mid\nu_{2}\equiv\left(  +\right)  \right\rangle $ is an
abelian group and $\varphi,\psi:Q\rightarrow Q$ are commuting automophisms
$\varphi\circ\psi=\psi\circ\varphi$, and $c\in Q$ is fixed.
\end{theorem}

If $\mathcal{Q}_{medial}$ has an idempotent element (denoted by $0$), then%
\begin{equation}
\mu_{2}\left[  a,b\right]  =\nu_{2}\left[  \varphi\left(  a\right)
,\psi\left(  b\right)  \right]  =\varphi\left(  a\right)  +\psi\left(
b\right)  ,\ \ \ \forall a,b\in Q,
\end{equation}
It follows from the Toyoda theorem, that medial quasigroups are isotopic to
abelian groups, and their structure theories are very close \cite{bru44,mur41}.

The mediality property (\ref{mm}) for binary semigroups leads to various
consequences \cite{chr69,nordahl}. Indeed, every medial semigroup
$\mathcal{S}_{medial}=\left\langle S\mid\mu_{2}\right\rangle $ is a
\textit{Putcha semigroup} ($b\in S^{1}aS^{1}\Rightarrow b^{m}\in S^{1}%
a^{2}S^{1}$, $\forall a,b\in S$, $m\in\mathbb{N}$, $S^{1}=S\cup\left\{
1\right\}  $), and therefore $\mathcal{S}_{medial}$ can be decomposed into the
semilattice ($a^{2}=a\wedge ab=ba$, $\forall a,b\in S$) of \textit{Archimedean
semigroups} ($\forall a,b\in S$, $\exists m,k\in\mathbb{N}$, $a^{m}%
=S^{1}bS^{1}\wedge b^{k}=S^{1}aS^{1}\wedge ab=ba$). If a medial semigroup
$\mathcal{S}_{medial}$ is left (right) \textit{cancellative},
$ab=ac\Rightarrow b=c$ ($ba=ca\Rightarrow b=c$), then it is left (right)
\textit{commutative} $abc=bac$ ($abc=acb$), $\forall a,b,c\in S$ and left
(right) \textit{separative}, $ab=a^{2}\wedge ba=b^{2}$, $\forall a,b\in S$
($ab=b^{2}\wedge ba=a^{2}$) (for a review, see, \cite{nagy2001}).

For a binary group $\left\langle G\mid\mu_{2}\right\rangle $ mediality implies
commutativity, because, obviously, $abcd=acbd\Rightarrow bc=cb$, $\forall
a,b,c,d\in G$. This is not the case for polyadic groups, where mediality
implies semicommutativity only (see e.g., \cite{gla/gle82,dog16}).

Let $\mathcal{A}=\left\langle A\mid\mu_{2},\nu_{2};\lambda_{1}\right\rangle $
be a binary $\Bbbk$-algebra, not necessarily unital, cancellative and
associative. Then mediality provides the corresponding behavior which depends
on the properties of the \textquotedblleft vector
multiplication\textquotedblright\ $\mu_{2}$. For instance, for unital
cancellative and associative algebras, mediality implies commutativity, as for
groups \cite{gla/gle82}.

\section{\textsc{Almost medial binary graded algebras}}

Consider an associative binary algebra $\mathcal{A}$ over a field $\Bbbk$. We
introduce a weaker version of gradation than in (\ref{maa}).

\begin{definition}
An associative algebra $\mathcal{A}$ is called a \textit{binary higher graded
algebra} over $\Bbbk$, if the algebra multiplication of four ($=2^{2}$)
elements respects the gradation%
\begin{equation}
\mathbf{\mu}_{4}\left[  A_{a^{\prime}},A_{b^{\prime}},A_{c^{\prime}%
},A_{d^{\prime}}\right]  \equiv A_{a^{\prime}}\cdot A_{b^{\prime}}\cdot
A_{c^{\prime}}\cdot A_{d^{\prime}}\subseteq A_{a^{\prime}+^{\prime}b^{\prime}%
},\ \ \ \ \ \ \forall a^{\prime},b^{\prime},c^{\prime},d^{\prime}\in G,
\label{m4a}%
\end{equation}
where equality corresponds to \textit{ strong higher gradation}.
\end{definition}

Instead of (\ref{mt}) let us introduce the \textit{higher twisting function}
(\textit{higher twist factor}) for four ($=2^{2}$) elements $\mathbf{\tau
}:G^{\times4}\rightarrow\Bbbk$.

\begin{definition}
A \textit{twisted (binary) higher graded product} $\mathbf{\mu}_{4}^{\left(
\tau\right)  }$ is defined for homogeneous elements by%
\begin{equation}
\mathbf{\mu}_{4}^{\left(  \tau\right)  }\left[  a,b,c,d\right]  =\mathbf{\tau
}\left(  a^{\prime},b^{\prime},c^{\prime},d^{\prime}\right)  a\cdot b\cdot
c\cdot d,\ \ \ a,b,c,d\in A;\ \ a^{\prime},b^{\prime},c^{\prime},d^{\prime}\in
G. \label{m4}%
\end{equation}

\end{definition}

An analog of (total) associativity for the twisted binary higher graded
product operation $\mathbf{\mu}_{4}^{\left(  \tau\right)  }$ is the following
condition on seven elements ($7=2\cdot2^{2}-1$) for all $a,b,c,d,t,u,v\in A$%
\begin{align}
&  \mathbf{\mu}_{4}^{\left(  \tau\right)  }\left[  \mathbf{\mu}_{4}^{\left(
\tau\right)  }\left[  a,b,c,d\right]  ,t,u,v\right]  =\mathbf{\mu}%
_{4}^{\left(  \tau\right)  }\left[  a,\mathbf{\mu}_{4}^{\left(  \tau\right)
}\left[  b,c,d,t\right]  ,u,v\right] \nonumber\\
&  =\mathbf{\mu}_{4}^{\left(  \tau\right)  }\left[  a,b,\mathbf{\mu}%
_{4}^{\left(  \tau\right)  }\left[  c,d,t,u\right]  ,v\right]  =\mathbf{\mu
}_{4}^{\left(  \tau\right)  }\left[  a,b,c,\mathbf{\mu}_{4}^{\left(
\tau\right)  }\left[  d,t,u,v\right]  \right]  . \label{mm4}%
\end{align}

\begin{proposition}
If the twisted higher graded product satisfies the higher analog of
associativity given by (\ref{mm4}), then the twisting function becomes a
higher analog of the cocycle (\ref{s}) $\mathbf{\tau}\mapsto\mathbf{\sigma
}:G^{\times4}\rightarrow\Bbbk^{\times}$ on the abelian group $G$ satisfying
for all $a^{\prime},b^{\prime},c^{\prime},d^{\prime},t^{\prime},u^{\prime
},v^{\prime}\in G$%
\begin{align}
&  \mathbf{\sigma}\left(  a^{\prime},b^{\prime},c^{\prime},d^{\prime}\right)
\mathbf{\sigma}\left(  a^{\prime}+b^{\prime}+c^{\prime}+d^{\prime},t^{\prime
},u^{\prime},v^{\prime}\right)  =\mathbf{\sigma}\left(  b^{\prime},c^{\prime
},d^{\prime},t^{\prime}\right)  \mathbf{\sigma}\left(  a^{\prime},b^{\prime
}+c^{\prime}+d^{\prime}+t^{\prime},u^{\prime},v^{\prime}\right) \label{ss}\\
&  =\mathbf{\sigma}\left(  c^{\prime},d^{\prime},t^{\prime},u^{\prime}\right)
\mathbf{\sigma}\left(  a^{\prime},b^{\prime},c^{\prime}+d^{\prime}+t^{\prime
}+u^{\prime},v^{\prime}\right)  =\mathbf{\sigma}\left(  d^{\prime},t^{\prime
},u^{\prime},v^{\prime}\right)  \mathbf{\sigma}\left(  a^{\prime},b^{\prime
},c^{\prime},d^{\prime}+t^{\prime}+u^{\prime}+v^{\prime}\right)  .\nonumber
\end{align}

\end{proposition}

Next we propose a medial analog of almost commutativity as follows. Instead of
deforming commutativity by the grading twist factor $\varepsilon_{0}$ as in
(\ref{e0}), we deform the mediality (\ref{mm}) by the higher twisting function
$\mathbf{\tau}$ (\ref{m4}).

\begin{definition}
If the higher twisted product coincides with the medially symmetric product
(see (\ref{mm})) for all $a,b\in A$, we call the twisting function a
$0$-\textit{level mediality factor }$\mathbf{\tau}\mapsto\mathbf{\rho}%
_{0}:G^{\times4}\rightarrow\Bbbk^{\times}$, such that (cf. (\ref{e0}))%
\begin{align}
\mathbf{\mu}_{4}^{\left(  \rho_{0}\right)  }\left[  a,b,c,d\right]   &
=\mathbf{\mu}_{4}\left[  a,c,b,d\right]  ,\ \ \text{or}\label{r1}\\
\mathbf{\rho}_{0}\left(  a^{\prime},b^{\prime},c^{\prime},d^{\prime}\right)
a\cdot b\cdot c\cdot d  &  =a\cdot c\cdot b\cdot d,\ \ \ \forall a,b,c,d\in
A,\ \ \ a^{\prime},b^{\prime},c^{\prime},d^{\prime}\in G. \label{r2}%
\end{align}

\end{definition}

From (\ref{r2}) follows the normalization condition for the mediality factor%
\begin{equation}
\mathbf{\rho}_{0}\left(  a^{\prime},a^{\prime},a^{\prime},a^{\prime}\right)
=\mathtt{1},\ \ \ \forall a^{\prime}\in G. \label{raa}%
\end{equation}

\begin{definition}
A binary algebra $\mathcal{A}_{2}^{\left(  \rho_{0}\right)  }=\left\langle
A\mid\mathbf{\mu}_{2},\nu_{2}\right\rangle $ for which the higher twisted
product coincides with the medially symmetric product $\mathbf{\mu}%
_{4}^{\left(  \rho_{0}\right)  }\left[  a,b,c,d\right]  =\mathbf{\mu}%
_{4}\left[  a,c,b,d\right]  $ (\ref{r2}), is called a $0$\textit{-level almost
medial} ($\mathbf{\rho}_{0}$-commutative) algebra.
\end{definition}

\begin{proposition}
If the algebra for which (\ref{r2}) holds is associative, the 0-level
mediality factor $\mathbf{\rho}_{0}$ satisfies the relations%
\begin{align}
\mathbf{\rho}_{0}\left(  a^{\prime},b^{\prime},c^{\prime},d^{\prime}\right)
\mathbf{\rho}_{0}\left(  a^{\prime},c^{\prime},b^{\prime},d^{\prime}\right)
&  =\mathtt{1},\ \ \ a^{\prime},b^{\prime},c^{\prime},d^{\prime},f^{\prime
},g^{\prime},h^{\prime}\in G,\ \ \mathtt{1}\in\Bbbk,\label{r01}\\
\mathbf{\rho}_{0}\left(  a^{\prime},c^{\prime}+d^{\prime}+f^{\prime}%
+g^{\prime},b^{\prime},h^{\prime}\right)   &  =\mathbf{\rho}_{0}\left(
a^{\prime},c^{\prime},b^{\prime},d^{\prime}\right)  \mathbf{\rho}_{0}\left(
c^{\prime},d^{\prime},b^{\prime},f^{\prime}\right)  \mathbf{\rho}_{0}\left(
d^{\prime},f^{\prime},b^{\prime},g^{\prime}\right)  \mathbf{\rho}_{0}\left(
f^{\prime},g^{\prime},b^{\prime},h^{\prime}\right)  ,\label{r02}\\
\mathbf{\rho}_{0}\left(  a^{\prime},g^{\prime},b^{\prime}+c^{\prime}%
+d^{\prime}+f^{\prime},h^{\prime}\right)   &  =\mathbf{\rho}_{0}\left(
a^{\prime},g^{\prime},b^{\prime}c^{\prime}\right)  \mathbf{\rho}_{0}\left(
c^{\prime},g^{\prime},d^{\prime},f^{\prime}\right)  \mathbf{\rho}_{0}\left(
d^{\prime},g^{\prime},f^{\prime},h^{\prime}\right)  \mathbf{\rho}_{0}\left(
b^{\prime},g^{\prime},c^{\prime},d^{\prime}\right)  . \label{r03}%
\end{align}

\end{proposition}

\begin{proof}
As in (\ref{e01}), the relation (\ref{r01}) follows from applying (\ref{r2})
twice. The next ones follow from permutation in two ways using (\ref{r1}): for
(\ref{r02})%
\begin{align}
a\cdot b\cdot\left(  c\cdot d\cdot f\cdot g\right)  \cdot h  &  \mapsto
a\cdot\left(  c\cdot d\cdot f\cdot g\right)  \cdot b\cdot
h,\ \ \ \ a,c,d,f,g,b,h\in A,\\
a\cdot b\cdot c\cdot d\cdot f\cdot g\cdot h  &  \mapsto a\cdot c\cdot b\cdot
d\cdot f\cdot g\cdot h\mapsto a\cdot c\cdot d\cdot b\cdot f\cdot g\cdot
h\nonumber\\
&  \mapsto a\cdot c\cdot d\cdot f\cdot b\cdot g\cdot h\mapsto a\cdot c\cdot
d\cdot f\cdot g\cdot b\cdot h,
\end{align}
and for (\ref{r03})%
\begin{align}
a\cdot\left(  b\cdot c\cdot d\cdot f\right)  \cdot g\cdot h  &  \mapsto a\cdot
g\cdot\left(  b\cdot c\cdot d\cdot f\right)  \cdot h,\\
a\cdot b\cdot c\cdot d\cdot f\cdot g\cdot h  &  \mapsto a\cdot b\cdot c\cdot
d\cdot g\cdot f\cdot h\mapsto a\cdot b\cdot c\cdot g\cdot d\cdot f\cdot
h\nonumber\\
&  \mapsto a\cdot b\cdot g\cdot c\cdot d\cdot f\cdot h\mapsto a\cdot g\cdot
b\cdot c\cdot d\cdot f\cdot h.
\end{align}

\end{proof}

\begin{assertion}
If the $0$-level almost medial algebra $\mathcal{A}_{2}^{\left(  \rho
_{0}\right)  }$ is cancellative, then it is isomorphic to an almost
commutative algebra.
\end{assertion}

\begin{proof}
After cancellation by $a$ and $d$ in (\ref{r2}), we obtain $\varepsilon
_{0}\left(  b^{\prime},c^{\prime}\right)  b\cdot c=c\cdot b$, where%
\begin{equation}
\varepsilon_{0}\left(  b^{\prime},c^{\prime}\right)  =\mathbf{\rho}_{0}\left(
a^{\prime},b^{\prime},c^{\prime},d^{\prime}\right)  .
\end{equation}
In case $\mathcal{A}_{2}^{\left(  \rho_{0}\right)  }$ is unital, one can
choose $\varepsilon_{0}\left(  b^{\prime},c^{\prime}\right)  =\mathbf{\rho
}_{0}\left(  e^{\prime},b^{\prime},c^{\prime},e^{\prime}\right)
\equiv\mathbf{\rho}_{0}\left(  0^{\prime},b^{\prime},c^{\prime},0^{\prime
}\right)  $, since the identity $e\in A$ is zero graded.
\end{proof}

\subsection{Tower of higher binary mediality brackets}

By analogy with (\ref{L0}), let us deform the medial twisted product
$\mathbf{\mu}_{4}^{\left(  \rho_{0}\right)  }$ (\ref{r1}) by the function
$M_{0}^{\left(  \rho_{0}\right)  }:A\times A\times A\times A\rightarrow A$ as
follows%
\begin{equation}
\mathbf{\rho}_{0}\left(  a^{\prime},b^{\prime},c^{\prime},d^{\prime}\right)
a\cdot b\cdot c\cdot d=a\cdot c\cdot b\cdot d+M_{0}^{\left(  \rho_{0}\right)
}\left(  a,b,c,d\right)  ,\ \ \ \forall a,b,c,d\in A,\ \ \ a^{\prime
},b^{\prime},c^{\prime},d^{\prime}\in G, \label{M0}%
\end{equation}
where $\mathbf{\rho}_{0}$ is a $0$-level mediality factor (\ref{r1}) which
satisfies (\ref{r01})--(\ref{r03}).

Let us next introduce a $4$-ary multiplication $\mathbf{\mu}_{4}^{\left(
\rho_{0},M_{0}\right)  }\left[  a,b,c,d\right]  =M_{0}^{\left(  \rho
_{0}\right)  }\left(  a,b,c,d\right)  $, $\forall a,b,c,d\in A$.

\begin{definition}
A $4$-ary algebra
\begin{equation}
\mathcal{A}_{4}^{\left(  \rho_{0},M_{0}\right)  }=\left\langle A\mid
\mathbf{\mu}_{4}^{\left(  \rho_{0},M_{0}\right)  }\right\rangle
\end{equation}
is called a $0$\textit{-level medial bracket algebra}.
\end{definition}

\begin{proposition}
The $4$-ary algebra $\mathcal{A}_{4}^{\left(  \rho_{0},M_{0}\right)  }$ is
almost medial with the mediality factor $\left(  -\mathbf{\rho}_{0}%
^{-1}\right)  $.
\end{proposition}

\begin{proof}
Using (\ref{M0}) and (\ref{r01})--(\ref{r03}) we get $\mathbf{\rho}_{0}\left(
a^{\prime},c^{\prime},b^{\prime},d^{\prime}\right)  M_{0}^{\left(  \rho
_{0}\right)  }\left(  a,b,c,d\right)  +M_{0}^{\left(  \rho_{0}\right)
}\left(  a,c,b,d\right)  =0$, which can be rewritten in the almost medial form
(\ref{r2}) as $\left(  -\mathbf{\rho}_{0}\left(  a^{\prime},c^{\prime
},b^{\prime},d^{\prime}\right)  \right)  M_{0}^{\left(  \rho_{0}\right)
}\left(  a,b,c,d\right)  =M_{0}^{\left(  \rho_{0}\right)  }\left(
a,c,b,d\right)  $. From (\ref{r01}) we get%
\begin{equation}
\left(  -\mathbf{\rho}_{0}^{-1}\left(  a^{\prime},b^{\prime},c^{\prime
},d^{\prime}\right)  \right)  M_{0}^{\left(  \rho_{0}\right)  }\left(
a,b,c,d\right)  =M_{0}^{\left(  \rho_{0}\right)  }\left(  a,c,b,d\right)  .
\end{equation}

\end{proof}

Let us \textquotedblleft deform\textquotedblright\ (\ref{r2}) again
successively by introducing further \textquotedblleft
deforming\textquotedblright\ functions $M_{k}$ and higher level mediality
factors $\mathbf{\rho}_{k}:G\times G\times G\times G\rightarrow\Bbbk$ in the
following way.

\begin{definition}
The $k$\textit{-level mediality factor} $\mathbf{\rho}_{k}\left(  a^{\prime
},b^{\prime},c^{\prime},d^{\prime}\right)  $ is defined by the following
\textquotedblleft difference-like\textquotedblright\ equations%
\begin{align}
\mathbf{\rho}_{1}\left(  a^{\prime},b^{\prime},c^{\prime},d^{\prime}\right)
M_{0}^{\left(  \rho_{0}\right)  }\left(  a,b,c,d\right)   &  =M_{0}^{\left(
\rho_{0}\right)  }\left(  a,c,b,d\right)  +M_{1}^{\left(  \rho_{0},\rho
_{1}\right)  }\left(  a,b,c,d\right)  ,\label{m1}\\
\mathbf{\rho}_{2}\left(  a^{\prime},b^{\prime},c^{\prime},d^{\prime}\right)
M_{1}^{\left(  \rho_{0}\right)  }\left(  a,b,c,d\right)   &  =M_{1}^{\left(
\rho_{0}\right)  }\left(  a,c,b,d\right)  +M_{2}^{\left(  \rho_{0},\rho
_{1},\rho_{2}\right)  }\left(  a,b,c,d\right)  ,\\
&  \vdots\nonumber\\
\mathbf{\rho}_{k}\left(  a^{\prime},b^{\prime},c^{\prime},d^{\prime}\right)
M_{k-1}^{\left(  \rho_{0},\rho_{1},\ldots,\rho_{k-1}\right)  }\left(
a,b,c,d\right)   &  =M_{k-1}^{\left(  \rho_{0},\rho_{1},\ldots,\rho
_{k-1}\right)  }\left(  a,c,b,d\right)  +M_{k}^{\left(  \rho_{0},\rho
_{1},\ldots,\rho_{k}\right)  }\left(  a,b,c,d\right)  ,\label{mk}\\
\forall a,b,c,d  &  \in A,\ \ \ a^{\prime},b^{\prime},c^{\prime},d^{\prime}\in
G.\nonumber
\end{align}

\end{definition}

\begin{definition}
$k$\textit{-level almost mediality} is defined by the vanishing of the last
\textquotedblleft deforming\textquotedblright\ medial function%
\begin{equation}
M_{k}^{\left(  \rho_{0},\rho_{1},\ldots,\rho_{k}\right)  }\left(
a,b,c,d\right)  =0,\ \ \ \forall a,b,c,d\in A,
\end{equation}
and can be expressed in a form analogous to (\ref{e0}) and (\ref{r2})%
\begin{equation}
\mathbf{\rho}_{k}\left(  a^{\prime},b^{\prime},c^{\prime},d^{\prime}\right)
M_{k-1}^{\left(  \rho_{0},\rho_{1},\ldots,\rho_{k-1}\right)  }\left(
a,b,c,d\right)  =M_{k-1}^{\left(  \rho_{0},\rho_{1},\ldots,\rho_{k-1}\right)
}\left(  a,c,b,d\right)  .
\end{equation}

\end{definition}

\begin{proposition}
The higher level \textquotedblleft deforming\textquotedblright\ functions
$M_{i}^{\left(  \rho_{0},\rho_{1},\ldots,\rho_{i}\right)  }\left(
a,b,c,d\right)  $, $i=1,\ldots,k$ can be expressed through $M_{0}^{\left(
\rho_{0}\right)  }\left(  a,b,c,d\right)  $ from (\ref{M0}) multiplied by a
combination of the lower level mediality factors $\mathbf{\rho}_{i}\left(
a^{\prime},b^{\prime},c^{\prime},d^{\prime}\right)  $, $i=1,\ldots,k$.
\end{proposition}

\begin{proof}
It follows from the equations (\ref{m1})--(\ref{mk}).
\end{proof}

\section{\textsc{Medial n-ary algebras}}

We now extend the concept of almost mediality from binary to polyadic
($n$-ary) algebras in the unique way which uses the construction from the
previous section.

Let $\mathcal{A}^{\left(  n\right)  }=\left\langle A\mid\mu_{n},\nu
_{2}\right\rangle $ be an associative $n$-ary algebra (with $n$-ary linear
multiplication $A^{\otimes n}\rightarrow A$) over a field $\Bbbk$ with
(possible) polyadic unit $e$ (then $\mathcal{A}^{\left(  n\right)  }$ a unital
$\Bbbk$-algebra) defined by $\mu_{n}\left[  e^{n-1},a\right]  =a$, $\forall
a\in A$ (where $a$ can be on any place) and (binary) zero $z\in A$. We
restrict ourselves (as in \cite{mil/vin,goz/goz/rem}) by the binary addition
$\nu_{2}:A\otimes A\rightarrow A$ which is denoted by $\left(  +\right)  $
(for more general cases, see \cite{dup2019}). Now \textit{polyadic (total)
associativity} \cite{goz/goz/rem} can be defined as a kind of invariance
\cite{dup2018a}%
\begin{equation}
\mu_{n}\left[  \mathbf{a},\mu_{n}\left[  \mathbf{b}_{\left(  n\right)
}\right]  ,\mathbf{c}\right]  =invariant, \label{mass}%
\end{equation}
where $\mathbf{a,c}$ are (\textit{linear}) \textit{polyads} (sequences of
elements from $A$) of the necessary length \cite{pos}, $\mathbf{b}_{\left(
n\right)  }$ is a polyad of the length $n$, and the internal multiplication
can be on any place. To describe the mediality for arbitrary arity $n$ we need
the following matrix generalization of polyads (as was implicitly used in
\cite{dup2018a,dup2019}).

\begin{definition}
\label{def-matr}A \textit{matrix (}$n$-ary) \textit{polyad} $\mathbf{\hat{A}%
}_{\left(  n^{2}\right)  }\equiv\mathbf{\hat{A}}_{\left(  n\times n\right)  }$
of size $n\times n$ is the sequence of $n^{2}$ elements $\mathbf{\hat{A}%
}_{\left(  n\times n\right)  }=\left(  a_{ij}\right)  \in A^{\otimes n^{2}}$,
$i,j=1,\ldots,n$, and their product $\mathrm{A}_{n^{2}}^{\left(  \mu\right)
}:A^{\otimes n^{2}}\rightarrow A$ contains $n+1$ of $n$-ary multiplications
$\mu_{n}$, which can be written as (we use hat for matrices of arguments, even
informally)%
\begin{equation}
\mathrm{A}_{n^{2}}^{\left(  \mu\right)  }\equiv\left(  \mu_{n}\right)
^{\circ\left(  n+1\right)  }\left[  \mathbf{\hat{A}}_{\left(  n^{2}\right)
}\right]  =\mu_{n}\left[
\begin{array}
[c]{c}%
\mu_{n}\left[  a_{11},a_{12},\ldots,a_{1n}\right]  ,\\
\mu_{n}\left[  a_{21},a_{22},\ldots,a_{2n}\right]  ,\\
\vdots\\
\mu_{n}\left[  a_{n1},a_{n2},\ldots,a_{nn}\right]
\end{array}
\right]  \in A \label{am}%
\end{equation}
due to the total associativity (\ref{mass}) (by \textquotedblleft omitting
brackets\textquotedblright).
\end{definition}

This construction is the stack reshape of a matrix or row-major order of an array.

\begin{example}
In terms of matrix polyads the (binary) mediality property (\ref{mm}) becomes%
\begin{align}
\left(  \mu_{2}\right)  ^{\circ3}\left[  \mathbf{\hat{A}}_{\left(  4\right)
}\right]   &  =\left(  \mu_{2}\right)  ^{\circ3}\left[  \mathbf{\hat{A}%
}_{\left(  4\right)  }^{T}\right]  ,\ \ \ or\ \ \ \mathrm{A}_{4}^{\left(
\mu\right)  }=\mathrm{A}_{4}^{T\left(  \mu\right)  }\label{mam}\\
\mathbf{\hat{A}}_{\left(  4\right)  }  &  =\left(
\begin{array}
[c]{cc}%
a_{11} & a_{12}\\
a_{21} & a_{22}%
\end{array}
\right)  \Rightarrow\left(  a_{11},a_{12},a_{21},a_{22}\right)  \in
A^{\otimes4},
\end{align}
where $\mathbf{\hat{A}}_{\left(  4\right)  }^{T}$ is the transposed polyad
matrix representing the sequence $\left(  a_{11},a_{21},a_{12},a_{22}\right)
\in A^{\otimes4}$, $\mathrm{A}_{4}^{\left(  \mu\right)  }=\left(  \left(
a_{11}\cdot a_{12}\right)  \cdot\left(  a_{21}\cdot a_{22}\right)  \right)
\in A$ and $\mathrm{A}_{4}^{T\left(  \mu\right)  }=\left(  \left(  a_{11}\cdot
a_{21}\right)  \cdot\left(  a_{12}\cdot a_{22}\right)  \right)  \in A$ with
$\left(  \cdot\right)  \equiv\mu_{2}$.
\end{example}

\begin{definition}
A \textit{polyadic (}$n$\textit{-ary) mediality} property is defined by the
relation%
\begin{align}
\left(  \mu_{n}\right)  ^{\circ n+1}\left[  \mathbf{\hat{A}}_{\left(
n^{2}\right)  }\right]   &  =\left(  \mu_{n}\right)  ^{\circ n+1}\left[
\mathbf{\hat{A}}_{\left(  n^{2}\right)  }^{T}\right]
,\ \ \ or\ \ \ \mathrm{A}_{n^{2}}^{\left(  \mu\right)  }=\mathrm{A}_{n^{2}%
}^{T\left(  \mu\right)  },\label{mna}\\
\mathbf{\hat{A}}_{\left(  n^{2}\right)  }  &  =\left(  a_{ij}\right)  \in
A^{\otimes n^{2}}. \label{an2}%
\end{align}

\end{definition}

\begin{definition}
A \textit{polyadic medial twist map} $\chi_{medial}^{\left(  n^{2}\right)  }$
is defined on the matrix polyads as \cite{dup2018d}%
\begin{equation}
\mathbf{\hat{A}}_{\left(  n^{2}\right)  }\overset{\chi_{medial}^{\left(
n^{2}\right)  }}{\mapsto}\mathbf{\hat{A}}_{\left(  n^{2}\right)  }^{T}.
\label{an}%
\end{equation}

\end{definition}

\begin{definition}
A $n$-ary algebra $\mathcal{A}^{\left(  n\right)  }$ is called \textit{medial}%
, if it satisfies the $n$-ary mediality property (\ref{mna}) for all
$a_{ij}\in A$.
\end{definition}

It follows from (\ref{mam}), that \textsf{not all} medial binary algebras are abelian.

\begin{corollary}
If a binary medial algebra $\mathcal{A}^{\left(  2\right)  }$ is cancellative,
it is abelian.
\end{corollary}

\begin{assertion}
If a $n$-ary medial algebra $\mathcal{A}^{\left(  n\right)  }$ is
cancellative, each matrix polyad $\mathbf{\hat{A}}_{\left(  n^{2}\right)  }$
satisfies $n^{2}-2$ commutativity-like relations.
\end{assertion}

\section{\textsc{Almost medial n-ary graded algebras}}

The gradation for associative $n$-ary algebras was considered in
\cite{mil/vin,gne95}. Here we introduce a weaker version of gradation, because
we need to define the grading twist not for $n$-ary multiplication, i. e. the
polyads of the length $n$, but only for the matrix polyads (\ref{an2}) of the
length $n^{2}$ (for the binary case, see (\ref{m4a})).

\begin{definition}
An associative $n$-ary algebra $\mathcal{A}^{\left(  n\right)  }$ is called a
\textit{higher graded }$n$\textit{-ary algebra} over $\Bbbk$, if the algebra
multiplication of $n^{2}$ elements respects the gradation i.e.%
\begin{equation}
\left(  \mu_{n}\right)  ^{\circ n+1}\left[  A_{\left(  a_{ij}^{\prime}\right)
}\right]  \equiv\mu_{n}\left[
\begin{array}
[c]{c}%
\mu_{n}\left[  A_{a_{11}^{\prime}},A_{a_{12}^{\prime}},\ldots,A_{a_{1n}%
^{\prime}}\right]  ,\\
\mu_{n}\left[  A_{a_{21}^{\prime}},A_{a_{22}^{\prime}},\ldots,A_{a_{21}%
^{\prime}}\right]  ,\\
\vdots\\
\mu_{n}\left[  A_{a_{n1}^{\prime}},A_{a_{n2}^{\prime}},\ldots,A_{a_{nn}%
^{\prime}}\right]
\end{array}
\right]  \subseteq A_{a_{11}^{\prime}+\ldots+a_{nn}^{\prime}},\ \forall
a_{ij}^{\prime}\in G,\ \ i,j=1,\ldots,n
\end{equation}
where equality corresponds to \textit{ strong higher gradation}.
\end{definition}

Let us define the \textit{higher twisting function} (\textit{higher twist
factor}) for $n^{2}$ elements $\mathbf{\tau}_{n^{2}}:G^{\times n^{2}%
}\rightarrow\Bbbk$ by using matrix polyads (for $n=2$ see (\ref{m4})).

\begin{definition}
A\textit{ }$n$\textit{-ary higher graded twisted product} $\mathbf{\mu}%
_{n^{2}}^{\left(  \tau\right)  }$ is defined for homogeneous elements by%
\begin{equation}
\mathbf{\mu}_{n^{2}}^{\left(  \tau\right)  }\left[  \mathbf{\hat{A}}_{\left(
n^{2}\right)  }\right]  =\mathbf{\tau}_{n^{2}}\left(  \mathbf{\hat{A}%
}_{\left(  n^{2}\right)  }^{\prime}\right)  \mathrm{A}_{n^{2}}^{\left(
\mu\right)  },\ \ \ a_{ij}\in A;\ \ a_{ij}^{\prime}\in G,\ \ i,j=1,\ldots,n,
\end{equation}
where $\mathbf{\hat{A}}_{\left(  n^{2}\right)  }=\left(  a_{ij}\right)  \in
A^{\otimes n^{2}}$ is the matrix polyad of elements (\ref{an2}), and
$\mathbf{\hat{A}}_{\left(  n^{2}\right)  }^{\prime}=\left(  a_{ij}^{\prime
}\right)  \in G^{\otimes n^{2}}$ is the matrix polyad of their gradings .
\end{definition}

A medial analog of $n$-ary almost mediality can be introduced in a way
analogous to the binary case (\ref{r2}).

\begin{definition}
If the higher twisted product coincides with the medially symmetric product
(see (\ref{an})) for all $a_{ij}\in A$, we call the twisting function a
$0$-\textit{level }$n$-\textit{ary mediality factor }$\mathbf{\tau}_{n^{2}%
}\mapsto\mathbf{\rho}_{0}^{\left(  n^{2}\right)  }:G^{\times n^{2}}%
\rightarrow\Bbbk^{\times}$, such that (cf. (\ref{e0}))%
\begin{align}
\mathbf{\mu}_{n^{2}}^{\left(  \rho_{0}\right)  }\left[  \mathbf{\hat{A}%
}_{\left(  n^{2}\right)  }\right]   &  =\mathrm{A}_{n^{2}}^{T\left(
\mu\right)  },\ \ \text{or}\label{rn1}\\
\mathbf{\rho}_{0}^{\left(  n^{2}\right)  }\left(  \mathbf{\hat{A}}_{\left(
n^{2}\right)  }^{\prime}\right)  \mathrm{A}_{n^{2}}^{\left(  \mu\right)  }  &
=\mathrm{A}_{n^{2}}^{T\left(  \mu\right)  },\ \ \ a_{ij}\in A;\ \ a_{ij}%
^{\prime}\in G,\ \ i,j=1,\ldots,n. \label{rn2}%
\end{align}

\end{definition}

It follows from (\ref{rn2}) that the normalization condition for the $n$-ary
mediality factor is%
\begin{equation}
\mathbf{\rho}_{0}^{\left(  n^{2}\right)  }\left(  \overset{n^{2}}%
{\overbrace{a^{\prime},\ldots,a^{\prime}}}\right)  =\mathtt{1},\ \ \ \forall
a^{\prime}\in G.
\end{equation}

\begin{assertion}
The $0$-level $n$-ary mediality factor $\rho_{0}^{\left(  n^{2}\right)  }$
satisfies%
\begin{equation}
\mathbf{\rho}_{0}^{\left(  n^{2}\right)  }\left(  \mathbf{\hat{A}}_{\left(
n^{2}\right)  }^{\prime}\right)  \mathbf{\rho}_{0}^{\left(  n^{2}\right)
}\left(  \left(  \mathbf{\hat{A}}_{\left(  n^{2}\right)  }^{\prime}\right)
^{T}\right)  =\mathtt{1}. \label{rr}%
\end{equation}

\end{assertion}

\begin{proof}
It follows from (\ref{rn2}) and its transpose together with the relation
$\left(  B^{T}\right)  ^{T}=B$ for any matrix over $\Bbbk$.
\end{proof}

\begin{definition}
An $n$-ary algebra for which the higher twisted product coincides with the
medially symmetric product (\ref{rn2}), is called a $0$\textit{-level almost
medial} ($\mathbf{\rho}_{0}$-commutative) $n$-ary algebra $\mathcal{A}%
_{n}^{\left(  \rho_{0}\right)  }$.
\end{definition}

Recall \cite{bourbaki98}, that a tensor product of binary algebras can be
naturally endowed with a $\varepsilon_{0}$-graded structure in the following
way (in our notation). Let $\mathcal{A}_{2}^{\left(  \varepsilon_{0}\right)
}=\left\langle A\mid\mu_{2}^{\left(  a\right)  }\right\rangle $ and
$\mathcal{B}_{2}^{\left(  \varepsilon_{0}\right)  }=\left\langle A\mid\mu
_{2}^{\left(  b\right)  }\right\rangle $ be binary graded algebras with the
multiplications $\mu_{2}^{\left(  a\right)  }\equiv\left(  \cdot_{a}\right)  $
and $\mu_{2}^{\left(  b\right)  }\equiv\left(  \cdot_{b}\right)  $ and the
same commutation factor $\varepsilon_{0}$ (see (\ref{e0})), that is the same
$G$-graded structure. Consider the tensor product $\mathcal{A}_{2}^{\left(
\varepsilon_{0}\right)  }\otimes\mathcal{B}_{2}^{\left(  \varepsilon
_{0}\right)  }$, and introduce the \textit{total }$\varepsilon_{0}%
$\textit{-graded multiplication} $\left(  \mathcal{A}_{2}^{\left(
\varepsilon_{0}\right)  }\otimes\mathcal{B}_{2}^{\left(  \varepsilon
_{0}\right)  }\right)  \star^{\left(  \varepsilon_{0}\right)  }\left(
\mathcal{A}_{2}^{\left(  \varepsilon_{0}\right)  }\otimes\mathcal{B}%
_{2}^{\left(  \varepsilon_{0}\right)  }\right)  \longrightarrow\mathcal{A}%
_{2}^{\left(  \varepsilon_{0}\right)  }\otimes\mathcal{B}_{2}^{\left(
\varepsilon_{0}\right)  }$ defined by the deformation (cf. (\ref{mt}))%
\begin{equation}
\varepsilon_{0}\left(  b_{1}^{\prime},a_{2}^{\prime}\right)  \left(
a_{1}\otimes b_{1}\right)  \star^{\left(  \varepsilon_{0}\right)  }\left(
a_{2}\otimes b_{2}\right)  ,\ b_{1}^{\prime},a_{2}^{\prime}\in G,a_{i}\in
A,b_{i}\in B,i=1,2. \label{abe}%
\end{equation}

\begin{proposition}
If the $\varepsilon_{0}$-graded multiplication (\ref{abe}) satisfies (cf.
(\ref{e0}))
\begin{equation}
\varepsilon_{0}\left(  b_{1}^{\prime},a_{2}^{\prime}\right)  \left(
a_{1}\otimes b_{1}\right)  \star^{\left(  \varepsilon_{0}\right)  }\left(
a_{2}\otimes b_{2}\right)  =\left(  a_{1}\cdot_{a}a_{2}\right)  \otimes\left(
b_{1}\cdot_{b}b_{2}\right)  , \label{ab1}%
\end{equation}
then $\mathcal{A}_{2}^{\left(  \varepsilon_{0}\right)  }\otimes\mathcal{B}%
_{2}^{\left(  \varepsilon_{0}\right)  }$ is a $\varepsilon_{0}$-graded
commutative algebra.
\end{proposition}

\begin{proposition}
If $\mathcal{A}_{2}^{\left(  \varepsilon_{0}\right)  }$ and $\mathcal{B}%
_{2}^{\left(  \varepsilon_{0}\right)  }$ are associative, then $\mathcal{A}%
_{2}^{\left(  \varepsilon_{0}\right)  }\otimes\mathcal{B}_{2}^{\left(
\varepsilon_{0}\right)  }$ is also associative.
\end{proposition}

\begin{proof}
This follows from (\ref{abe}), (\ref{ab1}) and the properties of the
commutation factor $\varepsilon_{0}$ (\ref{e02})--(\ref{e03}).
\end{proof}

In the matrix form (\ref{ab1}) becomes (with $\star^{\left(  \varepsilon
_{0}\right)  }\equiv\mu_{2}^{\star\left(  \varepsilon_{0}\right)  }$)
\begin{equation}
\varepsilon_{0}\left(  b_{1}^{\prime},a_{2}^{\prime}\right)  \mu_{2}%
^{\star\left(  \varepsilon_{0}\right)  }\left[
\begin{array}
[c]{c}%
\mu_{2}^{\otimes}\left[  a_{1},b_{1}\right] \\
\mu_{2}^{\otimes}\left[  a_{2},b_{2}\right]
\end{array}
\right]  =\mu_{2}^{\otimes}\left[
\begin{array}
[c]{c}%
\mu_{2}^{\left(  a\right)  }\left[  a_{1},a_{2}\right] \\
\mu_{2}^{\left(  b\right)  }\left[  b_{1},b_{2}\right]
\end{array}
\right]  , \label{em}%
\end{equation}
where $\mu_{2}^{\otimes}$ is the standard binary tensor product. For numerous
generalizations (including braidings), see, e.g., \cite{pen/pan/oys}, and
refs. therein.

Now we can extend (\ref{em}) to almost medial algebras.

\begin{definition}
Let $\mathcal{A}_{2}^{\left(  \rho_{0}\right)  }$ and $\mathcal{B}%
_{2}^{\left(  \rho_{0}\right)  }$ be two binary medial algebras with the same
mediality factor $\mathbf{\rho}_{0}$. The total $\mathbf{\rho}_{0}$-mediality
graded multiplication $\mu_{2}^{\star\left(  \rho_{0}\right)  }:\left(
\mathcal{A}_{2}^{\left(  \rho_{0}\right)  }\otimes\mathcal{B}_{2}^{\left(
\rho_{0}\right)  }\right)  \star^{\left(  \rho_{0}\right)  }\left(
\mathcal{A}_{2}^{\left(  \rho_{0}\right)  }\otimes\mathcal{B}_{2}^{\left(
\rho_{0}\right)  }\right)  \longrightarrow\mathcal{A}_{2}^{\left(  \rho
_{0}\right)  }\otimes\mathcal{B}_{2}^{\left(  \rho_{0}\right)  }$ is defined
by the mediality deformation (cf. (\ref{mt}))%
\begin{equation}
\mathbf{\rho}_{0}\left(
\begin{array}
[c]{cc}%
a_{1}^{\prime} & b_{1}^{\prime}\\
a_{2}^{\prime} & b_{2}^{\prime}%
\end{array}
\right)  \left(  a_{1}\otimes b_{1}\right)  \star^{\left(  \rho_{0}\right)
}\left(  a_{2}\otimes b_{2}\right)  ,\ b_{1}^{\prime},a_{2}^{\prime}\in
G,a_{i}\in A,b_{i}\in B,i=1,2. \label{r0}%
\end{equation}

\end{definition}

\begin{proposition}
If the $\mathbf{\rho}_{0}$-graded multiplication (\ref{r0}) satisfies (cf.
(\ref{e0}))
\begin{equation}
\mathbf{\rho}_{0}\left(
\begin{array}
[c]{cc}%
a_{1}^{\prime} & b_{1}^{\prime}\\
a_{2}^{\prime} & b_{2}^{\prime}%
\end{array}
\right)  \mu_{2}^{\star\left(  \rho_{0}\right)  }\left[
\begin{array}
[c]{c}%
\mu_{2}^{\otimes}\left[  a_{1},b_{1}\right] \\
\mu_{2}^{\otimes}\left[  a_{2},b_{2}\right]
\end{array}
\right]  =\mu_{2}^{\otimes}\left[
\begin{array}
[c]{c}%
\mu_{2}^{\left(  a\right)  }\left[  a_{1},a_{2}\right] \\
\mu_{2}^{\left(  b\right)  }\left[  b_{1},b_{2}\right]
\end{array}
\right]  , \label{rab}%
\end{equation}
then $\mathcal{A}_{2}^{\left(  \rho_{0}\right)  }\otimes\mathcal{B}%
_{2}^{\left(  \rho_{0}\right)  }$ is a $\mathbf{\rho}_{0}$-graded binary
(almost medial) algebra.
\end{proposition}

Using the matrix form (\ref{rab}) one can generalize the $\mathbf{\rho}_{0}%
$-graded medial algebras to arbitrary arity.

Let $\mathcal{B}_{n}^{\left(  \rho_{0}\right)  ,1},\ldots,\mathcal{B}%
_{n}^{\left(  \rho_{0}\right)  ,n}$ be $n$ $\mathbf{\rho}_{0}$-graded (almost
medial) $n$-ary algebras ($\mathcal{B}_{n}^{\left(  \rho_{0}\right)
,i}=\left\langle B_{i}\mid\mu_{n}^{\left(  i\right)  }\right\rangle $) with
the same mediality factor $\mathbf{\rho}_{0}$ and the same graded structure.
Consider their tensor product $\mathcal{B}_{n}^{\left(  \rho_{0}\right)
,1}\otimes\ldots\otimes\mathcal{B}_{n}^{\left(  \rho_{0}\right)  ,n}$ and the
$\mathbf{\rho}_{0}$-graded $n$-ary multiplication $\mu_{n}^{\star\left(
\rho_{0}\right)  }$ on it.

\begin{proposition}
If the $\mathbf{\rho}_{0}$-graded $n$-ary multiplication $\mu_{n}%
^{\star\left(  \rho_{0}\right)  }$ satisfies (cf. (\ref{e0}))
\begin{align}
\mathbf{\rho}_{0}\left(
\begin{array}
[c]{ccc}%
b_{1}^{1\prime} & \ldots & b_{1}^{n\prime}\\
\vdots & \ldots & \vdots\\
b_{n}^{1\prime} & \ldots & b_{n}^{n\prime}%
\end{array}
\right)  \mu_{n}^{\star\left(  \rho_{0}\right)  }\left[
\begin{array}
[c]{c}%
\mu_{n}^{\otimes}\left[  b_{1}^{1},\ldots,b_{1}^{n}\right] \\
\vdots\\
\mu_{n}^{\otimes}\left[  b_{n}^{1},\ldots,b_{n}^{n}\right]
\end{array}
\right]   &  =\mu_{n}^{\otimes}\left[
\begin{array}
[c]{c}%
\mu_{n}^{\left(  1\right)  }\left[  b_{1}^{1},\ldots,b_{n}^{1}\right] \\
\vdots\\
\mu_{n}^{\left(  n\right)  }\left[  b_{n}^{1},\ldots,b_{n}^{n}\right]
\end{array}
\right]  ,\label{rm}\\
b_{1}^{i},\ldots,b_{n}^{i}  &  \in\mathcal{A}_{n}^{\left(  \rho_{0}\right)
,i},\ b_{n}^{i\prime},\ldots,b_{n}^{i\prime}\in G,\ i=1,\ldots,n.
\end{align}
then the tensor product $\mathcal{B}_{n}^{\left(  \rho_{0}\right)  ,1}%
\otimes\ldots\otimes\mathcal{B}_{n}^{\left(  \rho_{0}\right)  ,n}$ is a
$\mathbf{\rho}_{0}$-graded $n$-ary (almost medial) algebra.
\end{proposition}

Symbolically, we can write this in the form, similar to the almost mediality
condition (\ref{rn2})%
\begin{equation}
\mathbf{\rho}_{0}\left(  \mathbf{\hat{B}}_{\left(  n^{2}\right)  }^{\prime
}\right)  \mu_{n}^{\star\left(  \rho_{0}\right)  }\circ\mu_{n^{2}}^{\otimes
}\left[  \mathbf{\hat{B}}_{\left(  n^{2}\right)  }\right]  =\mu_{n}^{\otimes
}\circ\mu_{n}^{\left(  1\right)  }\circ\ldots\circ\mu_{n}^{\left(  n\right)
}\left[  \mathbf{\hat{B}}_{\left(  n^{2}\right)  }^{T}\right]  , \label{rb}%
\end{equation}
where%
\begin{equation}
\mathbf{\hat{B}}_{\left(  n^{2}\right)  }^{\prime}=\left(
\begin{array}
[c]{ccc}%
b_{1}^{1\prime} & \ldots & b_{1}^{n\prime}\\
\vdots & \ldots & \vdots\\
b_{n}^{1\prime} & \ldots & b_{n}^{n\prime}%
\end{array}
\right)  ,\ \ \ \mathbf{\hat{B}}_{\left(  n^{2}\right)  }=\left(
\begin{array}
[c]{ccc}%
b_{1}^{1} & \ldots & b_{1}^{n}\\
\vdots & \ldots & \vdots\\
b_{n}^{1} & \ldots & b_{n}^{n}%
\end{array}
\right)
\end{equation}
and $\mathbf{\hat{B}}_{\left(  n^{2}\right)  }^{T}$ is its transpose.

\begin{example}
In the lowest non-binary example, for 3 ternary $\mathbf{\rho}_{0}$-graded
algebras $\mathcal{A}_{3}^{\left(  \rho_{0}\right)  }=\left\langle A\mid
\mu_{3}^{\left(  a\right)  }\right\rangle $, $\mathcal{B}_{3}^{\left(
\rho_{0}\right)  }=\left\langle B\mid\mu_{3}^{\left(  b\right)  }\right\rangle
$, $\mathcal{C}_{3}^{\left(  \rho_{0}\right)  }=\left\langle C\mid\mu
_{3}^{\left(  c\right)  }\right\rangle $, from (\ref{rm}) we have the ternary
multiplication $\mu_{3}^{\star\left(  \rho_{0}\right)  }$ for their ternary
tensor product $\mathcal{A}_{3}^{\left(  \rho_{0}\right)  }\otimes
\mathcal{B}_{3}^{\left(  \rho_{0}\right)  }\otimes\mathcal{C}_{3}^{\left(
\rho_{0}\right)  }$ given by%
\begin{equation}
\mathbf{\rho}_{0}\left(
\begin{array}
[c]{ccc}%
a_{1}^{\prime} & b_{1}^{\prime} & c_{1}^{\prime}\\
a_{2}^{\prime} & b_{2}^{\prime} & c_{2}^{\prime}\\
a_{3}^{\prime} & b_{3}^{\prime} & c_{3}^{\prime}%
\end{array}
\right)  \mu_{3}^{\star\left(  \rho_{0}\right)  }\left(
\begin{array}
[c]{c}%
\left(  a_{1}\otimes b_{1}\otimes c_{1}\right)  \\
\left(  a_{2}\otimes b_{2}\otimes c_{2}\right)  \\
\left(  a_{3}\otimes b_{2}\otimes c_{3}\right)
\end{array}
\right)  =\left(  \mu_{3}^{\left(  a\right)  }\left[  a_{1},a_{2}%
,a_{3}\right]  \otimes\mu_{3}^{\left(  a\right)  }\left[  b_{1},b_{2}%
,b_{3}\right]  \otimes\mu_{3}^{\left(  c\right)  }\left[  c_{1},c_{2}%
,c_{3}\right]  \right)  ,
\end{equation}
where $a_{i}\in A,\ \ b_{i}\in B,\ \ c_{i}\in C,\ \ a_{i}^{\prime}%
,b_{i}^{\prime},c_{i}^{\prime}\in G$, $i=1,2,3$.
\end{example}

\subsection{Higher level mediality $n^{2}$-ary brackets}

Binary almost mediality algebras for $n=2$ were considered in (\ref{r2}),
together with the tower of mediality factors (\ref{M0}), (\ref{m1}%
)--(\ref{mk}). Here we generalize this construction to any arity $n$ which can
be done using the matrix polyad construction.

First, we deform the almost mediality condition (\ref{rn2})%
\begin{equation}
\mathbf{\rho}_{0}^{\left(  n^{2}\right)  }\left(  \mathbf{\hat{A}}_{\left(
n^{2}\right)  }^{\prime}\right)  \mathrm{A}_{n^{2}}^{\left(  \mu\right)
}=\mathrm{A}_{n^{2}}^{T\left(  \mu\right)  }+\mathrm{M}_{0}^{\left(  \rho
_{0}\right)  }\left(  \mathbf{\hat{A}}_{\left(  n^{2}\right)  }\right)
,\ \ \ a_{ij}\in A;\ \ a_{ij}^{\prime}\in G,\ \ i,j=1,\ldots,n, \label{ra}%
\end{equation}
where $\mathrm{M}_{0}^{\left(  \rho_{0}\right)  }:A^{\otimes n^{2}}\rightarrow
A$ is the \textit{higher mediality }$n^{2}$\textit{-ary} \textit{bracket of
}$0$\textit{-level}. Consider $\mathrm{M}_{0}^{\left(  \rho_{0}\right)  }$ as
a new $n^{2}$-ary (bracket) multiplication%
\begin{equation}
\mathbf{\mu}_{n^{2}}^{\left(  \rho_{0},M_{0}\right)  }\left[  \mathbf{\hat{A}%
}_{\left(  n^{2}\right)  }\right]  :=\mathrm{M}_{0}^{\left(  \rho_{0}\right)
}\left(  \mathbf{\hat{A}}_{\left(  n^{2}\right)  }\right)  . \label{mm0}%
\end{equation}

\begin{definition}
A $n^{2}$-ary algebra%
\begin{equation}
\mathcal{A}_{n^{2}}^{\left(  \rho_{0},M_{0}\right)  }=\left\langle
A\mid\mathbf{\mu}_{n^{2}}^{\left(  \rho_{0},M_{0}\right)  }\right\rangle
\end{equation}
is called a $0$\textit{-level mediality bracket }$n^{2}$\textit{-ary algebra}.
\end{definition}

\begin{proposition}
The $n^{2}$-ary algebra $\mathcal{A}_{n^{2}}^{\left(  \rho_{0},M_{0}\right)
}$is almost medial with the mediality factor $\left(  -\left(  \mathbf{\rho
}_{0}^{\left(  n^{2}\right)  }\right)  ^{-1}\right)  $.
\end{proposition}

\begin{proof}
We multiply the definition (\ref{ra}) by $\mathbf{\rho}_{0}^{\left(
n^{2}\right)  }\left(  \left(  \mathbf{\hat{A}}_{\left(  n^{2}\right)
}^{\prime}\right)  ^{T}\right)  $ and use (\ref{rr}) to obtain%
\begin{equation}
\mathbf{\rho}_{0}^{\left(  n^{2}\right)  }\left(  \left(  \mathbf{\hat{A}%
}_{\left(  n^{2}\right)  }^{\prime}\right)  ^{T}\right)  \mathrm{M}%
_{0}^{\left(  \rho_{0}\right)  }\left(  \mathbf{\hat{A}}_{\left(
n^{2}\right)  }\right)  =\mathrm{A}_{n^{2}}^{\left(  \mu\right)
}-\mathbf{\rho}_{0}^{\left(  n^{2}\right)  }\left(  \left(  \mathbf{\hat{A}%
}_{\left(  n^{2}\right)  }^{\prime}\right)  ^{T}\right)  \mathrm{A}_{n^{2}%
}^{T\left(  \mu\right)  }.
\end{equation}
Taking into account that the r.h.s. here is exactly $-\mathrm{M}_{0}^{\left(
\rho_{0}\right)  }\left(  \mathbf{\hat{A}}_{\left(  n^{2}\right)  }%
^{T}\right)  $, we have%
\begin{equation}
-\mathbf{\rho}_{0}^{\left(  n^{2}\right)  }\left(  \left(  \mathbf{\hat{A}%
}_{\left(  n^{2}\right)  }^{\prime}\right)  ^{T}\right)  \mathrm{M}%
_{0}^{\left(  \rho_{0}\right)  }\left(  \mathbf{\hat{A}}_{\left(
n^{2}\right)  }\right)  =\mathrm{M}_{0}^{\left(  \rho_{0}\right)  }\left(
\mathbf{\hat{A}}_{\left(  n^{2}\right)  }^{T}\right)  ,
\end{equation}
and using (\ref{rr}) again, we get%
\begin{equation}
-\mathbf{\rho}_{0}^{\left(  n^{2}\right)  }\left(  \mathbf{\hat{A}}_{\left(
n^{2}\right)  }^{\prime}\right)  ^{-1}\mathrm{M}_{0}^{\left(  \rho_{0}\right)
}\left(  \mathbf{\hat{A}}_{\left(  n^{2}\right)  }\right)  =\mathrm{M}%
_{0}^{\left(  \rho_{0}\right)  }\left(  \mathbf{\hat{A}}_{\left(
n^{2}\right)  }^{T}\right)  ,
\end{equation}
which should be compared with (\ref{rn2}).
\end{proof}

Now we \textquotedblleft deform\textquotedblright\ (\ref{ra}) successively by
defining further $n^{2}$-ary brackets\ $\mathrm{M}_{k}$ and higher level
mediality factors $\mathbf{\rho}_{k}^{\left(  n^{2}\right)  }:G^{\times n^{2}%
}\rightarrow\Bbbk$ as follows.

\begin{definition}
The $k$\textit{-level mediality }$n^{2}$-ary \textit{brackets and factors are
defined by}%
\begin{align}
\mathbf{\rho}_{1}^{\left(  n^{2}\right)  }\left(  \mathbf{\hat{A}}_{\left(
n^{2}\right)  }^{\prime}\right)  \mathrm{M}_{0}^{\left(  \rho_{0}\right)
}\left(  \mathbf{\hat{A}}_{\left(  n^{2}\right)  }\right)   &  =\mathrm{M}%
_{0}^{\left(  \rho_{0}\right)  }\left(  \mathbf{\hat{A}}_{\left(
n^{2}\right)  }^{T}\right)  +\mathrm{M}_{1}^{\left(  \rho_{0},\rho_{1}\right)
}\left(  \mathbf{\hat{A}}_{\left(  n^{2}\right)  }\right)  ,\label{mm1}\\
\mathbf{\rho}_{2}^{\left(  n^{2}\right)  }\left(  \mathbf{\hat{A}}_{\left(
n^{2}\right)  }^{\prime}\right)  \mathrm{M}_{1}^{\left(  \rho_{0},\rho
_{1}\right)  }\left(  \mathbf{\hat{A}}_{\left(  n^{2}\right)  }\right)   &
=\mathrm{M}_{1}^{\left(  \rho_{0},\rho_{1}\right)  }\left(  \mathbf{\hat{A}%
}_{\left(  n^{2}\right)  }^{T}\right)  +\mathrm{M}_{2}^{\left(  \rho_{0}%
,\rho_{1},\rho_{2}\right)  }\left(  \mathbf{\hat{A}}_{\left(  n^{2}\right)
}\right)  ,\\
&  \vdots\nonumber\\
\mathbf{\rho}_{k}^{\left(  n^{2}\right)  }\left(  \mathbf{\hat{A}}_{\left(
n^{2}\right)  }^{\prime}\right)  \mathrm{M}_{k-1}^{\left(  \rho_{0},\rho
_{1},\ldots,\rho_{k-1}\right)  }\left(  \mathbf{\hat{A}}_{\left(
n^{2}\right)  }\right)   &  =\mathrm{M}_{k-1}^{\left(  \rho_{0},\rho
_{1},\ldots,\rho_{k-1}\right)  }\left(  \mathbf{\hat{A}}_{\left(
n^{2}\right)  }^{T}\right)  +\mathrm{M}_{k}^{\left(  \rho_{0},\rho_{1}%
,\ldots,\rho_{k}\right)  }\left(  \mathbf{\hat{A}}_{\left(  n^{2}\right)
}\right) \label{mmk}\\
\forall a_{ij}  &  \in A;\ \ a_{ij}^{\prime}\in G,\ \ i,j=1,\ldots,n.\nonumber
\end{align}

\end{definition}

\begin{definition}
$k$\textit{-level }$n^{2}$-ary \textit{almost mediality} is given by the
vanishing of the last \textquotedblleft deforming\textquotedblright\ medial
$n^{2}$-ary bracket%
\begin{equation}
\mathrm{M}_{k}^{\left(  \rho_{0},\rho_{1},\ldots,\rho_{k}\right)  }\left(
\mathbf{\hat{A}}_{\left(  n^{2}\right)  }\right)  =0,\ \ \ \forall a_{ij}\in
A,
\end{equation}
and has the form%
\begin{equation}
\mathbf{\rho}_{k}^{\left(  n^{2}\right)  }\left(  \mathbf{\hat{A}}_{\left(
n^{2}\right)  }^{\prime}\right)  \mathrm{M}_{k-1}^{\left(  \rho_{0},\rho
_{1},\ldots,\rho_{k-1}\right)  }\left(  \mathbf{\hat{A}}_{\left(
n^{2}\right)  }\right)  =\mathrm{M}_{k-1}^{\left(  \rho_{0},\rho_{1}%
,\ldots,\rho_{k-1}\right)  }\left(  \mathbf{\hat{A}}_{\left(  n^{2}\right)
}^{T}\right)  .
\end{equation}

\end{definition}

\begin{proposition}
The higher level \textquotedblleft deforming\textquotedblright%
\ functions\textit{ (}$n^{2}$-ary brackets) $\mathrm{M}_{i}^{\left(  \rho
_{0},\rho_{1},\ldots,\rho_{k}\right)  }\left(  \mathbf{\hat{A}}_{\left(
n^{2}\right)  }\right)  $, $i=1,\ldots,k$ can be expressed through
$\mathrm{M}_{0}^{\left(  \rho_{0}\right)  }\left(  \mathbf{\hat{A}}_{\left(
n^{2}\right)  }\right)  $ from (\ref{ra}) using a combination of the lower
level $n^{2}$-ary mediality factors $\mathbf{\rho}_{k}^{\left(  n^{2}\right)
}\left(  \mathbf{\hat{A}}_{\left(  n^{2}\right)  }^{\prime}\right)  $,
$i=1,\ldots,k$.
\end{proposition}

\begin{proof}
This follows from the equations (\ref{mm1})--(\ref{mmk}).
\end{proof}

\section{\textsc{Toyoda's theorem for almost medial algebras}}

The structure of the almost medial graded algebras (binary and $n$-ary) can be
established by searching for possible analogs of Toyoda's theorem (\ref{ma})
(see, \cite{bru44,mur41,toy41}) which is the main statement for medial
groupoids \cite{jez/kep} and quasigroups \cite{shcherbacov}. As Toyoda's
theorem connects medial algebras with abelian algebras, we can foresee that in
the same way the almost medial algebras can be connected with almost
commutative algebras.

First, let us consider almost medial graded binary algebras, as defined in
(\ref{r2})--(\ref{r03}).

\begin{theorem}
Let $\mathcal{A}_{2}^{\left(  \rho\right)  }=\left\langle A\mid\mu
_{2}\right\rangle $ be an almost medial ($\mathbf{\rho}$-commutative)
$G$-graded binary algebra, then there exists an almost commutative
($\varepsilon$-commutative $G$-graded binary algebra $\mathcal{\bar{A}}%
_{2}^{\left(  \varepsilon\right)  }=\left\langle A\mid\bar{\mu}_{2}%
\right\rangle $, two grading preserving automorphisms $\varphi_{1,2}$ and a
fixed element $h\in A$, such that (cf. (\ref{ma}))\footnote{We use the
multiplicative notation for the algebra $\mathcal{\bar{A}}_{2}^{\left(
\varepsilon\right)  }$, because it is non-commutative.}%
\begin{align}
\mu_{2}\left[  a,b\right]   &  =\bar{\mu}_{2}\left[  \bar{\mu}_{2}\left[
\varphi_{1}(a),\varphi_{2}\left(  b\right)  \right]  ,h\right]  \ \ \text{or
}\ a\cdot b=\varphi_{1}(a)\varphi_{2}\left(  b\right)  h,\label{mab}\\
\mathbf{\rho}\left(  a^{\prime},b^{\prime},c^{\prime},d^{\prime}\right)   &
=\varepsilon\left(  b^{\prime},c^{\prime}\right)  ,\forall a,b,c,d\in
A,\ \ \ a^{\prime},b^{\prime},c^{\prime},d^{\prime}\in G, \label{re}%
\end{align}
where we denote $\mu_{2}\equiv\left(  \cdot\right)  $ and $\bar{\mu}%
_{2}\left[  a,b\right]  \equiv ab$.
\end{theorem}

\begin{proof}
We use the \textquotedblleft linear\textquotedblright\ presentation
(\ref{mab}) for the product in $\mathcal{A}_{2}^{\left(  \rho\right)  }$ and
insert it into the condition of almost mediality (\ref{r2}) to obtain%
\begin{align}
\mathbf{\rho}\left(  a^{\prime},b^{\prime},c^{\prime},d^{\prime}\right)   &
\left(  a\cdot b\right)  \cdot\left(  c\cdot d\right)  =\left(  a\cdot
c\right)  \cdot\left(  b\cdot d\right)  \Rightarrow\nonumber\\
\mathbf{\rho}\left(  a^{\prime},b^{\prime},c^{\prime},d^{\prime}\right)   &
\varphi_{1}\left(  \varphi_{1}(a)\varphi_{2}\left(  b\right)  h\right)
\varphi_{2}\left(  \left(  \varphi_{1}(c)\varphi_{2}\left(  d\right)
h\right)  \right)  h\nonumber\\
&  =\varphi_{1}\left(  \varphi_{1}(a)\varphi_{2}\left(  c\right)  h\right)
\varphi_{2}\left(  \left(  \varphi_{1}(b)\varphi_{2}\left(  d\right)
h\right)  \right)  \Rightarrow\nonumber\\
\mathbf{\rho}\left(  a^{\prime},b^{\prime},c^{\prime},d^{\prime}\right)   &
\varphi_{1}\circ\varphi_{1}(a)\varphi_{1}\circ\varphi_{2}\left(  b\right)
\varphi_{1}\left(  h\right)  \varphi_{2}\circ\varphi_{1}(c)\varphi_{2}%
\circ\varphi_{2}\left(  d\right)  \varphi_{2}\left(  h\right)  h\nonumber\\
&  =\varphi_{1}\circ\varphi_{1}(a)\varphi_{1}\circ\varphi_{2}\left(  c\right)
\varphi_{1}\left(  h\right)  \varphi_{2}\circ\varphi_{1}(b)\varphi_{2}%
\circ\varphi_{2}\left(  d\right)  \varphi_{2}\left(  h\right)  h,
\end{align}
where $\left(  \circ\right)  $ is the composition of automorphisms. Using the
cancellativity of $\mathcal{\bar{A}}_{2}^{\left(  \varepsilon\right)  }$, we
get%
\begin{equation}
\mathbf{\rho}\left(  a^{\prime},b^{\prime},c^{\prime},d^{\prime}\right)
\varphi_{1}\circ\varphi_{2}\left(  b\right)  \varphi_{2}\circ\varphi
_{1}(c)=\varphi_{1}\circ\varphi_{2}\left(  c\right)  \varphi_{2}\circ
\varphi_{1}(b). \label{rf}%
\end{equation}

Because the automorphisms $\varphi_{1,2}$ preserve grading, after implementing
almost ($\varepsilon$-) commutativity (\ref{e0}), the r.h.s. of (\ref{rf})
becomes $\varepsilon\left(  b^{\prime},c^{\prime}\right)  \varphi_{2}%
\circ\varphi_{1}(b)\varphi_{1}\circ\varphi_{2}\left(  c\right)  $ which gives
(\ref{re}) for commuting automorphisms.
\end{proof}

The higher arity cases are more non-trivial, and very cumbersome. Therefore,
we restrict ourselves by the case $n=3$ only.

\begin{theorem}
Let $\mathcal{A}_{3}^{\left(  \rho\right)  }=\left\langle A\mid\mu_{3},\nu
_{2}\right\rangle $ be an almost medial ($\mathbf{\rho}$-commutative)
$G$-graded ternary algebra over a field $\Bbbk$. Then there exists an almost
commutative ($\varepsilon$-commutative $G$-graded binary algebra
$\mathcal{\bar{A}}_{2}^{\left(  \varepsilon\right)  }=\left\langle A\mid
\bar{\mu}_{2}\right\rangle $, three commuting grading preserving automorphisms
$\varphi_{1,2,3}$ and a fixed element $h\in A$, such that (cf. (\ref{ma}))%
\begin{align}
\mu_{3}\left[  a,b,c\right]   &  =\bar{\mu}_{2}\left[  \bar{\mu}_{2}\left[
\bar{\mu}_{2}\left[  \varphi_{1}(a),\varphi_{2}\left(  b\right)  \right]
,h\right]  ,h\right] \nonumber\\
&  \equiv\varphi_{1}(a)\varphi_{2}\left(  b\right)  \varphi_{3}\left(
c\right)  h,\ \forall a,b,c,h\in A\label{m3}\\
\mathbf{\rho}^{\left(  3^{2}\right)  }\left(  \mathbf{\hat{A}}_{\left(
3^{2}\right)  }^{\prime}\right)   &  =\varepsilon\left(  a_{12}^{\prime
},a_{31}^{\prime}\right)  \varepsilon\left(  a_{12}^{\prime},a_{21}^{\prime
}\right)  \varepsilon\left(  a_{13}^{\prime},a_{31}^{\prime}\right)
\varepsilon\left(  a_{13}^{\prime},a_{32}^{\prime}\right)  \varepsilon\left(
a_{23}^{\prime},a_{32}^{\prime}\right)  \varepsilon\left(  a_{23}^{\prime
},a_{31}^{\prime}\right)  ,\label{r3}\\
\mathbf{\hat{A}}_{\left(  3^{2}\right)  }^{\prime}  &  =\left(  a_{ij}%
^{\prime}\right)  ,\ \ \forall a_{ij}^{\prime}\in G,\ \ \ i,j=1,\ldots
,3,\nonumber
\end{align}
where we denote $\bar{\mu}_{2}\left[  a,b\right]  \equiv ab$.
\end{theorem}

\begin{proof}
Using the matrix form of ternary ($n=3$) almost regularity (\ref{rn2}) and
inserting there the ternary \textquotedblleft linear\textquotedblright%
\ presentation (\ref{m3}) we get (in matrix form), $\forall a_{ij}\in
A,\ \ \ i,j=1,\ldots,3,$%
\begin{align}
\mathbf{\rho}^{\left(  3^{2}\right)  }\left(  \mathbf{\hat{A}}_{\left(
3^{2}\right)  }^{\prime}\right)   &  \left(
\begin{array}
[c]{ccc}%
\varphi_{1}\circ\varphi_{1}\left(  a_{11}\right)  & \varphi_{1}\circ
\varphi_{2}\left(  a_{12}\right)  & \varphi_{1}\circ\varphi_{3}\left(
a_{13}\right) \\
\varphi_{2}\circ\varphi_{1}\left(  a_{21}\right)  & \varphi_{2}\circ
\varphi_{2}\left(  a_{22}\right)  & \varphi_{2}\circ\varphi_{3}\left(
a_{23}\right) \\
\varphi_{3}\circ\varphi_{1}\left(  a_{31}\right)  & \varphi_{3}\circ
\varphi_{2}\left(  a_{32}\right)  & \varphi_{3}\circ\varphi_{3}\left(
a_{33}\right)
\end{array}
\right) \\
&  =\left(
\begin{array}
[c]{ccc}%
\varphi_{1}\circ\varphi_{1}\left(  a_{11}\right)  & \varphi_{1}\circ
\varphi_{2}\left(  a_{21}\right)  & \varphi_{1}\circ\varphi_{3}\left(
a_{31}\right) \\
\varphi_{2}\circ\varphi_{1}\left(  a_{12}\right)  & \varphi_{2}\circ
\varphi_{2}\left(  a_{22}\right)  & \varphi_{2}\circ\varphi_{3}\left(
a_{32}\right) \\
\varphi_{3}\circ\varphi_{1}\left(  a_{13}\right)  & \varphi_{3}\circ
\varphi_{2}\left(  a_{23}\right)  & \varphi_{3}\circ\varphi_{3}\left(
a_{33}\right)
\end{array}
\right)  .
\end{align}
Applying the cancellativity of the binary algebra $\mathcal{\bar{A}}%
_{2}^{\left(  \varepsilon\right)  }$, we have%
\begin{align}
&  \mathbf{\rho}^{\left(  3^{2}\right)  }\left(  \mathbf{\hat{A}}_{\left(
3^{2}\right)  }^{\prime}\right)  \varphi_{1}\circ\varphi_{2}\left(
a_{12}\right)  \varphi_{1}\circ\varphi_{3}\left(  a_{13}\right)  \varphi
_{2}\circ\varphi_{1}\left(  a_{21}\right)  \varphi_{2}\circ\varphi_{3}\left(
a_{23}\right)  \varphi_{3}\circ\varphi_{1}\left(  a_{31}\right)  \varphi
_{3}\circ\varphi_{2}\left(  a_{32}\right) \nonumber\\
&  =\varphi_{1}\circ\varphi_{2}\left(  a_{21}\right)  \varphi_{1}\circ
\varphi_{3}\left(  a_{31}\right)  \varphi_{2}\circ\varphi_{1}\left(
a_{12}\right)  \varphi_{2}\circ\varphi_{3}\left(  a_{32}\right)  \varphi
_{3}\circ\varphi_{1}\left(  a_{13}\right)  \varphi_{3}\circ\varphi_{2}\left(
a_{23}\right)  . \label{ra3}%
\end{align}
Implementing almost ($\varepsilon$-) commutativity (\ref{e0}) on the r.h.s. of
(\ref{ra3}), we arrive (for pairwise commuting grading preserving
automorphisms $\varphi_{i}\circ\varphi_{j}=\varphi_{j}\circ\varphi
_{i},\ i,j=1,2,3$) at (\ref{r3}).
\end{proof}

\section{\textsc{Binary tensor categories}}

We now apply the above ideas to construct a special kind of \textit{categories
with multiplication} \cite{ben63,macl63} which appeared already in
\cite{tan39} and later on were called tensor categories and monoidal
categories (as they \textquotedblleft remind\textquotedblright\ us of the
structure of a monoid) \cite{maclane1}. For reviews, see, e.g.
\cite{cal/eti,mue2010}. The monoidal categories can be considered as the
\textit{categorification} \cite{bae/dol98} of a monoid object, and can be
treated as an instance of the \textit{microcosm principle}: \textquotedblleft%
\textsf{certain algebraic structures can be defined in any category equipped
with a categorified version of the same structure}\textquotedblright%
\ \cite{bae/dol98a}. We start from the definitions of categories
\cite{ada/her/str,borceaux1} and binary tensor categories \cite{maclane1} (in
our notation).

Let $\mathcal{C}$ be a category with the class of objects $\operatorname*{Ob}%
\mathcal{C}$ and morphisms $\operatorname*{Mor}\mathcal{C}$, such that the
arrow from the source $X_{1}$ to the target $X_{2}$ is defined by
$\operatorname*{Mor}\mathcal{C}\ni\mathsf{f}_{12}:X_{1}\rightarrow X_{2}$,
$X_{1,2}\in\operatorname*{Ob}\mathcal{C}$, and usually $\operatorname*{Hom}%
_{C}\left(  X_{1},X_{2}\right)  $ denotes all arrows which do not intersect.
If $\operatorname*{Ob}\mathcal{C}$ and $\operatorname*{Mor}\mathcal{C}$ are
sets, the category is \textit{small}. The composition $\left(  \circ\right)  $
of three morphisms, their associativity and the identity morphism
($\operatorname*{id}_{X}$) are defined in the standard way \cite{maclane1}.

If $\mathcal{C}$ and $\mathcal{C}^{\prime}$ are two categories, then a mapping
between them is called a \textit{covariant} \textit{functor} $\mathit{F}%
:\mathcal{C}\rightarrow\mathcal{C}^{\prime}$ which consists of two different
components: 1) the $X$-component is a mapping of objects $\mathit{F}%
_{\operatorname*{Ob}}:\operatorname*{Ob}\mathcal{C}\rightarrow
\operatorname*{Ob}\mathcal{C}^{\prime}$; 2) the $\mathsf{f}$-component is a
mapping of morphisms $\mathit{F}_{\operatorname*{Mor}}:\operatorname*{Mor}%
\mathcal{C}\rightarrow\operatorname*{Mor}\mathcal{C}^{\prime}$ such that
$\mathit{F}=\left\{  \mathit{F}_{\operatorname*{Ob}},\mathit{F}%
_{\operatorname*{Mor}}\right\}  $. A functor preserves the identity morphism
$\mathit{F}_{\operatorname*{Mor}}\left(  \operatorname*{id}_{X}\right)
=\operatorname*{id}_{\mathit{F}_{\operatorname*{Ob}}\left(  X\right)  }$ and
the composition of morphisms $\mathit{F}_{\operatorname*{Mor}}\left(
\mathsf{f}_{23}\circ\mathsf{f}_{12}\right)  =\mathit{F}_{\operatorname*{Mor}%
}\left(  \mathsf{f}_{23}\right)  \circ^{\prime}\mathit{F}_{\operatorname*{Mor}%
}\left(  \mathsf{f}_{12}\right)  $ ($=\mathit{F}_{\operatorname*{Mor}}\left(
\mathsf{f}_{12}\right)  \circ^{\prime}\mathit{F}_{\operatorname*{Mor}}\left(
\mathsf{f}_{23}\right)  $ for a \textit{contravariant functor}), where
$\left(  \circ^{\prime}\right)  $ is the composition in $\mathcal{C}^{\prime}$.

The (binary) \textit{product category} $\mathcal{C}\times\mathcal{C}^{\prime}$
consists of all pairs of objects $\left(  \operatorname*{Ob}\mathcal{C}%
,\operatorname*{Ob}\mathcal{C}^{\prime}\right)  $, morphisms $\left(
\operatorname*{Mor}\mathcal{C},\operatorname*{Mor}\mathcal{C}^{\prime}\right)
$ and identities $\left(  \operatorname*{id}_{X},\operatorname*{id}%
_{X^{\prime}}\right)  $, while the composition $\left(  \circ^{\prime\prime
}\right)  $ is made component-wise%
\begin{align}
&  \left(  \mathsf{f}_{23},\mathsf{f}_{23}^{\prime}\right)  \circ
^{\prime\prime}\left(  \mathsf{f}_{12},\mathsf{f}_{12}^{\prime}\right)
=\left(  \mathsf{f}_{23}\circ\mathsf{f}_{12},\mathsf{f}_{23}^{\prime}%
\circ^{\prime}\mathsf{f}_{12}^{\prime}\right)  ,\label{ff}\\
&  \mathsf{f}_{ij}:X_{i}\rightarrow X_{j},\forall X_{i}\in\operatorname*{Ob}%
\mathcal{C},\mathsf{f}_{ij}^{\prime}:X_{i}^{\prime}\rightarrow X_{j}^{\prime
},\forall X_{i}^{\prime}\in\operatorname*{Ob}\mathcal{C}^{\prime
},i,j=1,2,3,\nonumber
\end{align}
and by analogy this may be extended for more multipliers. A functor on a
binary product category is called a \textit{bifunctor} (\textit{multifunctor}%
). A functor consists of two components\footnote{Usually \cite{maclane1},
which are often denoted by the same letter, but for clarity we will
distinguish them, because their action, arguments and corresponding
commutative diagrams are different.} $\left\{  \mathit{F}_{\operatorname*{Ob}%
},\mathit{F}_{\operatorname*{Mor}}\right\}  $, and therefore a mapping between
two functors $\mathit{F}$ and $\mathit{G}$ should also be two-component
$\mathrm{T}^{FG}=\left\{  \mathrm{T}_{\operatorname*{Ob}}^{FG},\mathrm{T}%
_{\operatorname*{Mor}}^{FG}\right\}  $. Without other conditions
$\mathrm{T}^{FG}$ is called an \textit{infra-natural transformation} from
$\mathit{F}$ to $\mathit{G}$. A \textit{natural transformation} (denoted by
the double arrow $\mathrm{T}^{FG}:\mathit{F}\Rightarrow\mathit{G}$) is defined
by the consistency condition of the above mappings in $\mathcal{C}^{\prime}$%
\begin{equation}
\mathrm{T}_{\operatorname*{Ob}}^{FG}\circ^{\prime}\mathit{F}%
_{\operatorname*{Mor}}=\mathit{G}_{\operatorname*{Mor}}\circ^{\prime
}\mathrm{T}_{\operatorname*{Ob}}^{FG}. \label{tf}%
\end{equation}

Application to objects gives the following commutative diagram for the natural
transformations (bifunctoriality)\begin{equation}
\xymatrix@R+9mm@C+6mm@L+2mm{ 
\mathit{F}_{\operatorname*{Ob}}\left( X_{1}\right) \equiv X_{1}^{\prime F} 
\ar[d]_{\mathrm{T}_{\operatorname*{Ob}}^{FG}\left( X_{1}\right) } 
\ar[rrr]^{\mathit{F}_{\operatorname*{Mor}}\left( \mathsf{f}\right) \equiv \mathsf{f}^{\prime F}} 
\ar[drrr]^(0.6){\mathrm{T}_{\operatorname*{Mor}}^{FG}\left( \mathsf{f}\right)}
& && \mathit{F}_{\operatorname*{Ob}}\left( X_{2}\right) \equiv X_{2}^{\prime F} 
\ar[d]^{\mathrm{T}_{\operatorname*{Ob}}^{FG}\left( X_{2}\right)} 
\\ 
\mathit{G}_{\operatorname*{Ob}}\left( X_{1}\right) \equiv X_{1}^{\prime G} 
\ar[rrr]_{\mathit{G}_{\operatorname*{Mor}}\left( \mathsf{f}\right) \equiv \mathsf{f}^{\prime G} }
&&& \mathit{G}_{\operatorname*{Ob}}\left( X_{2}\right) \equiv X_{2}^{\prime G}}
\label{x-com}
\end{equation}which is the consistency of
the objects in $\mathcal{C}^{\prime}$ transformed by $\mathit{F}$ and
$\mathit{G}$. The the diagonal in (\ref{x-com}) may also be interpreted as the
action of the natural transformation on a morphism $\mathrm{T}%
_{\operatorname*{Mor}}^{FG}\left(  \mathsf{f}\right)  :\mathit{F}%
_{\operatorname*{Ob}}\left(  X_{1}\right)  \rightarrow\mathit{G}%
_{\operatorname*{Ob}}\left(  X_{2}\right)  $, $\mathsf{f}:X_{1}\rightarrow
X_{2}$, $\mathsf{f}\in\operatorname*{Mor}\mathcal{C}$, $X_{1},X_{2}%
\in\operatorname*{Ob}\mathcal{C}$ , such that%
\begin{equation}
\mathrm{T}_{\operatorname*{Mor}}^{FG}\left(  \mathsf{f}\right)  =\mathrm{T}%
_{\operatorname*{Ob}}^{FG}\left(  X_{2}\right)  \circ^{\prime}\mathit{F}%
_{\operatorname*{Mor}}\left(  \mathsf{f}\right)  =\mathit{G}%
_{\operatorname*{Mor}}\left(  \mathsf{f}\right)  \circ^{\prime}\mathrm{T}%
_{\operatorname*{Ob}}^{FG}\left(  X_{1}\right)  ,
\end{equation}
where the second equality holds valid due to the naturality (\ref{tf}).

In a concise form the natural transformations are described by the commutative diagram%

\begin{equation}
\xymatrix{ \mathcal{C} \ar@/^1pc/[rrr]^{\mathit{F}}="0" \ar@/_1pc/[rrr]_{\mathit{G}}="1" &&& \mathcal{C}^{\prime} \\ \ar@{=>}"0"+<0ex,-2ex> ;"1"+<0ex,+2ex>^{\mathrm{T}^{FG}} }
\end{equation}

For a category $\mathcal{C}$, the \textit{identity functor} $\mathit{Id}%
_{\mathcal{C}}=\left(  \mathit{Id}_{\mathcal{C},\operatorname*{Ob}%
},\mathit{Id}_{\mathcal{C},\operatorname*{Mor}}\right)  $ is defined by
$\mathit{Id}_{\mathcal{C},\operatorname*{Ob}}\left(  X\right)  =X$,
$\mathit{Id}_{\mathcal{C},\operatorname*{Mor}}\left(  \mathsf{f}\right)
=\mathsf{f}$, $\forall X\in\operatorname*{Ob}\mathcal{C}$, $\forall
\mathsf{f}\in\operatorname*{Mor}\mathcal{C}$. Two categories $\mathcal{C}$ and
$\mathcal{C}^{\prime}$ are \textit{equivalent}, if there exist two functors
$\mathit{F}$ and $\mathit{G}$ and two natural transformations $\mathrm{T}%
^{FG}:\mathit{Id}_{\mathcal{C}^{\prime}}\Rightarrow\mathit{F}\circ^{\prime
}\mathit{G}$ and $\mathrm{T}^{GF}:\mathit{G}\circ\mathit{F}\Rightarrow
\mathit{Id}_{\mathcal{C}}$.

For more details and standard properties of categories, see, e.g.
\cite{maclane1,ada/her/str,borceaux1} and refs therein.

The categorification \cite{bae/dol98,cra/yet} of most algebraic structures can
be provided by endowing categories with an additional operation
\cite{ben63,macl63} \textquotedblleft reminding\textquotedblright\ us of the
tensor product \cite{maclane1}.

A binary \textquotedblleft\textit{magmatic}\textquotedblright\ \textit{tensor
category} is $\left(  \mathcal{C},\mathit{M}^{\left(  2\otimes\right)
}\right)  $, where $\mathit{M}^{\left(  2\otimes\right)  }\equiv
\otimes:\mathcal{C}\times\mathcal{C}\rightarrow\mathcal{C}$ is a
bifunctor\footnote{We use this notation with brackets $\mathit{M}^{\left(
2\otimes\right)  }$ \cite{agu/mah}, because they are convenient for further
consideration of the $n$-ary case \cite{dup2019}.}. In component form the
bifunctor is $\mathit{M}^{\left(  2\otimes\right)  }=\left\{  \mathit{M}%
_{\operatorname*{Ob}}^{\left(  2\otimes\right)  },\mathit{M}%
_{\operatorname*{Mor}}^{\left(  2\otimes\right)  }\right\}  $, where
$\mathit{M}_{\operatorname*{Mor}}^{\left(  2\otimes\right)  }$ is%
\begin{align}
&  \mathit{M}_{\operatorname*{Mor}}^{\left(  2\otimes\right)  }\left[
\mathsf{f}_{11^{\prime}},\mathsf{f}_{22^{\prime}}\right]  =\mathit{M}%
_{\operatorname*{Ob}}^{\left(  2\otimes\right)  }\left[  X_{1},X_{2}\right]
\rightarrow\mathit{M}_{\operatorname*{Ob}}^{\left(  2\otimes\right)  }\left[
X_{1}^{\prime},X_{2}^{\prime}\right]  ,\label{mor}\\
&  \mathsf{f}_{ii^{\prime}}:X_{i}\rightarrow X_{i}^{\prime},\mathsf{f}%
_{ii^{\prime}}\in\operatorname*{Mor}\mathcal{C},\forall X_{i},X_{i^{\prime}%
}\in\operatorname*{Ob}\mathcal{C},\ \ \ i=1,2.\nonumber
\end{align}

The composition of the $\mathsf{f}$-components is determined by the binary
mediality property (cf. (\ref{mm}))%
\begin{align}
&  \mathit{M}_{\operatorname*{Mor}}^{\left(  2\otimes\right)  }\left[
\mathsf{f}_{23},\mathsf{g}_{23}\right]  \circ\mathit{M}_{\operatorname*{Mor}%
}^{\left(  2\otimes\right)  }\left[  \mathsf{f}_{12},\mathsf{g}_{12}\right]
=\mathit{M}_{\operatorname*{Mor}}^{\left(  2\otimes\right)  }\left[
\mathsf{f}_{23}\circ\mathsf{f}_{12},\mathsf{g}_{23}\circ\mathsf{g}%
_{12}\right]  ,\label{m2}\\
&  \mathsf{f}_{ij}:X_{i}\rightarrow X_{j},\mathsf{g}_{ij}:Y_{i}\rightarrow
Y_{j},\ \ \ \mathsf{f}_{ij},\mathsf{g}_{ij}\in\operatorname*{Mor}%
\mathcal{C},\forall X_{i},Y_{i}\in\operatorname*{Ob}\mathcal{C}%
,i=1,2,3.\nonumber
\end{align}

The identity of the tensor product satisfies%
\begin{equation}
\mathit{M}_{\operatorname*{Mor}}^{\left(  2\otimes\right)  }\left[
\operatorname*{id}\nolimits_{X_{1}},\operatorname*{id}\nolimits_{X_{2}%
}\right]  =\operatorname*{id}\nolimits_{\mathit{M}_{\operatorname*{Ob}%
}^{\left(  2\otimes\right)  }\left[  X_{1},X_{2}\right]  }. \label{id}%
\end{equation}

We call a category $\mathcal{C}$ a \textit{strict} (\textit{binary})
\textit{semigroupal} \cite{yetter,lu/ye/hu}\textit{ }(or \textit{strictly
associative semigroupal category} \cite{boy2007}, also, \textit{semi-monoidal}
\cite{koc2008}), if the bifunctor $\mathit{M}^{\left(  2\otimes\right)  }$
satisfies \textsf{only} (without unit objects and unitors) the binary
associativity condition $\left(  X_{1}\otimes X_{2}\right)  \otimes
X_{3}=X_{1}\otimes\left(  X_{2}\otimes X_{3}\right)  $ and $\left(
\mathsf{f}_{1}\otimes\mathsf{f}_{2}\right)  \otimes\mathsf{f}_{3}%
=\mathsf{f}_{1}\otimes\left(  \mathsf{f}_{2}\otimes\mathsf{f}_{3}\right)  $,
where $X_{i}\in\operatorname*{Ob}\mathcal{C}$, $\mathsf{f}_{i}\in
\operatorname*{Mor}\mathcal{C}$, $i=1,2,3$ (also denoted by $\mathtt{sSGCat}%
$). Strict associativity is the equivalence%
\begin{align}
\mathit{M}_{\operatorname*{Ob}}^{\left(  2\otimes\right)  }\left[
\mathit{M}_{\operatorname*{Ob}}^{\left(  2\otimes\right)  }\left[  X_{1}%
,X_{2}\right]  ,X_{3}\right]   &  =\mathit{M}_{\operatorname*{Ob}}^{\left(
2\otimes\right)  }\left[  X_{1},\mathit{M}_{\operatorname*{Ob}}^{\left(
2\otimes\right)  }\left[  X_{2},X_{3}\right]  \right]  ,\label{as1}\\
\mathit{M}_{\operatorname*{Mor}}^{\left(  2\otimes\right)  }\left[
\mathit{M}_{\operatorname*{Mor}}^{\left(  2\otimes\right)  }\left[
\mathsf{f}_{1},\mathsf{f}_{2}\right]  ,\mathsf{f}_{3}\right]   &
=\mathit{M}_{\operatorname*{Mor}}^{\left(  2\otimes\right)  }\left[
\mathsf{f}_{1},\mathit{M}_{\operatorname*{Mor}}^{\left(  2\otimes\right)
}\left[  \mathsf{f}_{2},\mathsf{f}_{3}\right]  \right]  . \label{as2}%
\end{align}

\begin{remark}
\label{rem-shape}Usually, only the first equation for the $X$-components is
presented in the definition of associativity (and other properties), while the
equation for the $\mathsf{f}$-components is assumed to be satisfied
\textquotedblleft automatically\textquotedblright\ having the same form
\cite{maclane1,stasheff}. In some cases, the diagrams for $\mathit{M}%
_{\operatorname*{Ob}}^{\left(  2\otimes\right)  }$ and $\mathit{M}%
_{\operatorname*{Mor}}^{\left(  2\otimes\right)  }$ can fail to coincide and
have different shapes, for instance, in the case of the \textit{dagger
categories} dealing with the \textquotedblleft reverse\textquotedblright%
\ morphisms \cite{abr/coe08}.
\end{remark}

The associativity relations guarantee that in any product of objects or
morphisms different ways of inserting parentheses lead to equivalent results
(as for semigroups).

In the case of a \textit{non-strict semigroupal category }$\mathtt{SGCat}$
(with no unit objects and unitors) \cite{yetter,boy2007} (see, also,
\cite{lu/ye/hu,elg2004,dav2007}) a collection of mappings can be introduced
which are just the isomorphisms (\textit{associators}) $\mathrm{A}^{\left(
3\otimes\right)  }=\left\{  \mathrm{A}_{\operatorname*{Ob}}^{\left(
3\otimes\right)  },\mathrm{A}_{\operatorname*{Mor}}^{\left(  3\otimes\right)
}\right\}  $ from the left side functor to the right side functor of
(\ref{as1})--(\ref{as2}) as%
\begin{equation}
\mathrm{A}_{\operatorname*{Ob}}^{\left(  3\otimes\right)  }\left(  X_{1}%
,X_{2},X_{3}\right)  :\mathit{M}_{\operatorname*{Ob}}^{\left(  2\otimes
\right)  }\left[  \mathit{M}_{\operatorname*{Ob}}^{\left(  2\otimes\right)
}\left[  X_{1},X_{2}\right]  ,X_{3}\right]  \overset{\simeq}{\rightarrow
}\mathit{M}_{\operatorname*{Ob}}^{\left(  2\otimes\right)  }\left[
X_{1},\mathit{M}_{\operatorname*{Ob}}^{\left(  2\otimes\right)  }\left[
X_{2},X_{3}\right]  \right]  , \label{a1}%
\end{equation}
where $\mathrm{A}_{\operatorname*{Mor}}^{\left(  3\otimes\right)  }$ may be
interpreted similar to the diagonal in (\ref{x-com}), because the associators
are natural transformations \cite{maclane1} or \textit{tri-functorial
isomorphisms} (in the terminology of \cite{boy2007}). Now different ways of
inserting parentheses in a product of $N$ objects give different results in
the absence of conditions on the associator $\mathrm{A}^{\left(
3\otimes\right)  }$. However, if the associator $\mathrm{A}^{\left(
3\otimes\right)  }$ satisfies some consistency relations, they can give
isomorphic results, such that the corresponding diagrams commute, which is the
statement of the \textit{coherence theorem} \cite{macl63,kel64}. This can also
be applied to $\mathtt{SGCat}$, because it can be proved independently of
existence of units \cite{yetter,boy2007,lu/ye/hu}. It was shown \cite{macl63}
that it is sufficient to consider one commutative diagram using the associator
(the \textit{associativity constraint}) for two different rearrangements of
parentheses for 3 tensor multiplications of 4 objects, giving the following
isomorphism%
\begin{equation}
\mathit{M}_{\operatorname*{Ob}}^{\left(  2\otimes\right)  }\left[
\mathit{M}_{\operatorname*{Ob}}^{\left(  2\otimes\right)  }\left[
\mathit{M}_{\operatorname*{Ob}}^{\left(  2\otimes\right)  }\left[  X_{1}%
,X_{2}\right]  ,X_{3}\right]  ,X_{4}\right]  \overset{\simeq}{\rightarrow
}\mathit{M}_{\operatorname*{Ob}}^{\left(  2\otimes\right)  }\left[
X_{1},\mathit{M}_{\operatorname*{Ob}}^{\left(  2\otimes\right)  }\left[
X_{2},\mathit{M}_{\operatorname*{Ob}}^{\left(  2\otimes\right)  }\left[
X_{3},X_{4}\right]  \right]  \right]  . \label{mob}%
\end{equation}
The associativity constraint is called a \textit{pentagon axiom}
\cite{maclane1}, such that the diagram\footnote{We omit $\mathit{M}%
_{\operatorname*{Ob}}^{\left(  2\otimes\right)  }$ in diagrams by leaving the
square brackets only and use the obvious subscripts in $\mathrm{A}^{\left(
3\otimes\right)  }$.} \begin{equation}
\xymatrix@R+3mm@C-20mm{
      &\left[  \left[  X_{1},\left[  X_{2},X_{3}\right]  \right]  ,X_{4}\right]
      \ar[rr]\ar@{}@<1.1ex>[rr]^{ \mathrm{A}_{1,23,4}^{\left(  3\otimes\right)  }}&&
      \left[  X_{1},\left[  \left[  X_{2},X_{3}\right]  ,X_{4}\right]  \right]
      \ar[dr]^<>(.6){\operatorname*{id}\nolimits_{X_1}\otimes\mathrm{A}_{2,3,4}^{\left(  3\otimes\right)  }}
      \\
      \left[  \left[
\left[  X_{1}%
,X_{2}\right]  ,X_{3}\right]  ,X_{4}\right]
      \ar[ur]^<>(.4){\mathrm{A}_{1,2,3}^{\left(  3\otimes\right)  }\otimes \operatorname*{id}\nolimits_{X_4}}
      \ar[rrrr]^{\simeq}
      \ar[drr]_<>(.5){\mathrm{A}_{12,3,4}^{\left(  3\otimes\right)  }}&&&&
      \left[ X_{1}, \left[X_{2},
\left[  %
 X_{3} ,X_{4}\right] \right] \right]
      \\
      &&\left[  \left[  X_{1},X_{2}\right]  ,\left[  X_{3},X_{4}\right]  \right]
      \ar[rru]_<>(.6){\mathrm{A}_{1,2,34}^{\left(  3\otimes\right)  }}
      \\
      } \label{diag1}
\end{equation} commutes.

A similar condition for morphisms, but in another context (for $H$-spaces),
was presented in \cite{sta63,stasheff}. Note that there exists a different
(but not alternative) approach to natural associativity without the use of the
pentagon axiom \cite{joy2001}.

The transition from the semigroupal non-strict category $\mathtt{SGCat}$ to
the monoidal non-strict category $\mathtt{MonCat}$ can be done in a way
similar to passing from a semigroup to a monoid: by adding the \textit{unit
object} $E\in\operatorname*{Ob}\mathcal{C}$ and the (\textit{right and
left})\textit{ unitors }$\mathrm{U}_{\left(  1\right)  }^{\left(
2\otimes\right)  }=\left\{  \mathrm{U}_{\left(  1\right)  \operatorname*{Ob}%
}^{\left(  2\otimes\right)  },\mathrm{U}_{\left(  1\right)
\operatorname*{Mor}}^{\left(  2\otimes\right)  }\right\}  $ and $\mathrm{U}%
_{\left(  2\right)  }^{\left(  2\otimes\right)  }=\left\{  \mathrm{U}_{\left(
2\right)  \operatorname*{Ob}}^{\left(  2\otimes\right)  },\mathrm{U}_{\left(
2\right)  \operatorname*{Mor}}^{\left(  2\otimes\right)  }\right\}  $
(\textquotedblleft unit morphisms\textquotedblright\ which are functorial
isomorphisms, natural transformations) \cite{maclane1}%
\begin{align}
\mathrm{U}_{\left(  1\right)  \operatorname*{Ob}}^{\left(  2\otimes\right)  }
&  :\mathit{M}_{\operatorname*{Ob}}^{\left(  2\otimes\right)  }\left[
X,E\right]  \overset{\simeq}{\rightarrow}X,\label{lr1}\\
\mathrm{U}_{\left(  2\right)  \operatorname*{Ob}}^{\left(  2\otimes\right)  }
&  :\mathit{M}_{\operatorname*{Ob}}^{\left(  2\otimes\right)  }\left[
E,X\right]  \overset{\simeq}{\rightarrow}X,\ \ \forall X\in\operatorname*{Ob}%
\mathcal{C}, \label{lr2}%
\end{align}
and $\mathrm{U}_{\left(  1,2\right)  \operatorname*{Mor}}^{\left(
2\otimes\right)  }$ can be viewed as the diagonal in the diagram of naturality
similar to (\ref{x-com}). The unitors are connected with the associator
$\mathrm{A}^{\left(  3\otimes\right)  }$, such that the diagram
(\textit{triangle axiom}) \begin{equation}
\xymatrix@R+5mm@C+10mm{
    \left[ \left[ X_1,E \right],X_2\right] 
    \ar[dr]_{\mathrm{U}_{\left(  1\right)  \operatorname*{Ob}}^{\left(  2\otimes\right)  }\otimes\operatorname*{id}\nolimits_{X_2}} 
    \ar[rr]^{\mathrm{A}_{\operatorname*{Ob}}^{\left(  3\otimes\right)  }}
      &&  \left[ X_1,\left[ E ,X_2\right]\right] 
            \ar[dl]^{\operatorname*{id}\nolimits_{X_1}\otimes \mathrm{U}_{\left(  2\right)  \operatorname*{Ob}}^{\left(  2\otimes\right)  }} \\
      & \left[ X_1,X_2 \right] }\label{diag2}
\end{equation} commutes.

Using the above, the definition of a \textit{binary non-strict monoidal
category} $\mathtt{MonCat}$ can be given as the 6-tuple $\left(
\mathcal{C},\mathit{M}^{\left(  2\otimes\right)  },\mathrm{A}^{\left(
3\otimes\right)  },E,\mathrm{U}^{\left(  2\otimes\right)  }\right)  $ such
that the pentagon axiom (\ref{diag1}) and the triangle axiom (\ref{diag2}) are
satisfied \cite{macl63,maclane1} (see, also, \cite{kel64,kel65}).

The following \textquotedblleft normalizing\textquotedblright\ relations for
the unitors of a monoidal non-strict category%
\begin{equation}
\mathrm{U}_{\left(  1\right)  \operatorname*{Ob}}^{\left(  2\otimes\right)
}\left(  E\right)  =\mathrm{U}_{\left(  2\right)  \operatorname*{Ob}}^{\left(
2\otimes\right)  }\left(  E\right)  ,\label{uu}%
\end{equation}
can be proven \cite{joy/str1}, as well as that the diagrams \begin{equation}
\xymatrix@R+5mm@C-10mm{
    \left[ \left[ X_1 ,X_2\right], E \right] 
    \ar[dr]_{\mathrm{U}_{\left(  1\right)  \operatorname*{Ob}}^{\left(  2\otimes\right)  }} 
    \ar[rr]^{\mathrm{A}_{\operatorname*{Ob}}^{\left(  3\otimes\right)  }}
      &&  \left[ X_1,\left[ X_2,E\right]\right] 
            \ar[dl]^{\operatorname*{id}\nolimits_{X_1}\otimes\mathrm{U}_{\left(  1\right)  \operatorname*{Ob}}^{\left(  2\otimes\right)  }} \\
      & \left[ X_1,X_2 \right] }\, \, \, \, \, \; \; \;\;\;\;\;\
      \xymatrix@R+5mm@C-10mm{
    \left[ \left[ E,X_1 \right],X_2\right] 
    \ar[dr]_{\mathrm{U}_{\left(  2\right)  \operatorname*{Ob}}^{\left(  2\otimes\right)  }\otimes\operatorname*{id}\nolimits_{X_2}} 
    \ar[rr]^{\mathrm{A}_{\operatorname*{Ob}}^{\left(  3\otimes\right)  }}
      &&  \left[E, \left[ X_1, X_2\right]\right] 
            \ar[dl]^{\mathrm{U}_{\left(  2\right)  \operatorname*{Ob}}^{\left(  2\otimes\right)  }} \\
      & \left[ X_1,X_2 \right] } 
      \label{diag3}
\end{equation} commute.

The \textit{coherence theorem} \cite{ben63,macl63} proves that any diagram in
a non-strict monoidal category, which can be built from an associator
satisfying the pentagon axiom (\ref{diag1}) and unitors satisfying the
triangle axiom (\ref{diag2}), commutes. Another formulation \cite{maclane1}
states that every monoidal non-strict category is (monoidally) equivalent to a
monoidal strict one (see, also, \cite{kassel}).

Thus, it is important to prove analogs of the coherence theorem for various
existing generalizations of categories (having weak modification of units
\cite{koc2008,joy/koc,and2017}, and from the \textquotedblleft periodic
table\textquotedblright\ of higher categories \cite{bae/dol}), as well as for
further generalizations (e.g., $n$-ary ones below).

\section{\label{sec-ntensor}\textsc{Polyadic tensor categories}}

The arity of the additional multiplication in a category (the tensor product)
was previously taken to be binary. Here we introduce categories with tensor
multiplication which \textquotedblleft remind\textquotedblright\ $n$-ary
semigroups, $n$-ary monoids and $n$-ary groups \cite{dor3,pos} (see, also,
\cite{galmak1}), i.e. we provide the categorification \cite{cra/fre,cra/yet}
of \textquotedblleft higher-arity\textquotedblright\ structures according to
the Baez-Dolan microcosm principle \cite{bae/dol98a}. In our considerations we
use the term \textquotedblleft tensor category\textquotedblright\ in a wider
context, because it can include not only binary monoid-like structures and
their combinations, but also $n$-ary-like algebraic structures. It is
important to note that our construction is different from other higher
generalizations of categories\footnote{The terms \textquotedblleft$k$-ary
algebraic category\textquotedblright\ and \textquotedblleft$k$-ary
category\textquotedblright\ appeared in \cite{her71} and \cite{shu2012},
respectively, but they describe different constructions.}, such as
$2$-categories \cite{kel/str} and bicategories \cite{ben67}, $n$-categories
\cite{Bae97,Lei2002} and $n$-categories of $n$-groups \cite{ald/noo},
multicategories \cite{lam69,lei98,cru/shu}, $n$-tuple categories and multiple
categories \cite{grandis}, iterated ($n$-fold) monoidal categories
\cite{bal/fie/sch/vog1}, iterated icons \cite{che/gur}, and obstructed
categories \cite{dup/mar7,dup/mar2018a}. We introduce the categorification of
\textquotedblleft higher-arity\textquotedblright\ structures along
\cite{dup2019} and consider their properties, some of them are different from
the binary case (as in $n$-ary (semi)groups and $n$-ary monoids).

Let $\mathcal{C}$ be a category \cite{maclane1}, and introduce an additional
multiplication as an $n$-ary tensor product as in \cite{dup2018a,dup2019}.

\begin{definition}
An $n$-ary tensor product in a category $\mathcal{C}$ is an $n$-ary functor%
\begin{equation}
\mathit{M}^{\left(  n\otimes\right)  }:\overset{n}{\overbrace{\mathcal{C}%
\times\ldots\times\mathcal{C}}}\rightarrow\mathcal{C} \label{mc}%
\end{equation}
having the component form $\mathit{M}^{\left(  n\otimes\right)  }=\left\{
\mathit{M}_{\operatorname*{Ob}}^{\left(  n\otimes\right)  },\mathit{M}%
_{\operatorname*{Mor}}^{\left(  n\otimes\right)  }\right\}  $ where the
$\mathsf{f}$-component $\mathit{M}_{\operatorname*{Mor}}^{\left(
n\otimes\right)  }$ is%
\begin{align}
&  \mathit{M}_{\operatorname*{Mor}}^{\left(  n\otimes\right)  }\left[
\mathsf{f}_{11^{\prime}},\mathsf{f}_{22^{\prime}},\ldots\mathsf{f}%
_{nn^{\prime}}\right]  =\mathit{M}_{\operatorname*{Ob}}^{\left(
n\otimes\right)  }\left[  X_{1},X_{2},\ldots X_{n}\right]  \rightarrow
\mathit{M}_{\operatorname*{Ob}}^{\left(  n\otimes\right)  }\left[
X_{1}^{\prime},X_{2}^{\prime},\ldots X_{n}^{\prime}\right]  ,\label{mf}\\
&  \mathsf{f}_{ii^{\prime}}:X_{i}\rightarrow X_{i}^{\prime},\mathsf{f}%
_{ii^{\prime}}\in\operatorname*{Mor}\mathcal{C},\forall X_{i},X_{i}^{\prime
}\in\operatorname*{Ob}\mathcal{C},i=1,\ldots,n.\nonumber
\end{align}

\end{definition}

The $n$-ary composition of the $\mathsf{f}$-components (morphism products of
length $n$) is determined by the $n$-ary mediality property (cf. (\ref{mna}))%
\begin{align}
&  \mathit{M}_{\operatorname*{Mor}}^{\left(  n\otimes\right)  }\left[
\mathsf{f}^{\left(  1,1\right)  },\mathsf{f}^{\left(  1,2\right)  }%
,\ldots,\mathsf{f}^{\left(  1,n\right)  }\right]  \circ\ldots\circ
\mathit{M}_{\operatorname*{Mor}}^{\left(  n\otimes\right)  }\left[
\mathsf{f}^{\left(  n,1\right)  },\mathsf{f}^{\left(  n,2\right)  }%
,\ldots,\mathsf{f}^{\left(  n,n\right)  }\right]  ,\nonumber\\
&  =\mathit{M}_{\operatorname*{Mor}}^{\left(  n\otimes\right)  }\left[
\mathsf{f}^{\left(  1,1\right)  }\circ\mathsf{f}^{\left(  2,1\right)  }%
\circ\ldots\circ\mathsf{f}^{\left(  n,1\right)  },\ldots,\mathsf{f}^{\left(
1,n\right)  }\circ\ldots\circ\mathsf{f}^{\left(  n,n\right)  }\right]  ,\\
&  \mathsf{f}^{\left(  i,j\right)  }\in\operatorname*{Mor}\mathcal{C}%
,\ \ \ i,j=1,2\ldots,n.\nonumber
\end{align}

The identity morphism of the $n$-ary tensor product satisfies%
\begin{equation}
\mathit{M}_{\operatorname*{Mor}}^{\left(  n\otimes\right)  }\left[
\operatorname*{id}\nolimits_{X_{1}},\operatorname*{id}\nolimits_{X_{2}}%
,\ldots,\operatorname*{id}\nolimits_{X_{n}}\right]  =\operatorname*{id}%
\nolimits_{\mathit{M}_{\operatorname*{Ob}}^{\left(  n\otimes\right)  }\left[
X_{1},X_{2}\ldots,X_{n}\right]  }. \label{idn}%
\end{equation}

\begin{definition}
An $n$-ary tensor product $\mathit{M}^{\left(  n\otimes\right)  }$ which can
be constructed from a binary tensor product $\mathit{M}^{\prime\left(
2\otimes\right)  }$ by successive (iterative) repetitions is called an
\textit{arity-reduced tensor product}\footnote{By analogy with the
\textquotedblleft derived $n$-ary group\textquotedblright\ \cite{dor3,pos}.},
and otherwise it is called an \textit{arity-nonreduced tensor product}.
\end{definition}

Categories containing iterations of the binary tensor product were considered
in \cite{bal/fie/sch/vog1,che/gur}. We will mostly be interested in the
arity-nonreducible tensor products and their corresponding categories.

\begin{definition}
\label{def-magcat}A polyadic ($n$-ary) \textquotedblleft\textit{magmatic}%
\textquotedblright\ \textit{tensor category} is $\left(  \mathcal{C}%
,\mathit{M}^{\left(  n\otimes\right)  }\right)  $, where $\mathit{M}^{\left(
n\otimes\right)  }$ is an $n$-ary tensor product (functor (\ref{mc})), and it
is called an \textit{arity-reduced category }or \textit{arity-nonreduced
category} depending on its tensor product.
\end{definition}

\subsection{Polyadic semigroupal categories}

We call sequences of objects and morphisms $X$-polyads and $\mathsf{f}%
$-polyads \cite{pos}, and denote them $\mathfrak{X}$ and $\mathfrak{f}$,
respectively (as in (\ref{mass})).

\begin{definition}
\label{def-assoc}The $n$-ary functor $\mathit{M}^{\left(  n\otimes\right)  }$
is \textit{totally }($n$-ary)\textit{ associative}, if it satisfies the
following $\left(  n-1\right)  $ pairs of $X$ equivalences%
\begin{equation}
\mathit{M}_{\operatorname*{Ob}}^{\left(  n\otimes\right)  }\left[
\mathfrak{X},\mathit{M}_{\operatorname*{Ob}}^{\left(  n\otimes\right)
}\left[  \mathfrak{Y}\right]  ,\mathfrak{Z}\right]  =equivalent,\label{me1}%
\end{equation}
where $\mathfrak{X},\mathfrak{Y},\mathfrak{Z}$ are $X$-polyads of the
necessary length, and the total length of each $\left(  \mathfrak{X}%
,\mathfrak{Y},\mathfrak{Z}\right)  $-polyad is $2n-1$, while the internal
tensor products in (\ref{me1}) can be on any of the $n$ places.
\end{definition}

\begin{example}
In the ternary case ($n=3$) the total associativity for the $X$-polyads of the
length $5=2\cdot3-1$ gives $2=3-1$ pairs of equivalences%
\begin{align}
\mathit{M}_{\operatorname*{Ob}}^{\left(  3\otimes\right)  }\left[
\mathit{M}_{\operatorname*{Ob}}^{\left(  3\otimes\right)  }\left[  X_{1}%
,X_{2},X_{3}\right]  ,X_{4},X_{5}\right]   &  =\mathit{M}_{\operatorname*{Ob}%
}^{\left(  3\otimes\right)  }\left[  X_{1},\mathit{M}_{\operatorname*{Ob}%
}^{\left(  3\otimes\right)  }\left[  X_{2},X_{3},X_{4}\right]  ,X_{5}\right]
\nonumber\\
&  =\mathit{M}_{\operatorname*{Ob}}^{\left(  3\otimes\right)  }\left[
X_{1},X_{2},\mathit{M}_{\operatorname*{Ob}}^{\left(  3\otimes\right)  }\left[
X_{3},X_{4},X_{5}\right]  \right]  ,\\
\forall X_{i} &  \in\operatorname*{Ob}\mathcal{C},\ \ \ \forall\mathsf{f}%
_{i}\in\operatorname*{Mor}\mathcal{C},\ \ \ i=1,\ldots,5.\nonumber
\end{align}

\end{example}

\begin{definition}
\label{def-semicat}A category $\left(  \mathcal{C},\mathit{M}^{\left(
n\otimes\right)  }\right)  $ is called a \textit{polyadic }($n$-\textit{ary})
\textit{strict} \textit{semigroupal} \textit{category} $\mathtt{sSGCat}_{n}$,
if the bifunctor $\mathit{M}^{\left(  n\otimes\right)  }$ satisfies objects
and unitors) the $n$-ary associativity condition (\ref{me1}).
\end{definition}

Thus, in a polyadic strict semigroupal category for any (allowed, i.e. having
the size $k\left(  n-1\right)  +1$, $\forall k\in\mathbb{N}$, where $k$ is the
number of $n$-ary tensor multiplications) product of objects (or morphisms),
all different ways of inserting parentheses give equivalent results (as for
$n$-ary semigroups).

\subsection{$N$-ary coherence}

As in the binary case (\ref{a1}), the transition to non-strict categories
results in the consideration of \textsf{independent} isomorphisms instead of
the equivalence (\ref{me1}).

\begin{definition}
The $\left(  n-1\right)  $ pairs of $X$ and $\mathsf{f}$ isomorphisms
$\mathrm{A}^{\left(  2n-1\right)  \otimes}=\left\{  \mathrm{A}%
_{\operatorname*{Ob}}^{\left(  2n-1\right)  \otimes},\mathrm{A}%
_{\operatorname*{Mor}}^{\left(  2n-1\right)  \otimes}\right\}  $ such that%
\begin{equation}
\mathrm{A}_{i,\operatorname*{Ob}}^{\left(  2n-1\right)  \otimes}%
:\mathit{M}_{\operatorname*{Ob}}^{\left(  n\otimes\right)  }\left[
\mathfrak{X},\mathit{M}_{i,\operatorname*{Ob}}^{\left(  n\otimes\right)
}\left[  \mathfrak{Y}\right]  ,\mathfrak{Z}\right]  \overset{\simeq
}{\rightarrow}\mathit{M}_{\operatorname*{Ob}}^{\left(  n\otimes\right)
}\left[  \mathfrak{X},\mathit{M}_{i+1,\operatorname*{Ob}}^{\left(
n\otimes\right)  }\left[  \mathfrak{Y}\right]  ,\mathfrak{Z}\right]  ,
\end{equation}
are called $n$-\textit{ary associators} being $\left(  2n-1\right)  $-place
natural transformations, where $\mathrm{A}_{\operatorname*{Mor}}^{\left(
2n-1\right)  \otimes}$ may be viewed as corresponding diagonals as in
(\ref{x-com}). Here $i=1,\ldots,n-1$ is the place of the internal brackets.
\end{definition}

In the ternary case ($n=3$) we have $2=3-1$ pairs of the ternary associators%
\begin{equation}
\mathrm{A}_{1,\operatorname*{Ob}}^{\left(  5\otimes\right)  }:\mathit{M}%
_{\operatorname*{Ob}}^{\left(  3\otimes\right)  }\left[  \mathit{M}%
_{\operatorname*{Ob}}^{\left(  3\otimes\right)  }\left[  X_{1},X_{2}%
,X_{3}\right]  ,X_{4},X_{5}\right]  \overset{\simeq}{\rightarrow}%
\mathit{M}_{\operatorname*{Ob}}^{\left(  3\otimes\right)  }\left[
X_{1},\mathit{M}_{\operatorname*{Ob}}^{\left(  3\otimes\right)  }\left[
X_{2},X_{3},X_{4}\right]  ,X_{5}\right]  , \label{a51}%
\end{equation}
and%
\begin{equation}
\mathrm{A}_{2,\operatorname*{Ob}}^{\left(  5\otimes\right)  }:\mathit{M}%
_{\operatorname*{Ob}}^{\left(  3\otimes\right)  }\left[  X_{1},\mathit{M}%
_{\operatorname*{Ob}}^{\left(  3\otimes\right)  }\left[  X_{2},X_{3}%
,X_{4}\right]  ,X_{5}\right]  \overset{\simeq}{\rightarrow}\mathit{M}%
_{\operatorname*{Ob}}^{\left(  3\otimes\right)  }\left[  X_{1},X_{2}%
,\mathit{M}_{\operatorname*{Ob}}^{\left(  3\otimes\right)  }\left[
X_{3},X_{4},X_{5}\right]  \right]  . \label{a52}%
\end{equation}

It is now definite that different ways of inserting parentheses in a product
of $N$ objects will give different results (the same will be true for
morphisms as well), if we do not impose constraints on the associators. We
anticipate that we will need (as in the binary case (\ref{mob})) only one more
(i.e. three) tensor multiplication than appears in the associativity
conditions (\ref{me1}) to make a commutative diagram for the following
isomorphism of $3\cdot\left(  n-1\right)  +1=3n-2$ objects%
\begin{align}
&  \mathit{M}_{\operatorname*{Ob}}^{\left(  n\otimes\right)  }\left[
\mathit{M}_{\operatorname*{Ob}}^{\left(  n\otimes\right)  }\left[
\mathit{M}_{\operatorname*{Ob}}^{\left(  n\otimes\right)  }\left[
X_{1},\ldots,X_{n}\right]  ,X_{n+1},\ldots,X_{2n-1}\right]  ,X_{2n}%
,\ldots,X_{3n-2}\right]  \nonumber\\
&  \overset{\simeq}{\rightarrow}\mathit{M}_{\operatorname*{Ob}}^{\left(
n\otimes\right)  }\left[  X_{1},\ldots,X_{n-1},\mathit{M}_{\operatorname*{Ob}%
}^{\left(  n\otimes\right)  }\left[  X_{n},\ldots,X_{2n-2},\mathit{M}%
_{\operatorname*{Ob}}^{\left(  n\otimes\right)  }\left[  X_{2n-1}%
,\ldots,X_{3n-2}\right]  \right]  \right]  .\label{mobn}%
\end{align}

\begin{conjecture}
[$N$-\textsf{ary coherence}]If the $n$-ary associator $\mathrm{A}^{\left(
2n-1\right)  \otimes}$ satisfies such $n$-\textit{ary coherence conditions}
that the isomorphism (\ref{mobn}) takes place, then any diagram containing
$\mathrm{A}^{\left(  2n-1\right)  \otimes}$ together with the identities
(\ref{idn}) commutes.
\end{conjecture}

The $n$-ary coherence conditions are described by a \textquotedblleft$\left(
n^{2}+1\right)  $-gon\textquotedblright, which is the pentagon (\ref{diag1})
for $n=2$\ (for classification of \textquotedblleft$N$-gons\textquotedblright%
\ see, e.g., \cite{wenninger}).

\begin{definition}
\label{def-nonstrscat}A category $\left(  \mathcal{C},\mathit{M}^{\left(
n\otimes\right)  }\right)  $ is called a \textit{polyadic }($n$-\textit{ary})
\textit{non-strict} \textit{semigroupal} \textit{category} $\mathtt{sSGCat}%
_{n}$, if the bifunctor $\mathit{M}^{\left(  n\otimes\right)  }$ satisfies the
$n$-ary coherence.
\end{definition}

\begin{example}
\label{exam-deca}In the ternary case $n=3$ we have $2$ pairs of $5$-place
associators (\ref{a51})--(\ref{a52}) $\mathrm{A}_{1}^{\left(  5\otimes\right)
}$ and $\mathrm{A}_{2}^{\left(  5\otimes\right)  }$ which act on $7=3\cdot3-2$
objects (\ref{mobn}). We consider the diagram for objects only, then the
associativity constraint for the associators $\mathrm{A}_{1,\operatorname*{Ob}%
}^{\left(  5\otimes\right)  }$ and $\mathrm{A}_{2,\operatorname*{Ob}}^{\left(
5\otimes\right)  }$ will be a \textit{decagon axiom} requiring that the
diagram

\begin{equation}
\xymatrix@R+3mm@C-20mm@L+2mm{
      && \left[ X_{1}, \left[X_{2},\left[X_{3},X_{4},  X_{5} \right] , X_{6} \right] ,X_{7} \right]
      \ar[drr]^(.7){\mathrm{A}_{2,\operatorname*{Ob}1,2,345,6,7}^{\left(  5\otimes\right) }}
      \\
      \left[ \left[X_{1}, X_{2},\left[X_{3},X_{4},  X_{5} \right]\right] , X_{6} ,X_{7}  \right]
      \ar[urr]^(.3){\mathrm{A}_{1,\operatorname*{Ob}1,2,345,6,7}^{\left( 5\otimes\right) }}
      &&&& 
      \left[ X_{1}, X_{2},\left[\left[ X_{3},X_{4}, X_{5}\right] , X_{6} ,X_{7} \right] \right]
      \ar[d]^{\operatorname*{id}\nolimits_{X_{1}}\otimes\operatorname*{id}\nolimits_{X_{2}}\otimes
      \mathrm{A}_{1,\operatorname*{Ob}3,4,5,6,7}^{\left(  5\otimes\right) }}
      \\
     \left[\left[ X_{1},\left[ X_{2},X_{3},X_{4}\right] ,  X_{5}\right] , X_{6} ,X_{7} \right]
      \ar[u]^{\mathrm{A}_{2,\operatorname*{Ob}1,2,3,4,5}^{\left(  5\otimes\right) }\otimes
      \operatorname*{id}\nolimits_{X_{6}}\otimes\operatorname*{id}\nolimits_{X_{7}}} &&&&
     \left[ X_{1}, X_{2},\left[X_{3},\left[ X_{4}, X_{5} , X_{6}\right] ,X_{7} \right] \right]
      \ar[d]^{\operatorname*{id}\nolimits_{X_{1}}\otimes\operatorname*{id}\nolimits_{X_{2}}\otimes
      \mathrm{A}_{2,\operatorname*{Ob}3,4,5,6,7}^{\left(  5\otimes\right) }}
      \\
    \left[  \left[ \left[  X_{1},X_{2} ,X_{3}\right]  ,X_{4},X_{5}\right]  ,X_{6} ,X_{7}\right]
      \ar[u]^{\mathrm{A}_{1,\operatorname*{Ob}1,2,3,4,5}^{\left(  5\otimes\right) }\otimes
      \operatorname*{id}\nolimits_{X_{6}}\otimes\operatorname*{id}\nolimits_{X_{7}}}
      \ar[d]_{\mathrm{A}_{1,\operatorname*{Ob}123,4,5,6,7}^{\left(  5\otimes\right) }}
      \ar[rrrr]^{\simeq}_{(\ref{mobn})}
     &&&&
     \left[ X_{1}, X_{2},\left[X_{3}, X_{4},\left[ X_{5} , X_{6}  ,X_{7} \right]\right]\right]
      \\
      \left[\left[ X_{1}, X_{2},X_{3}\right],\left[ X_{4},  X_{5} , X_{6} \right],X_{7}  \right] 
      \ar[drr]_(.3){\mathrm{A}_{2,\operatorname*{Ob}123,4,5,6,7}^{\left(  5\otimes\right) }}
      &&&&
      \left[ X_{1},\left[ X_{2}, X_{3},X_{4}\right],\left[  X_{5} , X_{6} ,X_{7} \right] \right]
      \ar[u]_{\mathrm{A}_{2,\operatorname*{Ob}1,2,3,4,567}^{\left(  5\otimes\right) }}
      \\
      &&
      \left[ \left[ X_{1},X_{2} X_{3}\right],X_{4},\left[  X_{5} , X_{6} ,X_{7} \right] \right]
      \ar[rru]_(.7){\mathrm{A}_{1,\operatorname*{Ob}1,2,3,4,567}^{\left(  5\otimes\right) }}
            } \label{diag4}
\end{equation} commutes (cf. (\ref{diag4}) and the pentagon axiom
(\ref{diag1}) for binary non-strict tensor categories).
\end{example}

\section{$N$\textsc{-ary units, unitors and quertors}}

Introducing $n$-ary analogs of units and unitors is non-trivial, because in
$n$-ary structures there are various possibilities: one unit, many units, all
elements are units or there are no units at all (see, e.g., for $n$-ary groups
\cite{dor3,pos,galmak1}, and for $n$-ary monoids \cite{pop/pop}). A similar
situation is expected in category theory after proper categorification
\cite{cra/fre,cra/yet,bae/dol98} of $n$-ary structures.

\subsection{Polyadic monoidal categories\label{subsec-nmon}}

Let $\left(  \mathcal{C},\mathit{M}^{\left(  n\otimes\right)  },\mathrm{A}%
^{\left(  2n-1\right)  \otimes}\right)  $ be an $n$-ary non-strict semigroupal
category $\mathtt{SGCat}_{n}$ (see \textbf{Definition \ref{def-semicat}}) with
$n$-ary tensor product $\mathit{M}^{\left(  n\otimes\right)  }$ and the
associator $\mathrm{A}^{\left(  2n-1\right)  \otimes}$ satisfying $n$-ary
coherence. If a category has a \textit{unit neutral sequence} of objects
$\mathfrak{E}_{\left(  n-1\right)  }=\left(  E_{1},\ldots,E_{i}\right)  $,
$E_{i}\in\operatorname*{Ob}\mathcal{C},$ $i=1,\ldots,n-1$, we call it a
\textit{unital} category. Note that the unit neutral sequence may not be
unique. If all $E_{i}$ coincide $E_{i}=E\in\operatorname*{Ob}\mathcal{C}$,
then $E$ is called a \textit{unit object} of $\mathcal{C}$. The $n$-ary
\textit{unitors }$\mathrm{U}_{\left(  i\right)  }^{\left(  n\otimes\right)  }%
$, $i=1,\ldots,n$ ($n$-ary \textquotedblleft unit morphisms\textquotedblright%
\ being natural transformations) are defined by%
\begin{equation}
\mathrm{U}_{\left(  i\right)  \operatorname*{Ob}}^{\left(  n\otimes\right)
}:\mathit{M}_{\operatorname*{Ob}}^{\left(  n\otimes\right)  }\left[
E_{1},\ldots E_{i-1},X,E_{i+1},\ldots E_{n}\right]  \overset{\simeq
}{\rightarrow}X,\ \ \ \forall X,E_{i}\in\operatorname*{Ob}\mathcal{C}%
,i=1,\ldots,n-1.
\end{equation}

The $n$-ary unitors $\mathrm{U}_{\left(  i\right)  }^{\left(  n\otimes\right)
}$ are compatible with the $n$-ary associators $\mathrm{A}^{\left(
2n-1\right)  \otimes}$ by the analog of the triangle axiom (\ref{diag2}). In
the binary case (\ref{lr1})--(\ref{lr2}), we have $\mathrm{U}_{\left(
1\right)  }^{\left(  2\otimes\right)  }=\mathrm{R}^{\left(  2\otimes\right)
}$, $\mathrm{U}_{\left(  2\right)  }^{\left(  2\otimes\right)  }%
=\mathrm{L}^{\left(  2\otimes\right)  }$.

\begin{definition}
\label{def-nmon}A polyadic ($n$\textit{-ary})\textit{ non-strict monoidal
category} $\mathtt{MonCat}_{n}$ is a polyadic ($n$-ary) non-strict semigroupal
category $\mathtt{SGCat}_{n}$ endowed with a unit neutral sequence
$\mathfrak{E}_{\left(  n-1\right)  }$ and $n$ unitors\textit{ }$\mathrm{U}%
_{\left(  i\right)  }^{\left(  n\otimes\right)  }$, $i=1,\ldots,n$, that is a
5-tuple $\left(  \mathcal{C},\mathit{M}^{\left(  n\otimes\right)  }%
,\mathrm{A}^{\left(  n\otimes\right)  },\mathfrak{E}_{\left(  n-1\right)
},\mathrm{U}^{\left(  n\otimes\right)  }\right)  $ satisfying the
\textquotedblleft$\left(  n^{2}+1\right)  $-gon\textquotedblright\ axiom for
the $\left(  n-1\right)  $ associators $\mathrm{A}_{\left(  i\right)
}^{\left(  2n-1\right)  \otimes}$ and the triangle axiom (the analog of
(\ref{diag2})) for the unitors and associators compatibility condition.
\end{definition}

\begin{example}
\label{exam-mon3}If we consider the ternary non-strict monoidal category
$\mathtt{MonCat}_{3}$ with one unit object $E\in\operatorname*{Ob}\mathcal{C}%
$, then we have $2$ associators $\mathrm{A}_{1}^{\left(  5\otimes\right)  }$
and $\mathrm{A}_{2}^{\left(  5\otimes\right)  }$ satisfying the decagon axiom
(\ref{diag4}) and $3$ unitors%
\begin{align}
\mathrm{U}_{\left(  1\right)  \operatorname*{Ob}}^{\left(  3\otimes\right)  }
&  :\mathit{M}_{\operatorname*{Ob}}^{\left(  3\otimes\right)  }\left[
X,E,E\right]  \overset{\simeq}{\rightarrow}X,\\
\mathrm{U}_{\left(  2\right)  \operatorname*{Ob}}^{\left(  3\otimes\right)  }
&  :\mathit{M}_{\operatorname*{Ob}}^{\left(  3\otimes\right)  }\left[
E,X,E\right]  \overset{\simeq}{\rightarrow}X,\\
\mathrm{U}_{\left(  3\right)  \operatorname*{Ob}}^{\left(  3\otimes\right)  }
&  :\mathit{M}_{\operatorname*{Ob}}^{\left(  3\otimes\right)  }\left[
E,E,X\right]  \overset{\simeq}{\rightarrow}X,\ \ \ \forall X\in
\operatorname*{Ob}\mathcal{C},
\end{align}
which satisfy the \textquotedblleft normalizing\textquotedblright\ conditions
$\mathrm{U}_{\left(  i\right)  \operatorname*{Ob}}^{\left(  3\otimes\right)
}\left(  E\right)  =E$, $i=1,2,3$ and the ternary analog of the triangle axiom
(\ref{diag2}), such that the diagram
\begin{equation}
\xymatrix@R+10mm@C-5mm@L+2mm{
      &&  \left[E,\left[E,X,  E \right] , E \right] 
      \ar[drr]^(.7){\mathrm{A}_{2,\operatorname*{Ob}}^{\left(  5\otimes\right) }}
      \ar[dd]^(.73){\mathrm{U}_{\left(  2\right)  \operatorname*{Ob}E,EXE,E}^{\left(  3\otimes\right)  }}
      \\
      \left[\left[ E,E,X\right],  E  , E \right]
      \ar[urr]^(.3){\mathrm{A}_{1,\operatorname*{Ob}}^{\left( 5\otimes\right) }}
      \ar[d]_{\mathrm{U}_{\left(  1\right)  \operatorname*{Ob}EEX,E,E}^{\left(  3\otimes\right)  }}&&
            && 
      \left[E,E,\left[X,  E  , E \right]\right]
      \ar[d]^{\mathrm{U}_{\left(  3\right)  \operatorname*{Ob}E,E,XEE}^{\left(  3\otimes\right)  }}
      \\\left[E,  E,X \right]
           \ar[drr]_(.3){\mathrm{U}_{\left(  3\right)  \operatorname*{Ob}E,E,X}^{\left(  3\otimes\right)  }}
      && \left[E,X,  E \right] \ar[d]^(.35){\mathrm{U}_{\left(  2\right)  \operatorname*{Ob}E,X,E}^{\left(  3\otimes\right) }}
      &&
      \left[X,E,  E \right]
            \ar[dll]^(.4){\mathrm{U}_{\left(  1\right)  \operatorname*{Ob}X,E,E}^{\left(  3\otimes\right)  }}
      \\
      &&
      X
                  } \label{diag8}
\end{equation} 
commutes.
\end{example}

\subsection{Polyadic nonunital groupal categories\label{subsec-groupal}}

The main result of $n$-ary group theory \cite{dor3,pos} is connected with
units and neutral polyads: if they exist, then such $n$-ary group is reducible
to a binary group. A similar statement can be true in some sense for categories.

\begin{conjecture}
If a polyadic ($n$-ary) tensor category has unit object and unitors, it can be
arity-reducible to a binary category, such that the $n$-ary product can be
obtained by iterations of the binary tensor product.
\end{conjecture}

Therefore, it would be worthwhile to introduce and study non-reducible
polyadic tensor categories which do not possess unit objects and unitors at
all. This can be done by \textquotedblleft categorification\textquotedblright%
\ of the \textit{querelement} concept \cite{dor3}. Recall that, for instance,
in a ternary group $\left\langle G\mid\mu_{3}\right\rangle $ for an element
$g\in G$ a querelement $\bar{g}$ is uniquely defined by $\mu_{3}\left[
g,g,\bar{g}\right]  =g$, which can be treated as a generalization of the
inverse element concept to the $n$-ary case. The mapping $g\rightarrow\bar{g}$
can be considered as an additional unary operation (\textit{queroperation}) in
the ternary (and $n$-ary) group, while viewing it as an abstract algebra
\cite{gle/gla} such that the notion of the identity is not used. The (binary)
category of $n$-ary groups and corresponding functors were considered in
\cite{mic1,mic84a,ian91}.

Let $\left(  \mathcal{C},\mathit{M}^{\left(  n\otimes\right)  },\mathrm{A}%
^{\left(  2n-1\right)  \otimes}\right)  $ be a polyadic ($n$-ary) non-strict
semigroupal category, where $\mathit{M}^{\left(  n\otimes\right)  }$ is the
$n$-ary tensor product, and $\mathrm{A}^{\left(  2n-1\right)  \otimes}$ is the
associator making the \textquotedblleft$\left(  n^{2}+1\right)  $%
-gon\textquotedblright\ diagram of $n$-ary coherence commutative. We propose a
\textquotedblleft categorification\textquotedblright\ analog of the
queroperation to be a covariant endofunctor of $\mathcal{C}$.

\begin{definition}
\label{def-q}A \textit{querfunctor} $\mathit{Q}:\mathcal{C}\rightarrow
\mathcal{C}$ is an endofunctor of $\mathcal{C}$ sending $\mathit{Q}%
_{\operatorname*{Ob}}\left(  X\right)  =\bar{X}$ and $\mathit{Q}%
_{\operatorname*{Mor}}\left(  \mathsf{f}\right)  =\mathsf{\bar{f}}$, where
$\bar{X}$ and $\mathsf{\bar{f}}$ are the \textit{querobject} and the
\textit{quermorphism} of $X$ and $\mathsf{f}$, respectively, such that the $i$
diagrams ($i=1,\ldots,n$) \begin{equation}
\xymatrix@R+5mm@C+7mm@L+2mm{
\left[\overset{n}{\overbrace {X,\ldots,X}}\right]
    \ar[drr]_{\mathit{P} r_{\operatorname*{Ob}}^{\left(  n\otimes \right)  }} 
    \ar[rrrr]^(.45){\overset{i-1}{\overbrace{\operatorname*{id}\nolimits_{X}
    \otimes,\ldots,\otimes\operatorname*{id}\nolimits_{X}}}
    \otimes{\mathit{Q}_{\operatorname*{Ob}}}\otimes\overset{n-i}{\overbrace
{\operatorname*{id}\nolimits_{X}\otimes,\ldots,\otimes\operatorname*{id}\nolimits_{X}}}}
      && && \left[
\overset{i-1}{\overbrace{X,\ldots,X}},\bar{X},\overset{n-i}{\overbrace
{X,\ldots,X}}\right]
            \ar[dll]^{\mathrm{Q}_{\left(  i\right)  \operatorname*{Ob}}^{\left(  n\otimes\right)  }} \\
      & & X } \label{diag6}
\end{equation} commute (and analogously for
morphisms), where $\mathrm{Q}_{\left(  i\right)  }^{\left(  n\otimes\right)
}$ are \textit{quertors}%
\begin{equation}
\mathrm{Q}_{\left(  i\right)  \operatorname*{Ob}}^{\left(  n\otimes\right)
}:\mathit{M}_{\operatorname*{Ob}}^{\left(  n\otimes\right)  }\left[
\overset{i-1}{\overbrace{X,\ldots,X}},\bar{X},\overset{n-i}{\overbrace
{X,\ldots,X}}\right]  \overset{\simeq}{\rightarrow}X,\ \ \ \forall
X\in\operatorname*{Ob}\mathcal{C},\ \ i=1,\ldots,n,
\end{equation}
and $\mathit{P}r^{\left(  n\otimes\right)  }:\mathcal{C}^{n\otimes}%
\rightarrow\mathcal{C}$ is the projection. The action on morphisms
$\mathrm{Q}_{\left(  i\right)  \operatorname*{Mor}}^{\left(  n\otimes\right)
}$ can be found using the diagonal arrow in the corresponding natural
transformation, as in (\ref{x-com}).
\end{definition}

\begin{example}
\label{exam-q}In the ternary case we have (for objects) the querfunctor
$\mathit{Q}_{\operatorname*{Ob}}\left(  X\right)  =\bar{X}$ and 3 quertor
isomorphisms%
\begin{align}
\mathrm{Q}_{\left(  1\right)  \operatorname*{Ob}}^{\left(  3\otimes\right)  }
&  :\mathit{M}_{\operatorname*{Ob}}^{\left(  3\otimes\right)  }\left[  \bar
{X},X,X\right]  \overset{\simeq}{\rightarrow}X,\label{qx1}\\
\mathrm{Q}_{\left(  2\right)  \operatorname*{Ob}}^{\left(  3\otimes\right)  }
&  :\mathit{M}_{\operatorname*{Ob}}^{\left(  3\otimes\right)  }\left[
X,\bar{X},X\right]  \overset{\simeq}{\rightarrow}X,\label{qx2}\\
\mathrm{Q}_{\left(  3\right)  \operatorname*{Ob}}^{\left(  3\otimes\right)  }
&  :\mathit{M}_{\operatorname*{Ob}}^{\left(  3\otimes\right)  }\left[
X,X,\bar{X}\right]  \overset{\simeq}{\rightarrow}X,\ \ \ \forall
X\in\operatorname*{Ob}\mathcal{C}. \label{qx3}%
\end{align}
The three quertors $\mathrm{Q}_{\left(  i\right)  \operatorname*{Ob}}^{\left(
3\otimes\right)  }$ and the querfunctor $\mathit{Q}$ are connected with two
ternary associators $\mathrm{A}_{1,\operatorname*{Ob}}^{\left(  5\otimes
\right)  },\mathrm{A}_{2,\operatorname*{Ob}}^{\left(  5\otimes\right)  }$
(\ref{a51})--(\ref{a52}) such that the following
diagram
\begin{equation}
\xymatrix@R+9mm@C+1mm@L+2mm{
&&\left[ X, X ,X\right]
\ar[dll]_{\mathit{D}iag^{\left(  3\otimes\right)  }
      \otimes\operatorname*{id} \otimes\operatorname*{id}}
\ar[drr]^{\operatorname*{id} \otimes\operatorname*{id}\otimes\mathit{D}iag^{\left(  3\otimes\right)  }}
 \ar[d]^(0.55){\operatorname*{id}\otimes\mathit{D}iag^{\left(  3\otimes\right)  }\otimes\operatorname*{id}}
\\\left[ \left[ X,X,X \right],X,X\right] 
    \ar[rr]^{\mathrm{A}_{1,\operatorname*{Ob}}^{\left(  5\otimes\right)  }}
\ar[d]_{\operatorname*{id} \otimes\operatorname*{id}\otimes\mathit{Q}_{\operatorname*{Ob}}
      \otimes\operatorname*{id} \otimes\operatorname*{id}}
&&\left[X, \left[ X,X,X \right],X\right]
      \ar[rr]^{\mathrm{A}_{2,\operatorname*{Ob}}^{\left(  5\otimes\right)  }}
      \ar[d]_{\operatorname*{id} \otimes\operatorname*{id}\otimes\mathit{Q}_{\operatorname*{Ob}}
      \otimes\operatorname*{id} \otimes\operatorname*{id}}
&&\left[X,X,\left[ X,X,X \right]\right]\ar[d]_{\operatorname*{id} \otimes\operatorname*{id}\otimes\mathit{Q}_{\operatorname*{Ob}}
      \otimes\operatorname*{id} \otimes\operatorname*{id}}
\\
    \left[ \left[ X,X,\bar{X} \right],X,X\right] 
    \ar[drr]_(.3){\mathrm{Q}_{\left(  3\right)  \operatorname*{Ob}}^{\left(  3\otimes\right)  }
    \otimes\operatorname*{id}\nolimits_{X}\otimes\operatorname*{id}\nolimits_{X}} 
    \ar[rr]^{\mathrm{A}_{1,\operatorname*{Ob}}^{\left(  5\otimes\right)  }}
      &&  \left[ X,\left[X,\bar{X}, X\right],X\right] 
      \ar[rr]^{\mathrm{A}_{2,\operatorname*{Ob}}^{\left(  5\otimes\right)  }}
      \ar[d]^(.4){\operatorname*{id}\nolimits_{X}\otimes
      \mathrm{Q}_{\left(  2\right)  \operatorname*{Ob}}^{\left(  3\otimes\right)  }\otimes\operatorname*{id}\nolimits_{X}}
      &&  \left[ X,X,\left[ \bar{X} ,X,X\right]\right] 
            \ar[dll]^(.3){\operatorname*{id}\nolimits_{X}\otimes\operatorname*{id}\nolimits_{X}\otimes
            \mathrm{Q}_{\left(  1\right)  \operatorname*{Ob}}^{\left(  3\otimes\right)  }} \\
      && \left[ X,X,X \right] }\label{diag7}
\end{equation} commutes, where $\mathit{D}iag^{\left(  n\otimes
\right)  }:\mathcal{C}\rightarrow\mathcal{C}^{n\otimes}$ is the diagonal.
\end{example}

\begin{definition}
\label{def-grcat}A\textit{ nonunital non-strict groupal category
}$\mathtt{GCat}_{n}$ is $\left(  \mathcal{C},\mathit{M}^{\left(
n\otimes\right)  },\mathrm{A}^{\left(  2n-1\right)  \otimes},\mathit{Q}%
,\mathrm{Q}^{\left(  n\otimes\right)  }\right)  $, i.e. a polyadic non-strict
semigroupal category $\mathtt{SGCat}_{n}$ equipped with the querfunctor
$\mathit{Q}$ and the quertors $\mathrm{Q}^{\left(  n\otimes\right)  }$
satisfying (\ref{diag6}).
\end{definition}

\begin{conjecture}
There exist polyadic nonunital non-strict groupal categories which are
arity-non\-reducible (see \textbf{Definition \ref{def-magcat}}), and so their
$n$-ary tensor product cannot be presented in the form of binary tensor
product iterations.
\end{conjecture}

\section{\textsc{Braided tensor categories}}

The next step in the investigation of binary tensor categories is
consideration of the tensor product \textquotedblleft
commutativity\textquotedblright\ property. The tensor product can be
\textquotedblleft commutative\textquotedblright\ such that for a tensor
category $\mathcal{C}$ there exists the equivalence $X\otimes Y=Y\otimes X$,
$\forall X,Y,\in\operatorname*{Ob}\mathcal{C}$, and such tensor categories are
called \textit{symmetric} \cite{maclane1}. By analogy with associativity, one
can introduce non-strict \textquotedblleft commutativity\textquotedblright,
which leads to the notion of a braided (binary) tensor category and the
corresponding coherence theorems \cite{joy/str1}. Various generalizations of
braiding were considered in \cite{gar/fra,dup/mar7,dup/mar2018c}, and their
higher versions are found, e.g., in \cite{kap/voe94,bat10,webs05}.

\subsection{Braided binary tensor categories}

Let $\left(  \mathcal{C},\mathit{M}^{\left(  2\otimes\right)  },\mathrm{A}%
^{\left(  3\otimes\right)  }\right)  $ be a non-strict semigroupal category
with the bifunctor $\mathit{M}^{\left(  2\otimes\right)  }$ and the associator
$\mathrm{A}^{\left(  3\otimes\right)  }$ (\ref{a1}) satisfying the pentagon
axiom (\ref{diag1}) \cite{yetter,boy2007}.

\begin{definition}
A (\textit{binary})\textit{ braiding }$\mathrm{B}^{\left(  2\otimes\right)
}=\left\{  \mathrm{B}_{\operatorname*{Ob}}^{\left(  2\otimes\right)
},\mathrm{B}_{\operatorname*{Mor}}^{\left(  2\otimes\right)  }\right\}  $ of a
semigroupal category $\mathtt{SGCat}_{2}$ is a natural transformation of the
bifunctor $\mathit{M}^{\left(  2\otimes\right)  }$ (bifunctorial isomorphism)
such that%
\begin{equation}
\mathrm{B}_{\operatorname*{Ob}}^{\left(  2\otimes\right)  }:\mathit{M}%
_{\operatorname*{Ob}}^{\left(  2\otimes\right)  }\left[  X_{1},X_{2}\right]
\overset{\simeq}{\rightarrow}\mathit{M}_{\operatorname*{Ob}}^{\left(
2\otimes\right)  }\left[  X_{2},X_{1}\right]  ,\ \ \ \forall X_{i}%
\in\operatorname*{Ob}\mathcal{C},i=1,2, \label{b1}%
\end{equation}
and the action on morphisms $\mathrm{B}_{\operatorname*{Mor}}^{\left(
2\otimes\right)  }$ may be interpreted as a diagonal, similarly to
(\ref{x-com}).
\end{definition}

\begin{definition}
A non-strict semigroupal category endowed with a binary braiding is called a
(\textit{binary})\textit{ braided semigroupal category} $\mathtt{bSGCat}_{2}$
$\left(  \mathcal{C},\mathit{M}^{\left(  2\otimes\right)  },\mathrm{A}%
^{\left(  3\otimes\right)  },\mathrm{B}^{\left(  2\otimes\right)  }\right)  $.
\end{definition}

The braiding $\mathrm{B}^{\left(  2\otimes\right)  }$ is connected with the
associator $\mathrm{A}^{\left(  3\otimes\right)  }$ by the \textit{hexagon
identity}
\begin{equation}
\xymatrix@R+3mm@C-5mm@L+1mm{
      &&  \left[\left[X_1,X_2 \right] , X_3 \right] 
      \ar[drr]^(.7){\mathrm{A}_{\operatorname*{Ob}1,2,3}^{\left(  3\otimes\right) }}
      \ar[dll]_(.7){\mathrm{B}_{\operatorname*{Ob}1,2}^{\left( 2\otimes\right) }
      \otimes\operatorname*{id}\nolimits_{X_{3}}}
      \\
      \left[\left[X_2,X_1 \right] , X_3 \right] 
      \ar[d]_{\mathrm{A}_{\operatorname*{Ob}2,1,3}^{\left(  3\otimes\right) }}&&&&            
     \left[X_1,\left[X_2 , X_3\right]  \right] 
      \ar[d]^{\mathrm{B}_{\left(  3\right)  \operatorname*{Ob}1,23}^{\left(  2\otimes\right)  }}
      \\\left[X_2,\left[ X_1, X_3  \right] \right] 
           \ar[drr]_(.3){\operatorname*{id}\nolimits_{X_{2}}\otimes
            \mathrm{B}_{\operatorname*{Ob}1,3}^{\left( 2\otimes\right) }}
      &&&&
      \left[\left[X_2,X_3 \right] , X_1 \right] 
            \ar[dll]^(.4){\mathrm{A}_{\operatorname*{Ob}2,3,1}^{\left(  3\otimes\right) }}
      \\
      &&
     \left[ X_2 , \left[ X_3,X_1\right] \right] 
                  } \label{diag9}
\end{equation} 
for objects, and similarly for the inverse associator.

\begin{definition}
A \textit{symmetric braided semigroupal category} $\mathtt{sbSGCat}_{2}$ has
the \textquotedblleft invertible\textquotedblright\ braiding%
\begin{align}
\mathrm{B}_{\operatorname*{Ob}X_{1},X_{2}}^{\left(  2\otimes\right)  }%
\circ\mathrm{B}_{\operatorname*{Ob}X_{2},X_{1}}^{\left(  2\otimes\right)  }
&  =\operatorname*{id}\nolimits_{X_{1}\otimes X_{2}}\ \ \ \ \ \text{or}%
\label{bb1}\\
\mathrm{B}_{\operatorname*{Ob}X_{2},X_{1}}^{\left(  2\otimes\right)  }  &
=\mathrm{B}_{\operatorname*{Ob}X_{1},X_{2}}^{\left(  2\otimes\right)
,-1},\ \ \ \forall X_{i}\in\operatorname*{Ob}\mathcal{C} \label{bb2}%
\end{align}

\end{definition}

A von Neumann regular generalization \cite{neu} (weakening) of (\ref{bb1})
leads to

\begin{definition}
A (\textit{von Neumann}) \textit{regular braided semigroupal category} is
defined by a braiding which satisfies \cite{dup/mar5,dup/mar2018c}%
\begin{equation}
\mathrm{B}_{\operatorname*{Ob}X_{1},X_{2}}^{\left(  2\otimes\right)  }%
\circ\mathrm{B}_{\operatorname*{Ob}X_{1},X_{2}}^{\ast\left(  2\otimes\right)
}\circ\mathrm{B}_{\operatorname*{Ob}X_{1},X_{2}}^{\left(  2\otimes\right)
}=\mathrm{B}_{\operatorname*{Ob}X_{1},X_{2}}^{\left(  2\otimes\right)  },
\end{equation}
where $\mathrm{B}_{\operatorname*{Ob}X_{1},X_{2}}^{\ast\left(  2\otimes
\right)  }$ is a generalized inverse \cite{pen1,nashed} of $\mathrm{B}%
_{\operatorname*{Ob}X_{1},X_{2}}^{\left(  2\otimes\right)  }$, and such that
$\mathrm{B}_{\operatorname*{Ob}X_{1},X_{2}}^{\ast\left(  2\otimes\right)
}\neq\mathrm{B}_{\operatorname*{Ob}X_{1},X_{2}}^{\left(  2\otimes\right)
,-1}$ (cf. (\ref{bb2})).
\end{definition}

\begin{proposition}
\label{prop-b-bin}If the (binary) braided semigroupal category is strict (the
associator becomes the equivalence (\ref{as1})--(\ref{as2}), and we can omit
internal brackets), then the diagram
\begin{equation}
\xymatrix@R+3mm@C+1mm@L+1mm{
      &&  \left[X_1,X_2  , X_3 \right] 
      \ar[drr]^(.7){\operatorname*{id}\nolimits_{X_{1}}\otimes
            \mathrm{B}_{\operatorname*{Ob}2,3}^{\left( 2\otimes\right) }}
      \ar[dll]_(.7){\mathrm{B}_{\operatorname*{Ob}1,2}^{\left( 2\otimes\right) }
      \otimes\operatorname*{id}\nolimits_{X_{3}}}
      \ar[ddll]^(.4){\mathrm{B}_{\operatorname*{Ob}1,23}^{\left( 2\otimes\right) }}
      \\
      \left[X_2,X_1  , X_3 \right] 
      \ar[d]_{\operatorname*{id}\nolimits_{X_{2}}\otimes
            \mathrm{B}_{\operatorname*{Ob}1,3}^{\left( 2\otimes\right) }}&&&&            
     \left[X_1,X_3 , X_2  \right] 
      \ar[d]^{\mathrm{B}_{\operatorname*{Ob}1,3}^{\left( 2\otimes\right) }
      \otimes\operatorname*{id}\nolimits_{X_{2}}}
      \ar[ddll]_(.4){\mathrm{B}_{\operatorname*{Ob}1,32}^{\left( 2\otimes\right) }}
      \\\left[X_2, X_3, X_1   \right] 
           \ar[drr]_(.3){\mathrm{B}_{\operatorname*{Ob}2,3}^{\left( 2\otimes\right) }
      \otimes\operatorname*{id}\nolimits_{X_{1}}}
      &&&&
      \left[X_3,X_1  , X_2 \right] 
            \ar[dll]^(.4){\operatorname*{id}\nolimits_{X_{3}}\otimes
            \mathrm{B}_{\operatorname*{Ob}1,2}^{\left( 2\otimes\right) }}
      \\
      &&
     \left[ X_3 ,  X_2,X_1 \right] 
                  } \label{diag11}
\end{equation} 
commutes
\cite{sta63,stasheff}.
\end{proposition}

\begin{proof}
The triangles commute due to the hexagon identity (\ref{diag9}) and the
internal rectangle commutes, because the binary braiding $\mathrm{B}^{\left(
2\otimes\right)  }$ is a natural transformation (bifunctorial isomorphism).
\end{proof}

Omitting indices (\ref{diag11}) becomes the \textit{Yang-Baxter equation} in
terms of tensor products \cite{dri89} (or the \textit{binary braid group}
relation---for their difference see \cite{str95})%
\begin{equation}
\left(  {{\mathrm{B}_{\operatorname*{Ob}}^{\left(  2\otimes\right)  }%
\otimes\operatorname*{id}}}\right)  \circ\left(  \operatorname*{id}%
\otimes\mathrm{B}_{\operatorname*{Ob}}^{\left(  2\otimes\right)  }\right)
\circ\left(  {{\mathrm{B}_{\operatorname*{Ob}}^{\left(  2\otimes\right)
}\otimes\operatorname*{id}}}\right)  =\left(  \operatorname*{id}%
\otimes\mathrm{B}_{\operatorname*{Ob}}^{\left(  2\otimes\right)  }\right)
\circ\left(  {{\mathrm{B}_{\operatorname*{Ob}}^{\left(  2\otimes\right)
}\otimes\operatorname*{id}}}\right)  \circ\left(  \operatorname*{id}%
\otimes\mathrm{B}_{\operatorname*{Ob}}^{\left(  2\otimes\right)  }\right)  .
\label{yb2}%
\end{equation}

If the braided semigroupal category $\mathtt{bSGCat}_{2}$ contains a unit
object, then we have

\begin{definition}
A (\textit{binary}) \textit{braided monoidal category} $\mathtt{MonCat}_{2}$
$\left(  \mathcal{C},\mathit{M}^{\left(  2\otimes\right)  },\mathrm{A}%
^{\left(  3\otimes\right)  },E,\mathrm{U}^{\left(  2\otimes\right)
},\mathrm{B}^{\left(  2\otimes\right)  }\right)  $ is $\mathtt{bSGCat}_{2}$
together with a unit object $E\in\operatorname*{Ob}\mathcal{C}$ satisfying the
triangle axiom (\ref{diag2}) and a unitor $\mathrm{U}^{\left(  2\otimes
\right)  }$ (\ref{lr1})--(\ref{lr2}) the compatibility condition with the
braiding $\mathrm{B}^{\left(  2\otimes\right)  }$ such that the diagram (for objects)\begin{equation}
\xymatrix@R+5mm@C+10mm{
    \left[ X, E \right] 
    \ar[dr]_{\mathrm{U}_{\left(  1\right)  \operatorname*{Ob}}^{\left(  2\otimes\right)  }} 
    \ar[rr]^{\mathrm{B}_{\operatorname*{Ob}}^{\left(  2\otimes\right)  }}
      &&  \left[ E, X\right] 
            \ar[dl]^{\mathrm{U}_{\left(  2\right)  \operatorname*{Ob}}^{\left(  2\otimes\right)  }} \\
      &  X }
      \label{diag10}
\end{equation} commutes.
\end{definition}

For more details on binary braided monoidal categories, see
\cite{fre/yet,joy/str1} and for review, see, e.g.,
\cite{cha/pre,eti/gel/nik/ost,bul/cae/pen/pys}.

\subsection{Braided polyadic tensor categories}

Higher braidings for binary tensor categories were considered (from an
$n$-category viewpoint) in \cite{man/sch,kap/voe94}. We will discuss them for
polyadic categories, defined above in \textbf{Section \ref{sec-ntensor}}. The
difference will be clearer if a polyadic category is not arity-reduced (see
\textbf{Definition \ref{def-magcat}}) and for non-unital groupal categories
(\textbf{Subsection }\ref{subsec-groupal}).

Let $\left(  \mathcal{C},\mathit{M}^{\left(  n\otimes\right)  },\mathrm{A}%
^{\left(  2n-1\right)  \otimes}\right)  $ be a polyadic non-strict semigroupal
category, where $\mathit{M}^{\left(  n\otimes\right)  }$ is a (not arity
reduced) $n$-ary tensor product ($n$-ary functor) and $\mathrm{A}^{\left(
2n-1\right)  \otimes}$ is an associator, i.e. $n-1$ different $\left(
2n-1\right)  $-ary natural transformations (see \textbf{Definition
\ref{def-semicat}}). Now the braiding becomes an $n$-ary natural
transformation, which leads to any of $n$ permutations from the symmetry
(permutation) group $S_{n}$, rather than one possibility only, as for the
binary braiding (\ref{b1}). Note that in the consideration of higher braidings
\cite{man/sch,kap/voe94} one (\textquotedblleft order
reversing\textquotedblright) element of $S_{n}$ was used $\sigma_{n}^{\left(
rev\right)  }\equiv\left(
\begin{array}
[c]{cccc}%
1 & 2 & \ldots & n\\
n & n-1 & \ldots & 1
\end{array}
\right)  \in S_{n}$. Thus, we arrive at the most general

\begin{definition}
\label{def-n-braid}An $n$-\textit{ary braiding} $\mathrm{B}^{n\otimes
}=\left\{  \mathrm{B}_{\operatorname*{Ob}}^{\left(  n\otimes\right)
},\mathrm{B}_{\operatorname*{Mor}}^{\left(  n\otimes\right)  }\right\}  $ of a
polyadic non-strict semigroupal category is an $n$-ary natural (or
infra-natural) transformation%
\begin{equation}
\mathrm{B}_{\operatorname*{Ob}}^{\left(  n\otimes\right)  }:\mathit{M}%
_{\operatorname*{Ob}}^{\left(  n\otimes\right)  }\left[  \mathfrak{X}\right]
\overset{\simeq}{\rightarrow}\mathit{M}_{\operatorname*{Ob}}^{\left(
n\otimes\right)  }\left[  \sigma_{n}\circ\mathfrak{X}\right]  , \label{bmm1}%
\end{equation}
where $\mathfrak{X}$ is an $X$-polyad (see \textbf{Definition \ref{def-assoc}%
}) of the necessary length (which is $n$ here), and $\sigma_{n}\in S_{n}$ are
permutations that may satisfy some consistency conditions. The action on
morphisms $\mathrm{B}_{\operatorname*{Mor}}^{\left(  n\otimes\right)  }$ may
be found from the corresponding diagonal of the natural transformation square
(cf. (\ref{x-com})).
\end{definition}

The binary non-mixed (standard) braiding (\ref{b1}) has $\sigma_{2}=\sigma
_{2}^{\left(  rev\right)  }=\left(
\begin{array}
[c]{cc}%
1 & 2\\
2 & 1
\end{array}
\right)  \in S_{2}$.

\begin{definition}
A polyadic (non-strict) semigroupal category endowed with the $n$-ary braiding
$\left(  \mathcal{C},\mathit{M}^{\left(  n\otimes\right)  },\mathrm{A}%
^{\left(  2n-1\right)  \otimes},\mathrm{B}^{\left(  n\otimes\right)  }\right)
$ is called a \textit{braided semigroupal polyadic category} $\mathtt{bSGCat}%
_{n}$.
\end{definition}

The $n$-ary braiding $\mathrm{B}^{\left(  n\otimes\right)  }$ is connected
with the associator $\mathrm{A}^{\left(  2n-1\right)  \otimes}$ by a polyadic
analog of the hexagon identity (\ref{diag9}).

\begin{example}
In the case $n=3$, the braided non-strict semigroupal ternary category
$\mathtt{bSGCat}_{3}$ contains two associators $\mathrm{A}_{1}^{\left(
5\otimes\right)  }$ and $\mathrm{A}_{2}^{\left(  5\otimes\right)  }$ (see
\textit{Example }\ref{exam-deca}) satisfying the decagon axiom (\ref{diag4}).
Let us take for the ternary braiding $\mathrm{B}^{\left(  3\otimes\right)  }$
its \textquotedblleft order reversing\textquotedblright\ version%
\begin{equation}
\mathrm{B}_{\operatorname*{Ob}}^{\left(  3\otimes\right)  }:\mathit{M}%
_{\operatorname*{Ob}}^{\left(  2\otimes\right)  }\left[  X_{1},X_{2}%
,X_{3}\right]  \overset{\simeq}{\rightarrow}\mathit{M}_{\operatorname*{Ob}%
}^{\left(  2\otimes\right)  }\left[  X_{3},X_{2},X_{1}\right]  ,\ \ \ \forall
X_{i}\in\operatorname*{Ob}\mathcal{C},\ \ i=1,2,3. \label{br1}%
\end{equation}

Then the ternary analog of the hexagon identity is the \textit{decagon
identity} such that the diagram
\begin{equation}
\xymatrix@R+3mm@C-12mm@L+2mm{
      &&  \left[\left[X_1,X_2,X_3 \right] , X_4,X_5 \right] 
      \ar[drr]^(.7){\mathrm{A}_{1,\operatorname*{Ob}1,2,3,4,5}^{\left(  5\otimes\right) }}
      \ar[dll]_(.7){\mathrm{B}_{\operatorname*{Ob}1,2,3}^{\left( 3\otimes\right) }
      \otimes\operatorname*{id}\nolimits_{X_{4}} \otimes\operatorname*{id}\nolimits_{X_{5}} }
      \\
      \left[\left[ X_3,X_2,X_1 \right],X_4 , X_5 \right] 
      \ar[d]_{\mathrm{A}_{1,\operatorname*{Ob}3,2,1,4,5}^{\left(  5\otimes\right) }}
      &&&&            
     \left[X_1,\left[X_2 , X_3,X_4\right],X_5  \right] 
      \ar[d]^{\mathrm{A}_{2,\operatorname*{Ob}1,2,3,4,5}^{\left(  5\otimes\right) }}
      \\ \left[X_3,\left[X_2 , X_1,X_4\right],X_5  \right]
      \ar[d]_{\operatorname*{id}\nolimits_{X_{3}}\otimes\mathrm{B}_{\operatorname*{Ob}2,1,4}^{\left( 3\otimes\right) }
      \otimes\operatorname*{id}\nolimits_{X_{5}} }
       &&&&
      \left[X_1,X_2 ,\left[ X_3,X_4,X_5\right]  \right]
      \ar[d]^{\mathrm{B}_{\operatorname*{Ob}1,2,345}^{\left( 3\otimes\right) }}
       \\  \left[X_3,\left[X_4 , X_1,X_2\right],X_5  \right]
      \ar[d]_{\mathrm{A}_{2,\operatorname*{Ob}3,4,1,2,5}^{\left(  5\otimes\right) }}
       &&&&
      \left[\left[ X_3,X_4,X_5\right],X_2, X_1  \right]
      \ar[d]^{\mathrm{A}_{1,\operatorname*{Ob}3,4,5,2,1}^{\left(  5\otimes\right) }}
      \\\left[X_3,X_4,\left[  X_1,X_2,X_5\right]  \right]
           \ar[drr]_(.3){\operatorname*{id}\nolimits_{X_{3}} 
           \otimes\operatorname*{id}\nolimits_{X_{4}}\otimes
           \mathrm{B}_{\operatorname*{Ob}1,2,5}^{\left( 3\otimes\right) }}
      &&&&
      \left[X_3,\left[X_4,X_5 ,X_2\right] , X_1 \right] 
            \ar[dll]^(.4){\mathrm{A}_{2,\operatorname*{Ob}{3,4,5,2,1}}^{\left(  5\otimes\right) }}
      \\
      &&
       \left[X_3,X_4,\left[ X_5 ,X_2 , X_1\right] \right] 
                  } \label{diag12}
\end{equation} 
commutes.
\end{example}

\begin{conjecture}
[\textsf{Braided} $n$-\textsf{ary coherence}]If the $n$-ary associator
$\mathrm{A}^{\left(  2n-1\right)  \otimes}$ satisfies such $n$-\textit{ary
coherence conditions} that the isomorphism (\ref{mobn}) takes place, and the
$n$-ary braiding $\mathrm{B}^{\left(  n\otimes\right)  }$ satisfies the
polyadic analog of the hexagon identity, then any diagram containing
$\mathrm{A}^{\left(  2n-1\right)  \otimes}$ and $\mathrm{B}^{\left(
n\otimes\right)  }$ commutes.
\end{conjecture}

\begin{proposition}
If the braided semigroupal ternary category $\mathtt{bSGCat}_{3}$ is strict
(the associators becomes equivalences, and we can omit internal brackets),
then the diagram containing only the ternary braidings $\mathrm{B}^{\left(
3\otimes\right)  }$
\begin{equation}
\xymatrix@R+3mm@C-7mm@L+2mm{
      &&  \left[X_1,X_2,X_3 , X_4,X_5 \right] 
      \ar[drr]^(.7){\operatorname*{id}\nolimits_{X_{1}} 
           \otimes\operatorname*{id}\nolimits_{X_{2}}\otimes
           \mathrm{B}_{\operatorname*{Ob}3,4,5}^{\left( 3\otimes\right) }}
      \ar[dll]_(.7){\mathrm{B}_{\operatorname*{Ob}1,2,3}^{\left( 3\otimes\right) }
      \otimes\operatorname*{id}\nolimits_{X_{4}} \otimes\operatorname*{id}\nolimits_{X_{5}} }
      \ar[dddll]^(.5){\mathrm{B}_{\operatorname*{Ob}1,2,345}^{\left( 3\otimes\right) }}
      \\ \left[X_3,X_2 , X_1,X_4,X_5  \right]
      \ar[d]_{\operatorname*{id}\nolimits_{X_{3}} 
           \otimes\operatorname*{id}\nolimits_{X_{2}}\otimes
           \mathrm{B}_{\operatorname*{Ob}1,4,5}^{\left( 3\otimes\right) } }
       &&&&
      \left[X_1,X_2 ,X_5,X_4,X_3  \right]
      \ar[d]^{\operatorname*{id}\nolimits_{X_{1}}\otimes\mathrm{B}_{\operatorname*{Ob}2,5,4}^{\left( 3\otimes\right) }
      \otimes\operatorname*{id}\nolimits_{X_{3}} }
      \ar[dddll]_(.5){\mathrm{B}_{\operatorname*{Ob}1,2,543}^{\left( 3\otimes\right) }}
       \\  \left[X_3,X_2 , X_5,X_4,X_1  \right] \;\;\quad\;\
      \ar[d]_{\operatorname*{id}\nolimits_{X_{3}}\otimes\mathrm{B}_{\operatorname*{Ob}2,5,4}^{\left( 3\otimes\right) }
      \otimes\operatorname*{id}\nolimits_{X_{1}} }
       &&&&
      \;\;\quad\;\;\left[X_1,X_4,X_5,X_2, X_3  \right]
      \ar[d]^{\mathrm{B}_{\operatorname*{Ob}1,4,5}^{\left( 3\otimes\right) }
      \otimes\operatorname*{id}\nolimits_{X_{2}} \otimes\operatorname*{id}\nolimits_{X_{3}} }
      \\\left[X_3,X_4,  X_5,X_2,X_1  \right]
           \ar[drr]_(.3){\mathrm{B}_{\operatorname*{Ob}3,5,4}^{\left( 3\otimes\right) }
      \otimes\operatorname*{id}\nolimits_{X_{2}} \otimes\operatorname*{id}\nolimits_{X_{1}} }
      &&&&
      \left[X_5,X_4,X_1 ,X_2 , X_3 \right] 
            \ar[dll]^(.4){\operatorname*{id}\nolimits_{X_{5}} 
           \otimes\operatorname*{id}\nolimits_{X_{4}}\otimes
           \mathrm{B}_{\operatorname*{Ob}1,2,3}^{\left( 3\otimes\right) }}
      \\
      &&
       \left[X_5,X_4, X_3 ,X_2 , X_1 \right] 
                  } \label{diag13}
\end{equation} 
commutes (cf. the binary braiding
(\ref{diag11})).
\end{proposition}

\begin{proof}
This is analogous to (\ref{diag11}).
\end{proof}

There follows from (\ref{diag13}), omitting indices, the \textit{ternary braid
group relation} in terms of tensor products (cf. the tetrahedron equation
\cite{baz/str,kap/voe94,bae/neu})%
\begin{align}
&  \left(  {{\mathrm{B}^{\left(  3\otimes\right)  }\otimes\operatorname*{id}%
\otimes\operatorname*{id}}}\right)  \circ\left(  {{{\operatorname*{id}}%
\otimes\mathrm{B}^{\left(  3\otimes\right)  }\otimes\operatorname*{id}}%
}\right)  \circ\left(  {{\operatorname*{id}\otimes\operatorname*{id}%
\otimes\mathrm{B}^{\left(  3\otimes\right)  }}}\right)  \circ\left(
{{\mathrm{B}^{\left(  3\otimes\right)  }\otimes\operatorname*{id}%
\otimes\operatorname*{id}}}\right) \nonumber\\
&  =\left(  {{\operatorname*{id}\otimes\operatorname*{id}\otimes
\mathrm{B}^{\left(  3\otimes\right)  }}}\right)  \circ\left(  {{\mathrm{B}%
^{\left(  3\otimes\right)  }\otimes\operatorname*{id}\otimes\operatorname*{id}%
}}\right)  \circ\left(  {{{\operatorname*{id}}\otimes\mathrm{B}^{\left(
3\otimes\right)  }\otimes\operatorname*{id}}}\right)  \circ\left(
{{\operatorname*{id}\otimes\operatorname*{id}\otimes\mathrm{B}^{\left(
3\otimes\right)  }}}\right)  ,
\end{align}
which was obtained in \cite{dup2018d} using another approach: by the
associative quiver technique from \cite{dup2018a}. For instance, the
$4$\textit{-ary braid group} relation for $4$-ary braiding ${{\mathrm{B}%
^{\left(  4\otimes\right)  }}}$ has the form%
\begin{align}
&  \left(  {{\mathrm{B}^{\left(  4\otimes\right)  }\otimes\operatorname*{id}%
\otimes\operatorname*{id}\otimes\operatorname*{id}}}\right)  \circ\left(
{{\operatorname*{id}{\otimes\mathrm{B}^{\left(  4\otimes\right)  }}%
\otimes{\operatorname*{id}}\otimes\operatorname*{id}}}\right)  \circ\left(
{{\operatorname*{id}\otimes{\operatorname*{id}}\otimes\mathrm{B}^{\left(
4\otimes\right)  }\otimes\operatorname*{id}}}\right)  \circ\left(
{{\operatorname*{id}\otimes\operatorname*{id}\otimes\operatorname*{id}%
\otimes\mathrm{B}^{\left(  4\otimes\right)  }}}\right) \nonumber\\
&  \circ\left(  {{\mathrm{B}^{\left(  4\otimes\right)  }\otimes
\operatorname*{id}\otimes\operatorname*{id}\otimes\operatorname*{id}}}\right)
=\left(  {{\operatorname*{id}\otimes\operatorname*{id}\otimes
\operatorname*{id}\otimes\mathrm{B}^{\left(  4\otimes\right)  }}}\right) \\
&  \circ\left(  {{\mathrm{B}^{\left(  4\otimes\right)  }\otimes
\operatorname*{id}\otimes\operatorname*{id}\otimes\operatorname*{id}}}\right)
\circ\left(  {{\operatorname*{id}\otimes{\operatorname*{id}}\otimes
\mathrm{B}^{\left(  4\otimes\right)  }\otimes\operatorname*{id}}}\right)
\circ\left(  {{\operatorname*{id}{\otimes\mathrm{B}^{\left(  4\otimes\right)
}}\otimes{\operatorname*{id}}\otimes\operatorname*{id}}}\right)  \circ\left(
{{\operatorname*{id}\otimes\operatorname*{id}\otimes\operatorname*{id}%
\otimes\mathrm{B}^{\left(  4\otimes\right)  }}}\right)  .\nonumber
\end{align}

For the non-mixed \textquotedblleft order reversing\textquotedblright\ $n$-ary
braiding (see \textbf{Definition} \ref{def-n-braid}) we have \cite{dup2018d}

\begin{proposition}
The $n$-ary braid equation contains $\left(  n+1\right)  $ multipliers, and
each one acts on $\left(  2n-1\right)  $ tensor products as%
\begin{align}
&  \left(  {{\mathrm{B}^{\left(  n\otimes\right)  }\otimes}}\overset
{n-1}{\overbrace{{{\operatorname*{id}\otimes\ldots\otimes\operatorname*{id}}}%
}}\right)  \circ\left(  {{{\operatorname*{id}\otimes}\mathrm{B}^{\left(
n\otimes\right)  }\otimes}}\overset{n-2}{\overbrace{{{\operatorname*{id}%
\otimes\ldots\otimes\operatorname*{id}}}}}\right)  \circ\left(
{{{\operatorname*{id}\otimes\operatorname*{id}\otimes}\mathrm{B}^{\left(
n\otimes\right)  }\otimes}}\overset{n-3}{\overbrace{{{\operatorname*{id}%
\otimes\ldots\otimes\operatorname*{id}}}}}\right)  \circ\ldots\nonumber\\
&  \circ\left(  \overset{n-2}{\overbrace{{{\operatorname*{id}\otimes
\ldots\otimes\operatorname*{id}}}}}{{{\otimes}\mathrm{B}^{\left(
n\otimes\right)  }\otimes{\operatorname*{id}}}}\right)  \circ\left(
\overset{n-1}{\overbrace{{{\operatorname*{id}\otimes\ldots\otimes
\operatorname*{id}}}}}{{\otimes\mathrm{B}^{\left(  n\otimes\right)  }}%
}\right)  \circ\left(  {{\mathrm{B}^{\left(  n\otimes\right)  }\otimes}%
}\overset{n-1}{\overbrace{{{\operatorname*{id}\otimes\ldots\otimes
\operatorname*{id}}}}}\right) \nonumber\\
&  =\left(  \overset{n-1}{\overbrace{{{\operatorname*{id}\otimes\ldots
\otimes\operatorname*{id}}}}}{{\otimes\mathrm{B}^{\left(  n\otimes\right)  }}%
}\right)  \circ\left(  {{\mathrm{B}^{\left(  n\otimes\right)  }\otimes}%
}\overset{n-1}{\overbrace{{{\operatorname*{id}\otimes\ldots\otimes
\operatorname*{id}}}}}\right)  \circ\left(  {{{\operatorname*{id}\otimes
}\mathrm{B}^{\left(  n\otimes\right)  }\otimes}}\overset{n-2}{\overbrace
{{{\operatorname*{id}\otimes\ldots\otimes\operatorname*{id}}}}}\right)
\circ\ldots\nonumber\\
&  \circ\left(  \overset{n-3}{\overbrace{{{\operatorname*{id}\otimes
\ldots\otimes\operatorname*{id}}}}}{{{\otimes}\mathrm{B}^{\left(
n\otimes\right)  }\otimes{\operatorname*{id}\otimes\operatorname*{id}}}%
}\right)  \circ\left(  \overset{n-2}{\overbrace{{{\operatorname*{id}%
\otimes\ldots\otimes\operatorname*{id}}}}}{{{\otimes}\mathrm{B}^{\left(
n\otimes\right)  }\otimes{\operatorname*{id}}}}\right)  \circ\left(
\overset{n-1}{\overbrace{{{\operatorname*{id}\otimes\ldots\otimes
\operatorname*{id}}}}}{{\otimes\mathrm{B}^{\left(  n\otimes\right)  }}%
}\right)  .
\end{align}

\end{proposition}

\begin{remark}
If a polyadic category is arity-nonreducible, then the higher $n$-ary braid
relations cannot be \textquotedblleft iterated\textquotedblright, i.e.
obtained from the lower $n$ ones.
\end{remark}

Consider a polyadic monoidal category $\mathtt{MonCat}_{n}$ with one unit
object $E$ (see \textbf{Definition} \ref{def-nmon}). Then the $n$-ary braiding
${{\mathrm{B}^{\left(  n\otimes\right)  }}}$ satisfies the triangle identity
connecting it with the unitors $\mathrm{U}^{\left(  n\otimes\right)  }$.

\begin{example}
In the case of the ternary monoidal category $\mathtt{MonCat}_{3}$ (see
\textit{Example} \ref{exam-mon3}) the \textquotedblleft order
reversing\textquotedblright\ braiding ${{\mathrm{B}^{\left(  3\otimes\right)
}}}$ (\ref{br1}) satisfies an additional triangle identity analogous to
(\ref{diag10}) such that the diagram\begin{equation}
\xymatrix@R+5mm@C+10mm{
    \left[ X, E ,E\right] 
    \ar[dr]_{\mathrm{U}_{\left(  1\right)  \operatorname*{Ob}}^{\left(  3\otimes\right)  }} 
    \ar[rr]^{\mathrm{B}_{\operatorname*{Ob}}^{\left(  3\otimes\right)  }}
      &&  \left[ E,E, X\right] 
            \ar[dl]^{\mathrm{U}_{\left(  3\right)  \operatorname*{Ob}}^{\left(  3\otimes\right)  }} \\
      &  X }
      \label{diag14}
\end{equation} commutes.
\end{example}

For the polyadic non-unital groupal category $\mathtt{GCat}_{n}$ (see
\textbf{Definition} \ref{def-grcat}) the $n$-ary braiding ${{\mathrm{B}%
^{\left(  n\otimes\right)  }}}$ should be consistent with the quertors
$\mathrm{U}^{\left(  n\otimes\right)  }$ and the querfunctor $\mathit{Q}$ (see
\textbf{Definition} \ref{def-q}).

\begin{definition}
A \textit{braided polyadic groupal category} $\mathtt{bGCat}_{n}$ is a
polyadic groupal category $\mathtt{GCat}_{n}$ endowed with the $n$-ary
braiding $\left(  \mathcal{C},\mathit{M}^{\left(  n\otimes\right)
},\mathrm{A}^{\left(  2n-1\right)  \otimes},\mathit{Q},\mathrm{Q}^{\left(
n\otimes\right)  },{{\mathrm{B}^{\left(  n\otimes\right)  }}}\right)  $.
\end{definition}

\begin{example}
In the ternary groupal category $\mathtt{GCat}_{3}$ (see \textit{Example}
\ref{exam-q}) the \textquotedblleft order reversing\textquotedblright%
\ braiding ${{\mathrm{B}^{\left(  3\otimes\right)  }}}$ (\ref{br1}) satisfies
the additional identity of consistency with the querfunctor $\mathit{Q}$ and
quertor $\mathrm{Q}^{\left(  3\otimes\right)  }$ such that the diagram\begin{equation}
\xymatrix@R+5mm@C+10mm@L+2mm{
&\left[ X, X ,X\right]
\ar[dl]_{\mathit{Q}_{\operatorname*{Ob}}
      \otimes\operatorname*{id} \otimes\operatorname*{id}}
\ar[dr]^{\operatorname*{id} \otimes\operatorname*{id}\otimes\mathit{Q}_{\operatorname*{Ob}}}
    \\
    \left[ \bar{X}, X ,X\right] 
    \ar[dr]_{\mathrm{Q}_{\left(  1\right)  \operatorname*{Ob}}^{\left(  3\otimes\right)  }} 
    \ar[rr]^{\mathrm{B}_{\operatorname*{Ob}}^{\left(  3\otimes\right)  }}
      &&  \left[ X,X, \bar{X}\right] 
            \ar[dl]^{\mathrm{Q}_{\left(  3\right)  \operatorname*{Ob}}^{\left(  3\otimes\right)  }} \\
      &  X }
      \label{diag15}
\end{equation} commutes.
\end{example}

The above diagrams ensure that some version of coherence can also be proven
for braided polyadic categories.

\section{\textsc{Medialed polyadic tensor categories}}

Here we consider a medial approach to braiding inspired by the first part of
our paper. As opposed to binary braiding which is defined by one unique
permutation (\ref{b1}), the $n$-ary braiding can be defined by the enormous
number of possible allowed permutations (\ref{bmm1}). Therefore, in most cases
only one permutation, that is the \textquotedblleft order
reversing\textquotedblright, is usually (and artificially) used (see, e.g.,
\cite{man/sch}) ignoring other possible cases. On the other side, for $n$-ary
structures it is natural to use the mediality property (\ref{am}) which is
\textsf{unique} in the $n$-ary case and for binary groups reduces to
commutativity. So we introduce a \textit{medialing} instead of braiding for
the tensor product in categories, and (by analogy with braided categories) we
call them \textit{medialed categories}.

Let $\left(  \mathcal{C},\mathit{M}^{\left(  n\otimes\right)  },\mathrm{A}%
^{\left(  2n-1\right)  \otimes}\right)  $ be a polyadic non-strict semigroupal
category $\mathtt{SGCat}_{n}$ (see \textbf{Definition \ref{def-semicat}}).

\begin{definition}
\label{def-med}An $n$\textit{-ary medialing} $\mathrm{M}^{\left(  n^{2}%
\otimes\right)  }$ (or \textquotedblleft medial braiding\textquotedblright) is
a \textit{mediality constraint} which is a natural (or infra-natural)
transformation of two composed $n$-ary tensor product functors $\mathit{M}%
^{\left(  n\otimes\right)  }$ (or functorial $n^{2}$-isomorphism)%
\begin{equation}
\mathrm{M}_{\operatorname*{Ob}}^{\left(  n^{2}\otimes\right)  }:\mathit{M}%
_{\operatorname*{Ob}}^{\left(  n\otimes\right)  }\left[
\begin{array}
[c]{c}%
\mathit{M}_{\operatorname*{Ob}}^{\left(  n\otimes\right)  }\left[
X_{11},X_{12},\ldots,X_{1n}\right]  ,\\
\mathit{M}_{\operatorname*{Ob}}^{\left(  n\otimes\right)  }\left[
X_{21},X_{22},\ldots,X_{2n}\right]  ,\\
\vdots\\
\mathit{M}_{\operatorname*{Ob}}^{\left(  n\otimes\right)  }\left[
X_{n1},X_{n2},\ldots,X_{nn}\right]
\end{array}
\right]  \overset{\simeq}{\rightarrow}\mathit{M}_{\operatorname*{Ob}}^{\left(
n\otimes\right)  }\left[
\begin{array}
[c]{c}%
\mathit{M}_{\operatorname*{Ob}}^{\left(  n\otimes\right)  }\left[
X_{11},X_{21},\ldots,X_{n1}\right]  ,\\
\mathit{M}_{\operatorname*{Ob}}^{\left(  n\otimes\right)  }\left[
X_{12},X_{22},\ldots,X_{n2}\right]  ,\\
\vdots\\
\mathit{M}_{\operatorname*{Ob}}^{\left(  n\otimes\right)  }\left[
X_{1n},X_{2n},\ldots,X_{nn}\right]
\end{array}
\right]  ,\label{mmx1}%
\end{equation}
where the action on morphisms $\mathit{M}_{\operatorname*{Mor}}^{\left(
n\otimes\right)  }$ can be viewed as the corresponding diagonal in the natural
transformation diagram as in (\ref{x-com}).
\end{definition}

\begin{remark}
\label{rem-advmed}The advantage of $n$-ary medialing is its
\textit{uniqueness}, because it does not contain a huge number of possible
permutations $\sigma_{n}\in S_{n}$ as does the $n$-ary braiding (\ref{bmm1}).
\end{remark}

\begin{example}
In the binary case $n=2$ we have (using the standard notation $\mathit{M}%
^{\left(  2\otimes\right)  }\longrightarrow\otimes$)%
\begin{equation}
\mathrm{M}_{\operatorname*{Ob}}^{\left(  4\otimes\right)  }:\left(
X_{1}\otimes X_{2}\right)  \otimes\left(  X_{3}\otimes X_{4}\right)
\overset{\simeq}{\rightarrow}\left(  X_{1}\otimes X_{3}\right)  \otimes\left(
X_{2}\otimes X_{4}\right)  ,\ \ \forall X_{i}\in\operatorname*{Ob}%
\mathcal{C},\ i=1,\ldots,4,\label{mx1}%
\end{equation}
which is called a \textit{binary medialing} by analogy with binary braiding
(\ref{b1}).
\end{example}

In the compact matrix notation (see \textbf{Definition} \ref{def-matr})
instead of (\ref{mmx1}) we have (symbolically)%
\begin{equation}
\mathrm{M}_{\operatorname*{Ob}}^{\left(  n^{2}\otimes\right)  }:\left(
\mathit{M}_{\operatorname*{Ob}}^{\left(  n\otimes\right)  }\right)
^{2}\left[  \mathbf{\hat{X}}_{\left(  n^{2}\right)  }\right]  \overset{\simeq
}{\rightarrow}\left(  \mathit{M}_{\operatorname*{Ob}}^{\left(  n\otimes
\right)  }\right)  ^{2}\left[  \mathbf{\hat{X}}_{\left(  n^{2}\right)  }%
^{T}\right]  ,
\end{equation}
where the matrix polyads of objects is (cf. (\ref{mna}))%
\begin{equation}
\mathbf{\hat{X}}_{\left(  n^{2}\right)  }=\left(  X_{ij}\right)  \in\left(
\operatorname*{Ob}\mathcal{C}\right)  ^{\otimes n^{2}},\ \ \ \ \ X_{ij}%
\in\operatorname*{Ob}\mathcal{C},
\end{equation}
and $\left(  \ \right)  ^{T}$ is matrix transposition.

\begin{definition}
A \textit{medialed polyadic semigroupal category} $\left(  \mathcal{C}%
,\mathit{M}^{\left(  n\otimes\right)  },\mathrm{A}^{\left(  2n-1\right)
\otimes},\mathrm{M}^{\left(  n^{2}\otimes\right)  }\right)  $ $\mathtt{mSGCat}%
_{n}$ is a polyadic non-strict semigroupal category $\mathtt{SGCat}_{n}$ (see
\textbf{Definition} \ref{def-nonstrscat}) endowed with the $n$-ary medialing
$\mathrm{M}^{\left(  n^{2}\otimes\right)  }$ satisfying the $n$\textit{-ary
medial coherence condition} (a medial analog of the hexagon identity
(\ref{diag9})).
\end{definition}

\begin{definition}
A \textit{medialed polyadic monoidal category} $\left(  \mathcal{C}%
,\mathit{M}^{\left(  n\otimes\right)  },\mathrm{A}^{\left(  2n-1\right)
\otimes},E,\mathrm{U}^{\left(  n\otimes\right)  },\mathrm{M}^{\left(
n^{2}\otimes\right)  }\right)  $ $\mathtt{mMonCat}_{n}$ is a medialed polyadic
semigroupal category $\mathtt{mSGCat}_{n}$ with the unit object $E\in
\operatorname*{Ob}\mathcal{C}$ and the unitor $\mathrm{U}^{\left(
n\otimes\right)  }$ satisfying some compatibility condition.
\end{definition}

Let us consider the polyadic nonunital groupal category $\mathtt{GCat}_{n}$
(see \textbf{Definition} \ref{def-grcat}), then the $n$-ary medialing
$\mathrm{M}^{\left(  n^{2}\otimes\right)  }$ should be consistent with the
quertors $\mathrm{U}^{\left(  n\otimes\right)  }$ and the querfunctor
$\mathit{Q}$ (see \textbf{Definition} \ref{def-q} and also the consistency
condition for the ternary braiding (\ref{diag15})) .

\begin{definition}
A \textit{braided polyadic groupal category} $\left(  \mathcal{C}%
,\mathit{M}^{\left(  n\otimes\right)  },\mathrm{A}^{\left(  2n-1\right)
\otimes},\mathit{Q},\mathrm{Q}^{\left(  n\otimes\right)  },\mathrm{M}%
{{^{\left(  n^{2}\otimes\right)  }}}\right)  $ $\mathtt{mGCat}_{n}$ is a
polyadic groupal category $\mathtt{GCat}_{n}$ endowed with the $n$-ary
medialing $\mathrm{M}^{\left(  n^{2}\otimes\right)  }$.
\end{definition}

\subsection{Medialed binary and ternary categories}

Due to the complexity of the relevant polyadic diagrams, it is not possible to
draw them in a general case for arbitrary arity $n$. Therefore, it would be
worthwhile to consider first the binary case, and then some of the diagrams
for the ternary case.

\begin{example}
Let $\left(  \mathcal{C},\mathit{M}^{\left(  2\otimes\right)  },\mathrm{A}%
^{\left(  3\otimes\right)  },\mathrm{M}^{\left(  4\otimes\right)  }\right)  $
be a binary medialed semigroupal category $\mathtt{mSGCat}_{2}$, and the
binary medialing be in (\ref{mx1}). Then the medial analog of the hexagon
identity (\ref{diag9}) is given by the binary medial coherence condition such
that the diagram
\begin{equation}
\xymatrix@R+10mm@C-1mm@L+1mm{
 \left[ \left[ \left[ \left[X_1,X_2\right],X_3 \right], X_4 \right],X_5\right] 
    \ar[d]^{\mathrm{A}_{\operatorname*{Ob}12,3,4}^{\left(  3\otimes\right) }\otimes
      \operatorname*{id}\nolimits_{X_{5}}}
\ar[rrrr]^{\mathrm{A}_{\operatorname*{Ob}1,2,3}^{\left(  3\otimes\right) }\otimes
      \operatorname*{id}\nolimits_{X_{4}}\otimes\operatorname*{id}\nolimits_{X_{5}}}
&&&& \left[ \left[ \left[ X_1,\left[ X_2,X_3\right] \right], X_4 \right],X_5\right] 
   \ar[d]_{\mathrm{A}_{\operatorname*{Ob}123,4,5}^{\left(  3\otimes\right) }}
\\
 \left[ \left[ \left[X_1,X_2\right], \left[X_3, X_4 \right] \right],X_5\right]
          \ar[d]^{\mathrm{M}_{\operatorname*{Ob}1,2,3,4}^{\left(  4\otimes\right) }\otimes
      \operatorname*{id}\nolimits_{X_{5}} }
        &&&& 
        \left[ \left[  X_1,\left[ X_2,X_3\right] \right],\left[ X_4 ,X_5\right]\right]
      \ar[d]_{\mathrm{M}_{\operatorname*{Ob}1,23,4,5}^{\left(  4\otimes\right) }}
             \\
   \left[ \left[ \left[X_1,X_3\right], \left[X_2, X_4 \right] \right],X_5\right]         
      \ar[d]^{\mathrm{A}_{\operatorname*{Ob}13,24,5}^{\left(  3\otimes\right) }}
           &&&&   
      \left[ \left[  X_1,X_4  \right],\left[\left[ X_2,X_3\right] ,X_5\right]\right]
      \ar[d]_{\mathrm{A}_{\operatorname*{Ob}1,4,235}^{\left(  3\otimes\right) }}
      \\
    \left[ \left[ X_1,X_3\right], \left[\left[X_2, X_4 \right] ,X_5\right] \right]
         \ar[d]^{\mathrm{M}_{\operatorname*{Ob}1,3,24,5}^{\left(  4\otimes\right) }}    
        &&&& 
          \left[  X_1,\left[ X_4  ,\left[\left[ X_2,X_3\right] ,X_5\right]\right]\right]
          \ar[d]_{\operatorname*{id}\nolimits_{X_{1}}\otimes\operatorname*{id}\nolimits_{X_{4}}
         \otimes \mathrm{A}_{\operatorname*{Ob}2,3,5}^{\left(  3\otimes\right) }}
          \\
      \left[ \left[ X_1,\left[X_2, X_4 \right]\right], \left[ X_3 ,X_5\right] \right]
             \ar[d]^{\mathrm{A}_{\operatorname*{Ob}1,2,4}^{\left(  3\otimes\right) -1}\otimes
      \operatorname*{id}\nolimits_{X_{3}}\otimes\operatorname*{id}\nolimits_{X_{5}}}
        &&&& 
          \left[X_1, \left[X_4,\left[ X_2 ,\left[   X_3 ,X_5\right] \right]\right]\right]
          \\
       \left[ \left[\left[ X_1,X_2\right], X_4 \right], \left[ X_3 ,X_5\right] \right]
           \ar[d]^{\mathrm{A}_{\operatorname*{Ob}124,3,5}^{\left(  3\otimes\right)-1 }}   
        &&&& 
          \left[ \left[X_1,X_4\right],\left[  X_2,\left[   X_3 ,X_5 \right]\right]\right] 
          \ar[u]^{\mathrm{A}_{\operatorname*{Ob}1,4,235}^{\left(  3\otimes\right) }}
             \\
     \left[ \left[\left[\left[ X_1,X_2\right], X_4 \right],  X_3\right] ,X_5 \right]
          \ar[d]^{\mathrm{A}_{\operatorname*{Ob}12,4,3}^{\left(  3\otimes\right) }\otimes
      \operatorname*{id}\nolimits_{X_{5}}}   
        &&&& 
        \left[ \left[X_1,X_4\right],\left[\left[  X_2 ,  X_3\right] ,X_5 \right]\right]    
          \ar[u]^{\operatorname*{id}\nolimits_{X_{1}}\otimes\operatorname*{id}\nolimits_{X_{4}}
         \otimes \mathrm{A}_{\operatorname*{Ob}2,3,5}^{\left(  3\otimes\right) }}
             \\
    \left[ \left[\left[X_1,X_2\right],\left[  X_4 ,  X_3\right]\right] ,X_5 \right]
          \ar[rrrr]^(0.55){\mathrm{M}_{\operatorname*{Ob}1,2,4,3}^{\left(  4\otimes\right) }\otimes
      \operatorname*{id}\nolimits_{X_{5}} }   
      &&&&
      \left[ \left[\left[X_1,X_4\right],\left[  X_2 ,  X_3\right]\right] ,X_5 \right]
       \ar[u]^{\mathrm{A}_{\operatorname*{Ob}14,23,5}^{\left(  3\otimes\right) }}
       }
       \label{diag16}
\end{equation} commutes.
\end{example}

If a medialed semigroupal category $\mathtt{mSGCat}_{2}$ contains a unit
object and the unitor, then we have

\begin{definition}
A \textit{medialed monoidal category} $\mathtt{mMonCat}_{2}$ $\left(
\mathcal{C},\mathit{M}^{\left(  2\otimes\right)  },\mathrm{A}^{\left(
3\otimes\right)  },E,\mathrm{U}^{\left(  2\otimes\right)  },\mathrm{M}%
^{\left(  4\otimes\right)  }\right)  $ is a (\textit{binary}) medialed
semigroupal category $\mathtt{mSGCat}_{2}$ together with a unit object
$E\in\operatorname*{Ob}\mathcal{C}$ and a unitor $\mathrm{U}^{\left(
2\otimes\right)  }$ (\ref{lr1})--(\ref{lr2}) satisfying the triangle axiom
(\ref{diag2}).
\end{definition}

For $\mathtt{mMonCat}_{2}$ the compatibility condition of the medialing
$\mathrm{M}^{\left(  4\otimes\right)  }$ with $E$ and $\mathrm{U}^{\left(
2\otimes\right)  }$ is given by the commutative diagram\begin{equation}
\xymatrix@R+12mm@C+6mm@L+1mm{
\left[ \left[ X_1, E \right], \left[ X, X_2 \right] \right]
\ar[rr]^{\mathrm{M}_{\operatorname*{Ob}}^{\left(  4\otimes\right)  }}
\ar[d]_{\mathrm{A}_{\operatorname*{Ob} X_1, E,X X_2}^{\left(  3\otimes\right)  }}
&&\left[ \left[ X_1, X \right], \left[ E, X_2 \right] \right]
\ar[d]^{\mathrm{A}_{\operatorname*{Ob} X_1, X, E X_2}^{\left(  3\otimes\right)  }}
\\
\left[  X_1,\left[ E , \left[ X, X_2 \right] \right]\right]
\ar[d]_{\operatorname*{id}\nolimits_{X_1}\otimes\mathrm{A}_{\operatorname*{Ob}}^{\left(  3\otimes\right) -1  }}
&& \left[  X_1,\left[ X , \left[ E, X_2 \right] \right]\right]
\ar[d]^{\operatorname*{id}\nolimits_{X_1}\otimes\mathrm{A}_{\operatorname*{Ob}}^{\left(  3\otimes\right) -1  }}
\\
\left[  X_1,\left[\left[ E ,  X\right] , X_2 \right] \right]
    \ar[dr]_{ \operatorname*{id}\nolimits_{X_1}\otimes\mathrm{U}_{\left(  2\right)  \operatorname*{Ob}}^{\left(  2\otimes\right)  }
    \otimes \operatorname*{id}\nolimits_{X_2}} 
      &&  \left[  X_1,\left[\left[ X ,  E\right] , X_2 \right] \right]
            \ar[dl]^{ \operatorname*{id}\nolimits_{X_1}\otimes\mathrm{U}_{\left(  1\right)  \operatorname*{Ob}}^{\left(  2\otimes\right)  }
    \otimes \operatorname*{id}\nolimits_{X_2}} \\
      & \left[  X_1,\left[ X, X_2 \right] \right]
      }
      \label{diag18}
\end{equation} which
is an analog of the triangle diagram for braiding (\ref{diag10}).

\begin{example}
In the ternary nonunital groupal category $\mathtt{GCat}_{3}$ (see
\textit{Example} \ref{exam-q}) the medialing $\mathrm{M}{{^{\left(
9\otimes\right)  }}}$ satisfies the additional identity of consistency with
the querfunctor $\mathit{Q}$ and quertor $\mathrm{Q}^{\left(  3\otimes\right)
}$ such that the diagram
\begin{equation}
\xymatrix@R+4mm@C-15mm@L+4.5mm{
&& \left[ \left[ X,X,X \right], \left[ X,X,X \right], \left[ X,X,X \right]\right] 
    \ar[dll]_(.7){\operatorname*{id} \otimes\mathit{Q}_{\operatorname*{Ob}}\otimes\mathit{Q}_{\operatorname*{Ob}}
      \otimes\operatorname*{id} \otimes\operatorname*{id}\otimes\mathit{Q}_{\operatorname*{Ob}}
      \otimes\operatorname*{id}\otimes\operatorname*{id}\otimes\operatorname*{id}}
\ar[drr]^(.7){\operatorname*{id} \otimes\operatorname*{id}\otimes\operatorname*{id}
      \otimes\mathit{Q}_{\operatorname*{Ob}} \otimes\operatorname*{id}\otimes\operatorname*{id}
      \otimes\mathit{Q}_{\operatorname*{Ob}}\otimes\mathit{Q}_{\operatorname*{Ob}}\otimes\operatorname*{id}}
\\
 \left[ \left[ X,\bar{X},\bar{X} \right], \left[ X,X,\bar{X}\right], \left[ X,X,X \right]\right] 
      \ar[rrrr]^{\mathrm{M}_{\operatorname*{Ob}}^{\left(  9\otimes\right)  }}
      \ar[d]^(.4){\operatorname*{id}\otimes\operatorname*{id}\otimes\operatorname*{id}
      \otimes\mathrm{Q}_{\left(  3\right)  \operatorname*{Ob}}^{\left(  3\otimes\right)  }
      \otimes\operatorname*{id}\otimes\operatorname*{id}\otimes\operatorname*{id}}
      &&&&  \left[ \left[ X,X,X \right], \left[ \bar{X},X,X \right], \left[ \bar{X},\bar{X},X \right]\right]
      \ar[d]_(.6){\operatorname*{id}\otimes\operatorname*{id}\otimes\operatorname*{id}
      \otimes\mathrm{Q}_{\left(  1\right)  \operatorname*{Ob}}^{\left(  3\otimes\right)  }
      \otimes\operatorname*{id}\otimes\operatorname*{id}\otimes\operatorname*{id}}
             \\
      \left[ \left[ X,\bar{X},\bar{X} \right],X, \left[ X,X,X \right]\right] 
      \ar[d]^{\mathrm{A}_{1,\operatorname*{Ob}X,\bar{X},\bar{X},X,X X X}^{\left(  5\otimes\right)  }}
      &&&&   
      \left[ \left[ X,X,X \right], X  \left[ \bar{X},\bar{X},X \right]\right]
      \ar[d]_{\mathrm{A}_{2,\operatorname*{Ob}X X X,X,\bar{X},\bar{X},X}^{\left(  5\otimes\right)-1  }
      }
      \\
         \left[  X,\left[ \bar{X},\bar{X} ,X\right] , \left[ X,X,X \right]\right]  
         \ar[d]^{\mathrm{A}_{2,\operatorname*{Ob}X,\bar{X},\bar{X},X,X X X}^{\left(  5\otimes\right)  }}    
        &&&& 
          \left[\left[ X,X,X \right], \left[ X, \bar{X},\bar{X}\right] ,X \right]
          \ar[d]_{\mathrm{A}_{1,\operatorname*{Ob}X X X,X,\bar{X},\bar{X},X}^{\left(  5\otimes\right)-1 }}
          \\
          \left[  X, \bar{X},\left[\bar{X} ,X , \left[ X,X,X \right]\right] \right]  
             \ar[d]^{\operatorname*{id}\otimes\operatorname*{id}\otimes
             \mathrm{A}_{2,\operatorname*{Ob}}^{\left(  5\otimes\right)-1  }}
        &&&& 
          \left[\left[\left[  X,X,X \right]  X,\bar{X},\right] ,\bar{X} ,X \right]
          \ar[d]_{\mathrm{A}_{1,\operatorname*{Ob}}^{\left(  5\otimes\right)  }
          \otimes\operatorname*{id}\otimes\operatorname*{id}}
           \\
          \left[  X, \bar{X},\left[\bar{X} ,\left[ X , X,X \right] ,X\right] \right] 
           \ar[d]^{\operatorname*{id}\otimes\operatorname*{id}\otimes
             \mathrm{A}_{1,\operatorname*{Ob}}^{\left(  5\otimes\right)-1  }}   
        &&&& 
          \left[\left[  X,\left[X,X  , X\right],\bar{X}\right] ,\bar{X} ,X \right]
          \ar[d]_{\mathrm{A}_{2,\operatorname*{Ob}}^{\left(  5\otimes\right)  }
          \otimes\operatorname*{id}\otimes\operatorname*{id}}
             \\
          \left[  X, \bar{X},\left[\left[\bar{X} , X , X\right],X  ,X\right] \right]  
          \ar[d]^{\operatorname*{id}\otimes\operatorname*{id}
      \otimes\mathrm{Q}_{\left(  1\right)  \operatorname*{Ob}}^{\left(  3\otimes\right)  }
      \otimes\operatorname*{id}\otimes\operatorname*{id}}   
        &&&& 
          \left[\left[  X,X,\left[X  , X,\bar{X}\right]\right] ,\bar{X} ,X \right]
          \ar[d]_{\operatorname*{id}\otimes\operatorname*{id}
      \otimes\mathrm{Q}_{\left(  3\right)  \operatorname*{Ob}}^{\left(  3\otimes\right)  }
      \otimes\operatorname*{id}\otimes\operatorname*{id}}
             \\
          \left[  X, \bar{X},\left[ X , X  ,X\right] \right]  
          \ar[d]^{\mathrm{A}_{2,\operatorname*{Ob}}^{\left(  5\otimes\right)-1  }}   
        &&&& 
          \left[\left[  X,X, X\right] ,\bar{X} ,X \right]
           \ar[d]_{\mathrm{A}_{1,\operatorname*{Ob}}^{\left(  5\otimes\right)  }}
             \\
          \left[  X, \left[\bar{X}, X , X \right]  ,X \right]     
          \ar[drr]_(.4){\operatorname*{id}
      \otimes\mathrm{Q}_{\left(  1\right)  \operatorname*{Ob}}^{\left(  3\otimes\right)  }\otimes\operatorname*{id}}
        &&&& 
           \left[  X,\left[ X, X ,\bar{X}\right] ,X \right]
          \ar[dll]^(.4){\operatorname*{id}
      \otimes\mathrm{Q}_{\left(  3\right)  \operatorname*{Ob}}^{\left(  3\otimes\right)  }\otimes\operatorname*{id}}
             \\
      && \left[ X,X,X \right] }\label{diag19}
\end{equation} commutes. An analog of the hexagon
identity in $\mathtt{GCat}_{3}$ can be expressed by a diagram which is similar
to (\ref{diag16}).
\end{example}

\newpage

\section{\textsc{Conclusions}}

Commutativity in polyadic algebraic structures is defined non-uniquely, if
consider permutations and their combinations. We proposed a canonical way out:
to substitute the commutativity property by mediality. Following this
\textquotedblleft commutativity-to-mediality\textquotedblright\ ansatz we
first investigated mediality for graded linear $n$-ary algebras and arrived at
the concept of almost mediality, which is an analog of almost commutativity.
We constructed \textquotedblleft deforming\textquotedblright\ medial brackets,
which could be treated as a medial analog of Lie brackets. We then proved
Toyoda's theorem for almost medial $n$-ary algebras. Inspired by the above as
examples, we proposed generalizing tensor and braided categories in a similar
way. We defined polyadic tensor categories with an additional $n$-ary tensor
multiplication for which a polyadic analog of the pentagon axiom was given.
Instead of braiding we introduced $n$-ary \textquotedblleft
medialing\textquotedblright\ which satisfies a medial analog of the hexagon
identity, and constructed the \textquotedblleft medialed\textquotedblright%
\ polyadic version of tensor categories. More details and examples will be
presented in a forthcoming paper.

\bigskip

\noindent \textbf{Acknowledgements}. The author would like to express his deep thankfulness
to Andrew James Bruce, Grigorij Kurinnoj, Mike Hewitt, Richard Kerner, Maurice Kibler,
Dimitry Leites, Yuri Manin, Thomas Nordahl, Valentin Ovsienko, Norbert Poncin,
Vladimir Tkach, Raimund Vogl, Alexander Voronov, and Wend Werner for numerous
fruitful discussions and valuable support.


\end{document}